\documentclass[11pt,a4paper]{siamart220329}

\usepackage[utf8]{inputenc}
\usepackage[english]{babel}

\usepackage{tikz}
\usetikzlibrary{shapes.geometric}
\usepackage{makecell}
\usepackage{graphicx}
\usepackage{fancyhdr}
\usepackage{dsfont}
\usepackage{stmaryrd}
\usepackage{booktabs}

\usepackage{amsfonts}
\usepackage{amssymb}
\usepackage{algorithm}

\usepackage{amsmath}

\usepackage{amsfonts}
\usepackage{amssymb}
\usepackage{mathtools}
\usepackage{enumitem}
\usepackage{multirow}
\usepackage{mathrsfs}
\usepackage{overpic}
\usepackage{pict2e}
\usepackage{comment}

\usepackage{algorithm}
\usepackage[noend]{algpseudocode}
\algblock{ParFor}{EndParFor}
\algnewcommand\algorithmicparfor{\textbf{parallel for}}
\algnewcommand\algorithmicpardo{\textbf{do}}
\algrenewtext{ParFor}[1]{\algorithmicparfor\ #1\ \algorithmicpardo}
\algtext*{EndParFor}

\usepackage{mathtools,booktabs}

\algrenewcommand\algorithmicrequire{\textbf{Input:}}
\algrenewcommand\algorithmicensure{\textbf{Output:}}

\usepackage{varwidth}

\usepackage[title]{appendix}

\usepackage{cite}
\usepackage{color}
\newcommand{\R}[0]{\mathbb{R}}

\title{Avoiding discretization issues for nonlinear eigenvalue problems}
\author{Matthew J. Colbrook\thanks{DAMTP, University of Cambridge, Cambridge, CB3 0WA. (\email{m.colbrook@damtp.cam.ac.uk})}
\and Alex Townsend\thanks{Department of Mathematics, Cornell University, Ithaca, NY 14853. (\email{townsend@cornell.edu})}}
\headers{Discretization effects for nonlinear eigenproblems}{M. Colbrook and A. Townsend}
\begin{document}

\maketitle

\begin{abstract}
The first step when solving an infinite-dimensional eigenvalue problem is often to discretize it. We show that one must be extremely careful when discretizing nonlinear eigenvalue problems. Using examples, we show that discretization can: (1) introduce spurious eigenvalues, (2) entirely miss spectra, and (3) bring in severe ill-conditioning. While there are many eigensolvers for solving matrix nonlinear eigenvalue problems, we propose a solver for general holomorphic infinite-dimensional nonlinear eigenvalue problems that avoids discretization issues, which we prove is stable and converges. Moreover, we provide an algorithm that computes the problem's pseudospectra with explicit error control, allowing verification of computed spectra. The algorithm and numerical examples are publicly available in \texttt{infNEP}, which is a software package written in MATLAB.
\end{abstract}



\section{Introduction}\label{sec:introduction}
Many nonlinear eigenvalue problems (NEPs) are derived from discretizing an infinite-dimensional problem. In fact, 25 out of the 52 NEPs from the NLEVP collection are derived by discretizing a continuous problem such as a differential eigenvalue problem~\cite{betcke2013nlevp}. The analysis typically centers on how to solve the resulting finite-dimensional NEP after discretization. 
However, as we show in this paper, discretizing an infinite-dimensional NEP can introduce serious problems. It can modify, destabilize, or destroy the desired eigenvalues, leading to the computed eigenvalues misrepresenting those of the original continuous problem (see~\cref{tab:woes} for a list of issues that we demonstrate in~\cref{sec:Examples} for six examples).

Given a non-empty open set $\Omega\subset\mathbb{C}$ and a matrix-valued function $F:\Omega\rightarrow \mathbb{C}^{n\times n}$, the matrix NEP consists of finding eigenvalues $\lambda\in\Omega$ and nonzero eigenvectors $v\in\mathbb{C}^n$ so that $F(\lambda)v=0$. There are many applications of NEPs in mechanical vibrations~\cite{lancaster2002lambda}, fluid-solid interactions~\cite{voss2002rational}, photonic crystals~\cite{sakoda2001photonic}, time-delay systems~\cite{jarlebring2008spectrum}, resonances~\cite{bindel2006theory}, and numerous other areas~\cite{steinbach2009boundary,kukelova2008polynomial,mehrmann2011nonlinear}. Many of these matrix NEPs are derived from discretizing differential operators, where nonlinearities arise from eigenvalue-dependent boundary conditions~\cite{botchev2009svd}, material parameters~\cite{engstrom2010complex}, particular basis functions~\cite{betcke2005reviving}, or truncating an infinite domain with transparent boundary conditions~\cite{liao2010nonlinear}. 

In this paper, we propose a solver for infinite-dimensional NEPs, which is a variant of the contour-based algorithm called Beyn's method~\cite{beyn2012integral}. Rather than first discretizing the NEP, our algorithm delays discretization until the last possible moment (see~\cref{sec:BeynsMethod}). By delaying discretization, we avoid modifying, destabilizing, and destroying eigenvalues and provably compute them accurately (see~\cref{sec:stab_analysis}).

We focus on NEPs that involve finding eigenvalues $\lambda\in\mathbb{C}$ and nonzero eigenfunctions\footnote{We call them ``eigenfunctions'' even though they may not live inside a function space.} $u\in\mathcal{H}$ such that
\begin{equation}\label{eq:NonlinearEigenproblem}\setlength\abovedisplayskip{6pt}\setlength\belowdisplayskip{6pt}
T(\lambda) u = 0, \qquad T(\lambda) : \mathcal{D}(T)\mapsto \mathcal{H}.
\end{equation} 
For each fixed $\lambda$, $T(\lambda)$ is a closed linear operator acting on a Hilbert space $\mathcal{H}$. To capture unbounded operators, such as differential operators, we assume that $T(\lambda)$ has a densely-defined domain $\mathcal{D}(T)\subset\mathcal{H}$. For a non-empty open subset $\Omega\subset\mathbb{C}$, we assume that $\Omega\ni\lambda \mapsto T(\lambda)u$ is holomorphic for each fixed $u\in\mathcal{D}(T)$~\cite[p.~375]{kato2013perturbation}. These theoretical assumptions extend the usual assumptions for matrix NEPs~\cite{guttel2017nonlinear} and allow us to develop a contour-based eigensolver for~\cref{eq:NonlinearEigenproblem} (see~\cref{sec:InfDimEigVals}). There are many families of differential NEPs that satisfy our theoretical assumption, such as boundary NEPs for partial differential equations, where the variable coefficients and boundary conditions depend holomorphically on the eigenvalue parameter $\lambda$~\cite{mennicken2003non,solov2006preconditioned}.\footnote{While the assumption that $\lambda \mapsto T(\lambda)u$ is holomorphic for each $\lambda$ uses a fixed domain $\mathcal{D}(T)$, it captures eigenvalue-dependent boundary conditions by combining the differential operator and the boundary operator to a two-component operator defined on a fixed space~\cite{schafke1965s}.}

The spectrum of $T$, denoted by $\Lambda(T)$, is the set of points $\lambda\in\Omega$ such that $T(\lambda)$ is not invertible. The resolvent set $\rho(T)=\Omega\setminus\Lambda(T)$ is relatively open in $\Omega$ and $T(z)^{-1}$ is bounded holomorphic for $z\in\rho(T)$~\cite[p.~367]{kato2013perturbation}. Because we deal with operators acting on infinite-dimensional spaces, $\Lambda(T)$ may contain points that are not eigenvalues satisfying~\cref{eq:NonlinearEigenproblem}. However, if $T(z)$ is a Fredholm operator for each $z\in\Omega$ and $\Lambda(T)\neq\Omega$, then $\Lambda(T)$ consists of isolated points that are eigenvalues with finite algebraic and geometric multiplicities~\cite{grigorieff1973approximation}. In this case, the spectrum is discrete~\cite[Sec.~2]{guttel2017nonlinear} and many of the finite-dimensional NEP theorems have an infinite-dimensional analogue~\cite{jeggle1977discrete}. In particular, when $T(z)$ is a Fredholm operator for each $z\in\Omega$ and $\Lambda(T)\neq\Omega$, there is a statement of Keldysh's theorem for infinite-dimensional NEPs (see~\cref{thm:Keldysh}) that underpins our contour-based NEP solver. Since Fredholm operators remain Fredholm after small perturbations~\cite[Thm.~1]{kato1958perturbation}, our assumptions are not brittle and allow for the design of numerical methods. The set of points $\lambda\in\Omega$ for which $T(\lambda)$ is not Fredholm is known as the essential spectrum. 

\begin{table} 
\caption{Discretization issues encountered in our six examples from the NLEVP collection. Our proposed InfBeyn algorithm avoids these issues using an infinite-dimensional approach (see~\cref{sec:BeynsMethod}).\vspace{-2mm}} 
\centering
\begin{tabular}{cc}
\toprule
Example & Observed discretization woes \\ 
\Xhline{2\arrayrulewidth}
\multirow{ 2}{*}{\texttt{acoustic\_wave\_1d}} & spurious eigenvalues \\ 
 & slow convergence \\
\midrule
\multirow{ 2}{*}{\texttt{acoustic\_wave\_2d}} & spurious eigenvalues \\ 
 & wrong multiplicity \\
\midrule
\multirow{ 3}{*}{\texttt{butterfly}} & spectral pollution \\ 
 & missed spectra\\ & wrong pseudospectra \\
\midrule
\multirow{ 2}{*}{\texttt{damped\_beam}} & slow convergence \\ 
 & resolved eigenfunctions with inaccurate eigenvalues\\
\midrule
\texttt{loaded\_string} & ill-conditioning from discretization\\
\midrule
\multirow{ 3}{*}{\texttt{planar\_waveguide} } & collapse onto ghost essential spectrum \\ 
 & failure for accumulating eigenvalues \\
& spectral pollution \\
\bottomrule
\end{tabular}\vspace{-2mm}
\label{tab:woes}
\end{table} 

While problems caused by discretization appear to be common folklore, there is no systematic study of their effects. We look at several examples for a lucid illustration (see~\cref{sec:Examples}). There are several troubling problems caused by the discretization of NEPs that deserve careful attention: 

\begin{description}[noitemsep,leftmargin=1em]
\item[$\bullet$ Spurious eigenvalues.] Spurious eigenvalues that are unrelated to the infinite-dimensional problem may arise due to discretization. These spurious eigenvalues can remain even as the discretization size increases to infinity, a phenomenon known as spectral pollution. This can occur even when the spectrum is purely discrete. 

\item[$\bullet$ Spectral invisibility.] Spectral invisibility refers to some (or all) of the eigenvalues of the NEP being missed by the discretization, even as the discretization size increases to infinity. Regions of spectra can be ``invisible'' to discretizations.

\item[$\bullet$ Ill-conditioning.] The infinite-dimensional NEP may have well-conditioned eigenvalues, while the discretized problem has ill-conditioned eigenvalues. Hence no, even stable, finite-dimensional solver can overcome this issue post-discretization.

\item[$\bullet$ Super-slow convergence.] In practice, we desire that the eigenvalues of the discretization rapidly converge to those of the infinite-dimensional problem. If one has slow convergence, it can be computationally prohibitive to compute eigenvalues representing the infinite-dimensional ones. Even if eigenfunctions are resolved accurately, the corresponding eigenvalues may be inaccurate.

\item[$\bullet$ Wrong multiplicity.] Eigenvalues of the infinite-dimensional problem may be well-approximated using discretizations but with the wrong multiplicity.

\item[$\bullet$ Accumulating eigenvalues.] The discrete spectrum can accumulate at the essential spectrum. This can be challenging for discretizations to resolve accurately.

\item[$\bullet$ Ghost essential spectra.] Many infinite-dimensional NEPs with discrete spectra arise from an underlying spectral problem that has essential spectra. For example, this is common in domain truncation for resonance computations. The eigenvalues of the discretized NEP may collapse onto the ``ghost'' essential spectrum of the underlying problem.
\end{description} 

To overcome these discretization issues, we introduce an infinite-dimensional analogue of Beyn's method (see~\cref{sec:BeynsMethod}). We call our algorithm InfBeyn (see~\cref{sec:BeynsMethod}), which is based on contour integration and adaptively discretizes only linear equations. It computes eigenvalues inside a region in the complex plane by integrating along the region's boundary. It forms a small generalized matrix eigenvalue problem whose eigenvalues match those of the infinite-dimensional NEP inside the contour. While previous approaches to infinite-dimensional NEPs are predominantly problem-specific, InfBeyn provides a general method that converges for any holomorphic NEP in regions where the spectrum is discrete. Moreover, we use techniques from infinite-dimensional randomized numerical linear algebra to prove the method is stable. Proving stability is a well-known problem for contour-based methods for finite-dimensional problems, but it is manageable in infinite dimensions.

In addition to computing eigenvalues of infinite-dimensional NEPs, we also compute pseudospectral sets to give us a more comprehensive understanding of the stability of a system's spectrum (see~\cref{sec:Pseudospectra}). Discretizing NEPs can also cause issues here (see~\cref{sec:butterfly}) and the pseudospectral sets for a discretization may be misleading, even when spectral pollution and invisibility do not occur (see~\cref{sec:OrrSommerfeld}). We provide the first general algorithm that converges to the pseudospectra of NEPs, even when the spectrum is not discrete. Moreover, the algorithm's output is guaranteed to be inside the true pseudospectral sets, thus directly verifying the computation, and that is how we verify the eigenvalues computed by InfBeyn.

The paper is structured as follows: In~\cref{sec:InfDimEigVals}, we detail infinite-dimensional tools for NEPs, including Keldysh's theorem, our InfBeyn algorithm, and how to compute pseudospectra. In~\cref{sec:stab_analysis}, we analyze the stability of InfBeyn by deriving pseudospectral set inclusions. In~\cref{sec:Examples}, we cover six examples from the NLEVP collection derived from infinite-dimensional NEPs and illustrate discretization woes. We conclude and point to future developments in~\cref{sec:Conclusion}. To accompany this paper, we have developed a publicly available MATLAB package called \texttt{infNEP} available at~\cite{GithubRepo}, which includes all of the examples and figures of this paper.

\section{Computational tools for infinite-dimensional NEPs}\label{sec:InfDimEigVals} 
We first state Keldysh's Theorem (see~\cref{thm:Keldysh}) before describing our infinite-dimensional analogue for Beyn's method (see~\cref{sec:BeynsMethod}) and procedure for computing pseudospectra (see~\cref{sec:Pseudospectra}). Rather than directly discretizing the NEP, we delay discretization until the last possible moment and only discretize linear equations. 

\subsection{Keldysh's theorem}\label{sec:Keldysh}
For contour-based methods for NEPs, Keldysh's theorem is an important expansion~\cite{keldysh1951characteristic,keldysh1971completeness} that is used to show that the NEP can be reduced to a linear generalized matrix eigenvalue problem. It can be stated as follows:
\begin{theorem}
Let $\Omega\subset\mathbb{C}$ be a nonempty domain and $T$ in~\cref{eq:NonlinearEigenproblem} be such that $\Omega\ni\lambda\mapsto T(\lambda)u$ is holomorphic for each fixed $u\in\mathcal{D}(T)$. If $T(\lambda)$ is Fredholm for all $\lambda\in\Omega$ and $m$ is the sum of the algebraic multiplicities of the eigenvalues inside $\Omega$,
\begin{equation}\setlength\abovedisplayskip{6pt}\setlength\belowdisplayskip{6pt}
T(z)^{-1} = V(zI-J)^{-1}W^*+R(z),\quad \forall z\in\rho(T).
\label{eq:Keldysh}
\end{equation}
Here $V,W$ are quasimatrices with $m$ columns of suitably normalized generalized right and left eigenvectors, respectively, $J\in\mathbb{C}^{m\times m}$ is a block diagonal matrix of the Jordan blocks of each eigenvalue inside $\Omega$, and $R(z)$ is a holomorphic remainder.
\label{thm:Keldysh} 
\end{theorem}
The precise form of $V,W$ and $J$ can be found in~\cite[Sec.~2.4]{guttel2017nonlinear}. The expansion shows that $T(z)^{-1}$ can be expressed in terms of $(zI-J)^{-1}$ up to a holomorphic remainder. This allows us to use contour integration involving $T(z)^{-1}$ to form an $m\times m$ generalized eigenvalue problem that shares the same eigenvalues as~\cref{eq:NonlinearEigenproblem} inside $\Omega$.

\subsection{An infinite-dimensional analogue of Beyn's method}\label{sec:BeynsMethod}
Beyn's method is efficient for solving matrix NEPs and is particularly useful when one wants to compute eigenvalues inside a known region~\cite{beyn2012integral}. This contour-based method uses Keldysh's expansion in~\cref{eq:Keldysh} to compute a smaller linear pencil whose spectral properties match those of the original problem inside the region enclosed by the contour. If $T(z)\in\mathbb{C}^{N\times N}$ is a matrix NEP, then it first computes the following two matrices: 
\begin{equation}\setlength\abovedisplayskip{6pt}\setlength\belowdisplayskip{6pt}
A_0 = \frac{1}{2\pi i}\int_{\Gamma} T(z)^{-1} V\, dz, \qquad A_1 = \frac{1}{2\pi i}\int_{\Gamma} zT(z)^{-1} V\, dz, 
\label{eq:discreteBeyn} 
\end{equation}
where $\Gamma$ is a closed rectifiable Jordan curve inside $\Omega$ enclosing $m$ eigenvalues of $T(z)$ (counted via algebraic multiplicity). Here, $V\in\mathbb{C}^{N\times (m+p)}$ is a matrix with $m\ll N$ that is often selected at random with independent standard Gaussian entries, and $p$ is a small oversampling factor (e.g., $p=5$) that we recommend for the robustness of the method.  After computing $A_0$ and $A_1$, Beyn's method solves an $m\times m$ generalized matrix eigenvalue problem related to $A_0$ and $A_1$. 

The usual way to apply Beyn's method is first to discretize the NEP and then use Beyn's method on the discretization. The dominating computational cost of Beyn's method is solving linear systems.  However, to overcome discretization concerns, we prefer an infinite-dimensional analogue of Beyn's method that we now describe. There are three essential ingredients to Beyn's method, which we generalized in turn: 

\begin{itemize}[leftmargin=*,noitemsep]

\item[(i)]\textbf{Randomly generated test functions.}
Beyn's method uses a random matrix $V\in\mathbb{C}^{N\times (m+p)}$ whose columns are standard Gaussian test vectors.  An analogue of standard Gaussian vectors is functions drawn from a Gaussian process (GP)~\cite{boulle2022a}. A function, $g$, drawn from a GP is an infinite-dimensional analogue of a vector drawn from a multivariate Gaussian distribution in the sense that samples from $g$ follow a multivariate Gaussian distribution. We describe the process for the case that $\mathcal{H}$ consists of $L^2(\mathcal{X})$ functions on a domain $\mathcal{X}\subseteq\mathbb{R}^d$ and the process is analogous for other Hilbert spaces. We write $g \sim \mathcal{GP}(0,K)$ for some continuous positive definite kernel $K:\mathcal{X} \times \mathcal{X} \to \R$ if for any $x_1, \ldots, x_n \in \mathcal{X}$,
$(g(x_1), \ldots, g(x_n))$ follows a multivariate Gaussian distribution with mean $(0, \ldots, 0)$ and covariance $K_{ij} = K(x_i, x_j)$ for $1 \leq i,j \leq n$. We typically use the squared exponential covariance kernel given by
\begin{equation}\setlength\abovedisplayskip{6pt}\setlength\belowdisplayskip{6pt}
K_{\mathrm{SE}}(x,y) = \frac{1}{\mu \sqrt{2 \pi}}\exp\left( -(x-y)^2/(2\mu^2) \right),\qquad \mu >0, 
\end{equation}
where $s_\mu = \mu \sqrt{2 \pi}$ is a scaling factor. The length scale parameter is denoted by $\mu$ and determines the correlation between samples of $g$. If $\mu$ is large, the samples $g(x_1), \ldots, g(x_n)$ are highly correlated, and $g$ is close to a constant function. If $\mu$ is small, then samples of $g$ are only weakly correlated, and $g$ is usually a highly oscillatory function.  We use $\mathcal{GP}(0,K_\mathrm{SE})$ to generate random functions in InfBeyn. 

\item[(ii)]\textbf{Contour integration.} 
InfBeyn computes the following two quasimatrices:\footnote{A quasimatrix is a matrix whose columns are functions instead of vectors. A $\mathcal{X}\times m$ quasimatrix has $m$ columns, and each column is a function defined on $\mathcal{X}$.}
\begin{equation}\label{eqn:A0}\setlength\abovedisplayskip{6pt}\setlength\belowdisplayskip{6pt}
A_0 = \frac{1}{2\pi i}\int_{\Gamma} T(z)^{-1} \mathcal{V}\, dz, \qquad A_1 = \frac{1}{2\pi i}\int_{\Gamma} zT(z)^{-1} \mathcal{V}\, dz, 
\end{equation}
where $\mathcal{V}$ is a $\mathcal{X}\times (m+p)$ quasimatrix columns so that each column is a function independently drawn from $\mathcal{GP}(0,K_\mathrm{SE})$. 

\item[(iii)]
\textbf{Solving a generalized matrix eigenvalue problems.} 
In InfBeyn, $A_0$ and $A_1$ are quasimatrices with $m + p$ columns. The related $m\times m$ linear pencil is constructed using the economized singular value decomposition (SVD) of $A_0$~\cite[Sec.~4]{townsend2015continuous}, i.e., 
\begin{equation}\setlength\abovedisplayskip{6pt}\setlength\belowdisplayskip{6pt}
A_0 = \mathcal{U} \Sigma_0 V_0^*,
\label{eq:continuousSVD} 
\end{equation}
where $\mathcal{U}$ is a quasimatrix with $m$ orthonormal columns in $L^2(\mathcal{X})$ and $V_0\in\mathbb{C}^{(m+p)\times m}$ is a matrix with orthonormal columns. We then solve the $m\times m$ generalized eigenvalue problem 
\begin{equation}\setlength\abovedisplayskip{6pt}\setlength\belowdisplayskip{6pt}
\mathcal{U}^*A_1V_0x = \lambda \Sigma_0 x,\qquad x\neq 0. 
\label{eq:FinalEigProblem} 
\end{equation} 
\end{itemize}

\begin{algorithm}[t]
\textbf{Input:} Nonlinear eigenvalue problem $T(z)u = 0$, contour $\Gamma$ enclosing $m$ eigenvalues.
\begin{algorithmic}[1]
\State Draw a $\mathcal{X}\times (m+p)$ quasimatrix whose columns are independently drawn from the Gaussian process  $\mathcal{GP}(0,K_\textrm{SE})$. 
\State Compute quasimatrices $A_0$ and $A_1$ in~\cref{eqn:A0} with a quadrature rule~\cref{eqn:approx_Beyn}. 
\State Compute the $m$-truncated SVD of $A_0$ in~\cref{eq:continuousSVD}. 
\State Form and solve the $m\times m$ generalized eigenvalue problem in~\cref{eq:FinalEigProblem} for eigenvalues $\lambda_j$ and eigenvectors $x_j\in\mathbb{C}^{m}$.
\end{algorithmic} \textbf{Output:} Eigenvalues $\lambda_1,\dots,\lambda_m$ in $\Omega$ and eigenfunctions $u_j=\mathcal{U} \Sigma_0 x_j$.
\caption{InfBeyn: Our infinite-dimensional Beyn's method for NEPs.}\label{alg:Beyn's method}
\end{algorithm}

For practical computation, InfBeyn approximates the contour integral in~\cref{eqn:A0} with a quadrature rule such as a mapped trapezoidal rule. Given a quadrature rule with nodes $z_1,\ldots,z_\ell$ and weights $w_1,\ldots,w_\ell$, one can approximate them by
\begin{equation}\setlength\abovedisplayskip{6pt}\setlength\belowdisplayskip{6pt}
\label{eqn:approx_Beyn}
A_0\approx\tilde{A}_0 = \frac{1}{2\pi i}\sum_{k=1}^\ell w_kT(z_k)^{-1}\mathcal{V}, \qquad A_1\approx\tilde{A}_1 = \frac{1}{2\pi i}\sum_{k=1}^\ell w_kz_kT(z_k)^{-1}\mathcal{V}.
\end{equation}
Since ${\rm rank}(A_j) = m$ for $j = 0,1$, we know that $\sigma_{m+1}(\tilde{A}_j)\approx 0$. By performing an $m$-truncated SVD, we ensure that $\tilde{A}_j$ is of rank $m$ for $j = 0,1$. The approach is summarized~\cref{alg:Beyn's method} for the case of simple eigenvalues. In the more general case, InfBeyn recovers an $m\times m$ linear pencil~\cref{eq:FinalEigProblem} with the same spectral properties as $T$ inside $\Gamma$.

For the majority of examples in this paper, $T(z)^{-1}\mathcal{V}$ involves solving a linear differential equation with $m+p$ right-hand sides. We do this by adaptively discretizing the differential equations and solving a linear system~\cite{olver2013fast}. To refine the accuracy of the final computed eigenvalues while keeping computational costs low, we first compute a rough estimate of the eigenvalues so that we can isolate them inside a small circular contour. Then, we repeat InfBeyn on each eigenvalue or a small cluster of eigenvalues. There are two reasons for this approach. First, our experiments find that this is more computationally efficient than increasing the oversampling parameter $p$ or the number of quadrature nodes $\ell$. Second, the analysis in~\cref{sec:stab_analysis} reveals that this is an important strategy for NEPs since the matrix $VW^*$ from~\cref{eq:Keldysh} may become ill-conditioned or rank degenerate if the contour is too large.

Several techniques exist for estimating the number of eigenvalues $m$. If $\Lambda(T)$ is discrete in the interior of $\Gamma$ and $\Gamma$ does not intersect $\Lambda(T)$, then
\begin{equation}\setlength\abovedisplayskip{6pt}\setlength\belowdisplayskip{6pt}
\label{eq:trace_formula}
m=\mathrm{Trace}\left(\frac{1}{2\pi i}\int_{\Gamma} T'(z)T(z)^{-1}\,dz\right).
\end{equation}
This can be approximated using an infinite-dimensional analogue of Hutchinson's method for trace estimation~\cite{contHutch}. Consequently, we assume that $m$ is known throughout the paper and focus on the algorithmic and theoretical aspects of our infinite-dimensional analogue of Beyn's method.

\subsection{Pseudospectra for nonlinear eigenvalue problems}\label{sec:Pseudospectra}
Pseudospectral sets are a mathematical quantity that provides insight into the stability of linear and nonlinear systems, including eigenvalue problems~\cite{trefethen2005spectra}. Consider the set of bounded holomorphic perturbations of $T$ of norm at most $\epsilon>0$, i.e.,
\begin{equation}\setlength\abovedisplayskip{6pt}\setlength\belowdisplayskip{6pt}
\label{gen_analytic_pert}
\mathcal{A}(\epsilon)=\left\{E:\Omega\rightarrow\mathcal{B}(\mathcal{H}) \text{ holomorphic}\, \Big\vert \, \sup_{z\in \Omega}\|E(z)\|<\epsilon\right\}.
\end{equation}
The $\epsilon$-pseudospectrum of $T$ is the following union of spectra of perturbed operators:
$$\setlength\abovedisplayskip{6pt}\setlength\belowdisplayskip{6pt}
\Lambda_{\epsilon}(T)\coloneqq \bigcup_{E\in \mathcal{A}(\epsilon)}\Lambda(T+E).
$$
One can show that this set remains unchanged if we drop the condition that perturbations are holomorphic. It is also common to consider structured perturbations~\cite{tisseur2001structured,green2006pseudospectra,michiels2006pseudospectra,higham2002more,wagenknecht2008structured}, which can additionally be dealt with using the infinite-dimensional techniques we describe in this section.

For linear matrix eigenvalue problems, if the pseudospectra are small around an eigenvalue, small perturbations do not perturb that eigenvalue very far. However, if the pseudospectra are large around an eigenvalue, then a small perturbation can cause that eigenvalue to move far away from its original position. A similar interpretation exists for NEPs in regions where $\Lambda(T)$ is discrete. That is, for sufficiently small $\epsilon$ (so that the spectrum remains discrete under perturbations), $\Lambda_{\epsilon}(T)$ can be equivalently defined via a backward error, i.e., 
$$\setlength\abovedisplayskip{6pt}\setlength\belowdisplayskip{6pt}
\Lambda_{\epsilon}(T)=\inf\left\{z\in\Omega \, \big\vert \, \eta_T(z)<\epsilon\right\},
$$
where $\eta_T(z)$ is a backward error defined in~ \cite{tisseur2000backward,higham1998structured}. An alternative characterization of $\Lambda_{\epsilon}(T)$ also holds when the spectrum is not discrete. The following is a straightforward generalization of~\cite[Prop.~4.1]{bindel2015localization} and~\cite[Thm.~1]{michiels2006pseudospectra} to infinite dimensions:
\begin{theorem}
\label{thm:equiv_pseudospec}
Let $\epsilon>0$. With perturbations measured as in~\cref{gen_analytic_pert}, we have
$$\setlength\abovedisplayskip{6pt}\setlength\belowdisplayskip{6pt}
\Lambda_{\epsilon}(T)=\left\{z\in\Omega \, \big\vert \, \|T(z)^{-1}\|^{-1}<\epsilon\right\},
$$
where we define $\|T(z)^{-1}\|^{-1}=0$ if $z\in\Lambda(T)$.
\end{theorem}
\begin{proof} 
Suppose that $z\notin\Lambda(T)$ and $\|T(z)^{-1}\|^{-1}<\epsilon$. Then, there exists a vector $u\in\mathcal{H}$ of unit norm with $\|T(z)^{-1}u\|>\epsilon^{-1}$. Let $v=T(z)^{-1}u$ and set $E=-uv^*/\|v\|^2$. Then, $\|E\|=1/\|v\|<\epsilon$ and $[T(z)+E]v=0$ so $z\in \Lambda_{\epsilon}(T)$. Hence, we find that
$$\setlength\abovedisplayskip{6pt}\setlength\belowdisplayskip{6pt}
\left\{z\in\Omega \, \big\vert \, \|T(z)^{-1}\|^{-1}<\epsilon\right\}\subset \Lambda_{\epsilon}(T).
$$
For the reverse set inclusion, suppose for a contradiction that $z\in\Lambda_{\epsilon}(T)$ but that $\|T(z)^{-1}\|^{-1}\geq\epsilon$. Then  $z\in\Lambda(T+E)$, for some $E\in\mathcal{A}(\epsilon)$ and hence
$$\setlength\abovedisplayskip{6pt}\setlength\belowdisplayskip{6pt}
\|T(z)^{-1}E(z)\|\leq \|T(z)^{-1}\|\|E(z)\|\leq \|E(z)\|/\epsilon<1.
$$
Note that $T(z)+E(z)=T(z)(I+T(z)^{-1}E(z)),$ and hence by a Neumann series argument we see that $z\notin\Lambda(T+E)$, which is a contradiction.
\end{proof}

\cref{thm:equiv_pseudospec} leads to a method for computing $\Lambda_{\epsilon}(T)$ that avoids discretization issues. Let $\{\mathcal{P}_n\}$ be a sequence of increasing finite-rank orthogonal projections such that $\lim_{n\rightarrow\infty}\mathcal{P}_n^*\mathcal{P}_nu=u$ for any $u\in\mathcal{H}$. Let $\mathcal{U}_n$ be the range of $\mathcal{P}_n$ and suppose that $\mathcal{U}=\cup_{n\in\mathbb{N}}\mathcal{U}_n$ forms a core of $T(z)$ for any $z\in\Omega$. Then we consider the function
\begin{equation}\setlength\abovedisplayskip{6pt}\setlength\belowdisplayskip{6pt}
\label{def_gamma_n}
\gamma_n(z,T):=\min\left\{\sigma_{\mathrm{inf}}(T(z)\mathcal{P}_n^*),\sigma_{\mathrm{inf}}(T(z)^*\mathcal{P}_n^*)\right\},
\end{equation}
where $\sigma_{\mathrm{inf}}$ denotes the smallest singular value. The following theorem shows how these functions approximate $\|T(z)^{-1}\|^{-1}$ and hence can be used to compute pseudospectra. 

\begin{theorem}
\label{injection_moduli_theorem}
The functions $\gamma_n$ satisfy
$$\setlength\abovedisplayskip{6pt}\setlength\belowdisplayskip{6pt}
\gamma_n(z,T)\geq \|T(z)^{-1}\|^{-1}\quad\text{and}\quad \lim_{n\rightarrow\infty}\gamma_n(z,T)=\|T(z)^{-1}\|^{-1},
$$
where the convergence is monotonic from above and uniform on compact subsets of $\Omega$. Moreover, if $z\in\rho(T)\cup \partial \Lambda(T)$, then $\gamma_n(z,T)=\sigma_{\mathrm{inf}}(T(z)\mathcal{P}_n^*).$
\end{theorem}
\begin{proof}
We first note that $\|T(z)^{-1}\|$ can be characterized by
\begin{equation}\setlength\abovedisplayskip{6pt}\setlength\belowdisplayskip{6pt}
\label{injection_moduli_alt}
\|T(z)^{-1}\|^{-1}=\min\left\{\sigma_{\mathrm{inf}}(T(z)),\sigma_{\mathrm{inf}}(T(z)^*)\right\},
\end{equation}
where for an unbounded operator $S$,
$$\setlength\abovedisplayskip{6pt}\setlength\belowdisplayskip{6pt}
\sigma_{\mathrm{inf}}(S)=\inf\{\|Su\|\, \vert \,u\in\mathcal{D}(S),\|u\|=1\}.
$$
The proof of~\cref{injection_moduli_alt} carries over from the case of linear pencils~\cite[Lem.~6.4]{colbrook2022foundations} without modification. Since $T(z)^{-1}$ is bounded holomorphic on $\rho(T)$, $\|T(z)^{-1}\|^{-1}$ is continuous on $\rho(T)$. We show that $\|T(z)^{-1}\|^{-1}$ is continuous on the whole of $\Omega$. Let $z_n\in\rho(T)$ with $z_n\rightarrow z\in\Lambda(T)$. For any $\epsilon>0$, there exists $u_\epsilon\in\mathcal{D}(T)$ of unit norm such that $\min\{\|T(z)u_\epsilon\|,\|T(z)^*u_\epsilon\|\}\leq\epsilon$. But $\min\{\|T(w)u_\epsilon\|,\|T(w)^*u_\epsilon\|\}$ is continuous in $w$ and hence $\limsup_{n\rightarrow\infty}\|T(z_n)^{-1}\|^{-1}\leq\epsilon$. Since $\epsilon>0$ was arbitrary and $\|T(z)^{-1}\|^{-1}$ is identically zero on $\Lambda(T)$, it follows that $\|T(z)^{-1}\|^{-1}$ is continuous.

From~\cref{injection_moduli_alt} it follows that $\gamma_n(z,T)\geq \|T(z)^{-1}\|^{-1}$. Given $z\in\Omega$ and $\epsilon>0$, let $u\in\mathcal{D}(T)$ of unit norm such that $\|T(z)u\|\leq \sigma_{\mathrm{inf}}(T(z))+\epsilon.$ Since $\mathcal{U}$ forms a core of $T(z)$, we may assume that $u\in\mathcal{U}$ and hence that $u=\mathcal{P}_nu$ for large $n$. It follows that $\limsup_{n\rightarrow\infty}\sigma_{\mathrm{inf}}(T(z)\mathcal{P}_n^*)\leq\sigma_{\mathrm{inf}}(T(z))+\epsilon$. We can argue in exactly the same manner for $T(z)^*$ and since $\epsilon>0$ was arbitrary, we have $\lim_{n\rightarrow\infty}\gamma_n(z,T)=\|T(z)^{-1}\|^{-1}$. Since the projections $\mathcal{P}_n$ are increasing, the functions $\gamma_n(z,T)$ decrease monotonically in $n$. Since $\|T(z)^{-1}\|^{-1}$ is continuous, Dini's theorem implies that the convergence of $\gamma_n(z,T)$ is uniform on compact subsets of $\Omega$.

Finally, if $z\in\rho(T)$, then $\sigma_{\mathrm{inf}}(T(z))=\sigma_{\mathrm{inf}}(T(z)^*)$. The proof of this carries over from the linear pencil case. Continuity of $\sigma_{\mathrm{inf}}(T(z))$ and $\sigma_{\mathrm{inf}}(T(z)^*)$ shows that this equality also holds for $z\in\partial\Lambda(T)$.
\end{proof}

Combining~\cref{thm:equiv_pseudospec,injection_moduli_theorem}, we find that for any integer $n$ we have
\begin{equation}\setlength\abovedisplayskip{6pt}\setlength\belowdisplayskip{6pt}
\label{pseudospectra_verified}
\left\{z\in\Omega \, \big\vert \, \gamma_n(z,T)<\epsilon \right\}\subset \Lambda_{\epsilon}(T).
\end{equation}
Moreover, since the convergence of $\gamma_n(z,T)$ to $\|T(z)^{-1}\|^{-1}$ is locally uniform, these approximations converge to $\Lambda_{\epsilon}(T)$ as $n\rightarrow\infty$ without spectral pollution or spectral invisibility.\footnote{This can be made precise in terms of the so-called Attouch--Wets topology which generalizes the Hausdorff metric to closed (including unbounded) subsets of $\mathbb{C}$~\cite{beer1993topologies}.} There are generally two ways to make this a practical computation:
\begin{itemize}[noitemsep,leftmargin=*]
	\item If we have discretizations of the finite-rank operators $T(z)\mathcal{P}_n^*$ and $T(z)^*\mathcal{P}_n^*$ that have finite lower bandwidths (or are well approximated by such matrices), we take rectangular truncations capturing the full range~\cite{colbrook2019compute}. With respect to the appropriate norms, which can differ in the domain and range space, the smallest singular values of the resulting matrices are the same as those of $T(z)\mathcal{P}_n^*$ and $T(z)^*\mathcal{P}_n^*$. Discretizations with finite lower bandwidths for the differential operators studied in this paper are provided by the ultraspherical spectral method~\cite{olver2013fast}.
	\item If we have discretizations of $\mathcal{P}_nT(z)^*T(z)\mathcal{P}_n^*$ and $\mathcal{P}_nT(z)T(z)^*\mathcal{P}_n^*$, then we can compute their smallest singular values and take square roots to compute $\gamma_n$~\cite{colbrook2022foundations}.
\end{itemize}
The first method should be preferred over the second wherever possible since it avoids the loss of precision owing to the square root. There are situations where the second method seems unavoidable~\cite{colbrook2021rigorous}. It can be shown that it is not always possible to compute $\Lambda_{\epsilon}(T)$ by discretizing with square, finite sections $\mathcal{P}_nT\mathcal{P}_n^*$ of $T$~\cite{colbrook2020foundations}. One must be careful if one wants to compute $\Lambda_{\epsilon}(T)$ by discretizing first.

\subsubsection{Example: Orr--Sommerfeld}\label{sec:OrrSommerfeld}
As an example of the inclusion~\cref{pseudospectra_verified}, we consider the classical Orr--Sommerfeld equation. When analyzing the temporal stability of fluid flows, this equation is a linear eigenvalue problem~\cite{orszag1971accurate,reddy1993pseudospectra}. However, if one considers 
spatial stability analysis, it becomes an NEP~\cite[Chapt. 7]{schmid2002stability}. 

\begin{figure}
\centering
\begin{minipage}[b]{1\textwidth}
\centering
\begin{overpic}[width=0.32\textwidth,trim={0mm 0mm 0mm 0mm},clip]{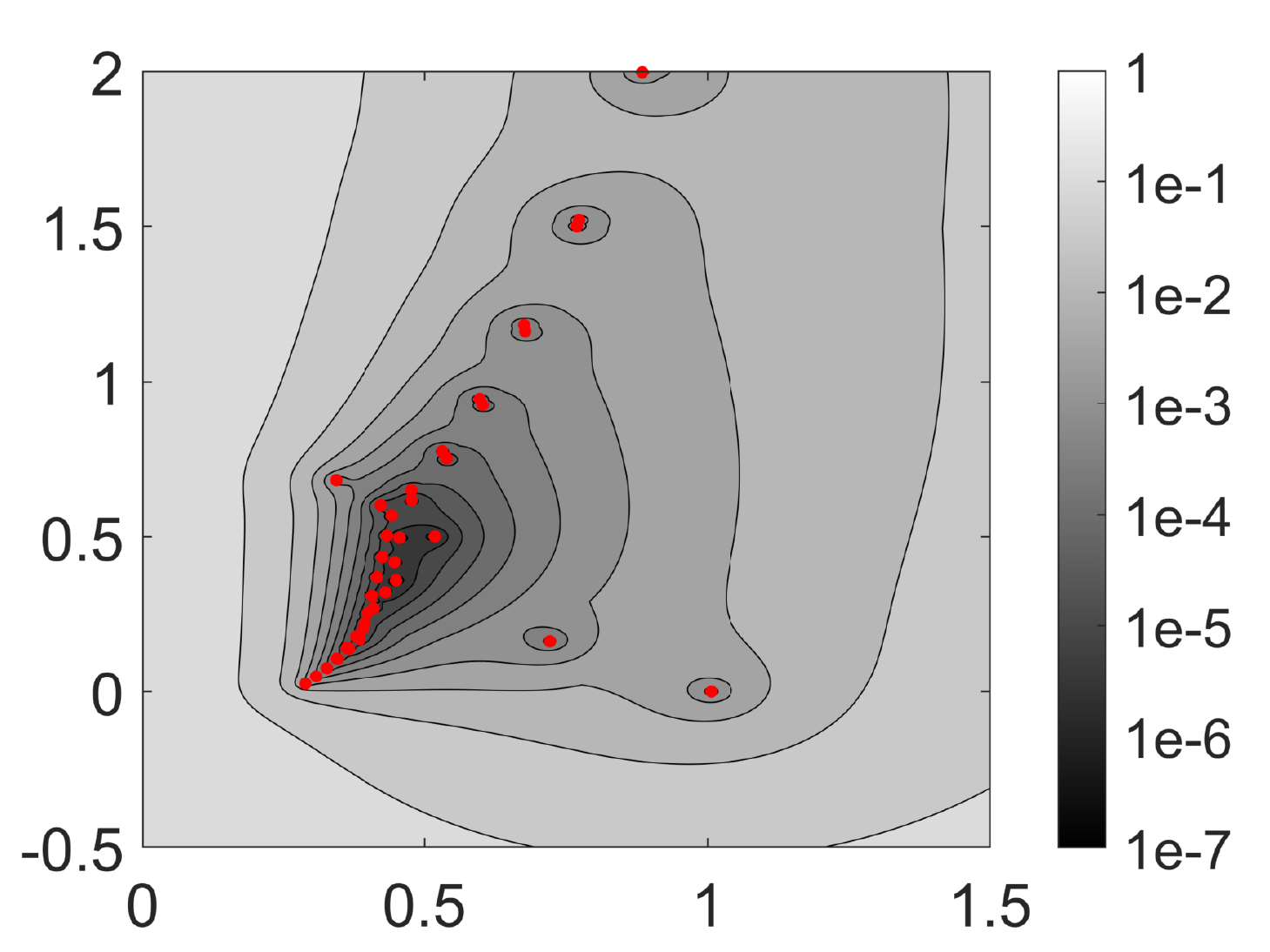}
\put (36,-5) {\small $\mathrm{Re}(\lambda)$}
\put (-4,30) {\small\rotatebox{90}{$\mathrm{Im}(\lambda)$}}
\linethickness{1.3pt}
\put (57,47){\color{blue}\vector(-1,1){10}}
\end{overpic}\hspace{1.5mm}
\begin{overpic}[width=0.32\textwidth,trim={0mm 0mm 0mm 0mm},clip]{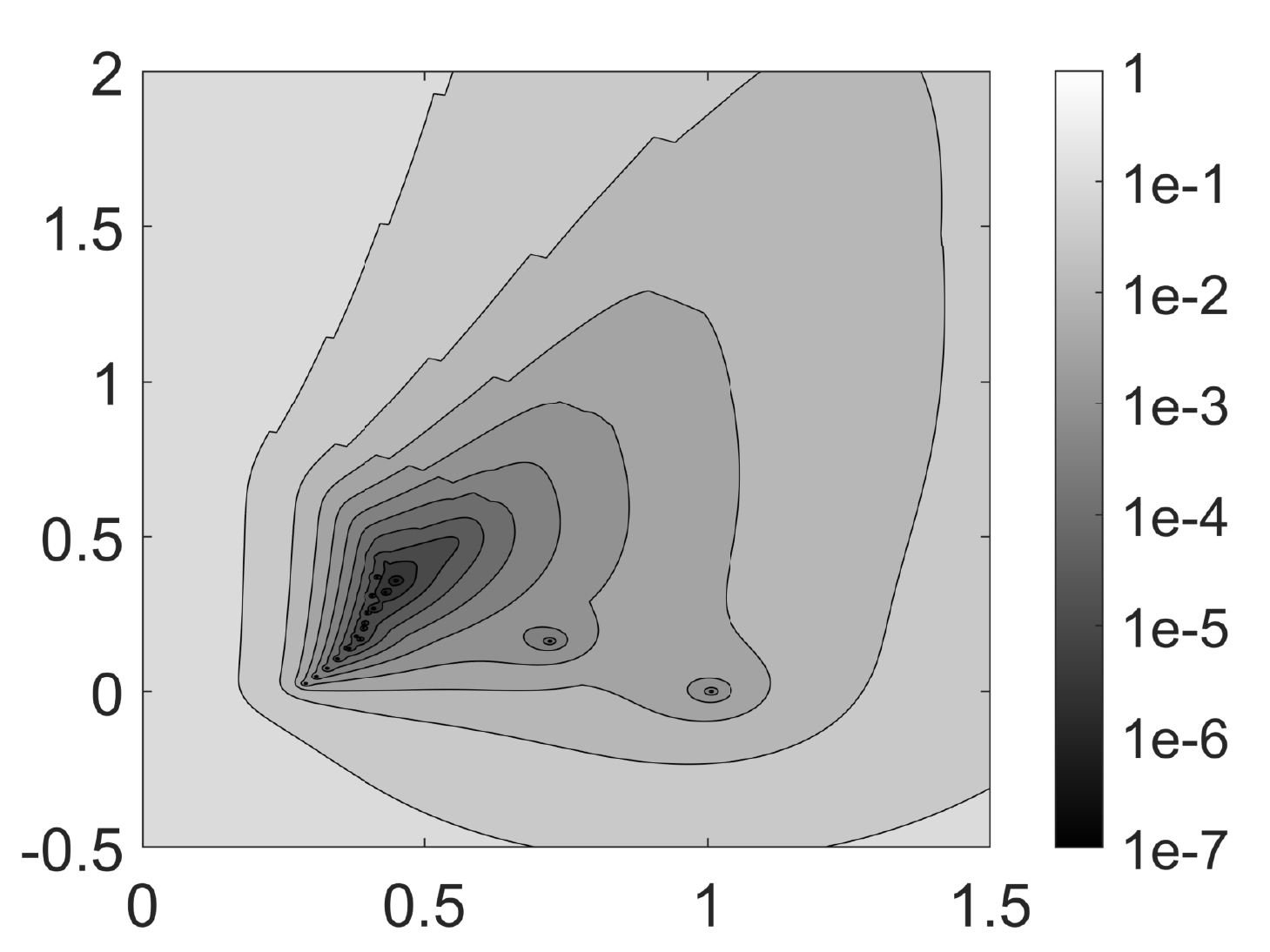}
\put (36,-5) {\small $\mathrm{Re}(\lambda)$}
\put (-4,30) {\small\rotatebox{90}{$\mathrm{Im}(\lambda)$}}
\end{overpic}\hspace{1.5mm}
\begin{overpic}[width=0.32\textwidth,trim={0mm 0mm 0mm 0mm},clip]{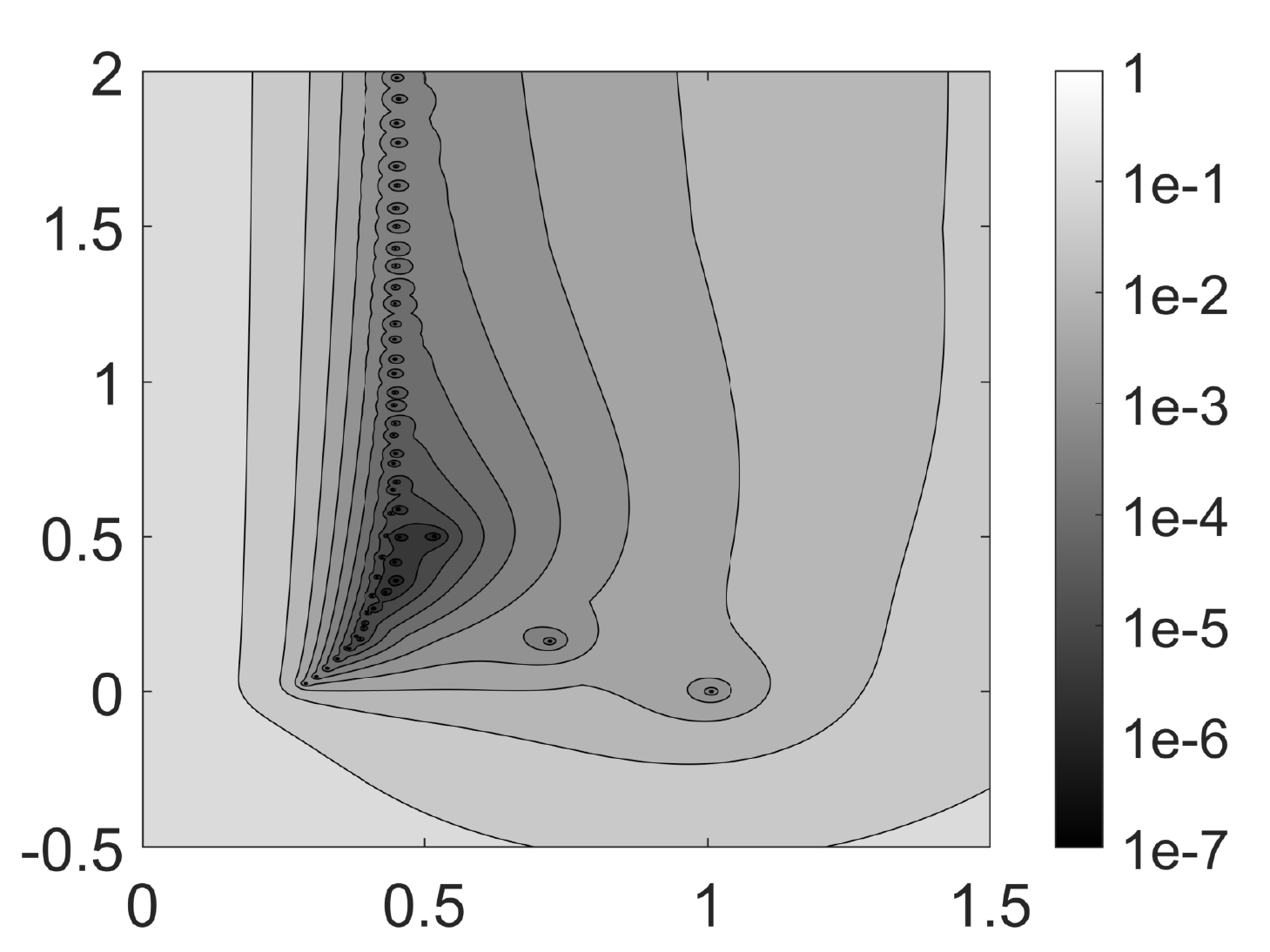}
\put (36,-5) {\small $\mathrm{Re}(\lambda)$}
\put (-4,30) {\small\rotatebox{90}{$\mathrm{Im}(\lambda)$}}
\end{overpic}
\end{minipage}\vspace{-3mm}
\caption{The computed pseudospectra for the Orr--Sommerfeld NEP. Left: The pseudospectra is computed by first discretizing with $n = 64$ and then computing the pseudospectra of the matrix NEP with an appropriate weight matrix. The eigenvalues are shown as red dots and include a spurious branch labeled by a blue arrow. Middle: The computed pseudospectra using the functions $\gamma_n$ from~\cref{def_gamma_n} for $n=64$. These pseudospectral sets are guaranteed to be inside the pseudospectral sets of the infinite-dimensional problem and converge as $n\rightarrow\infty$. Right: The computed pseudospectra using $\gamma_n$ for $n=128$.}
\label{fig:orr_somm1}
\end{figure}
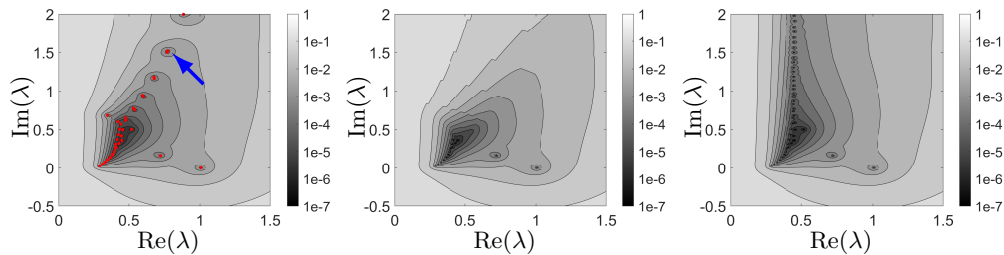

We consider a background plane Poiseuille flow $U(y)=1-y^2$ between two walls at $y = \pm 1$ with Reynolds number $R>0$ and a fixed real perturbation frequency $\omega\in\mathbb{R}$. To define a NEP, we need the following two operators:
\begin{align*}\setlength\abovedisplayskip{6pt}\setlength\belowdisplayskip{6pt}
\mathcal{A}(\lambda)\phi&=\left[\frac{1}{R}\mathcal{B}(\lambda)^2 + i\left(\lambda U(y) -\omega\right)\mathcal{B}(\lambda)+i\lambda U''(y)\right]\phi,\quad
\mathcal{B}(\lambda)\phi&=-\frac{d^2\phi}{dy^2}+\lambda^2\phi.
\end{align*}
The Orr--Sommerfeld operator is formally defined by $T(\lambda)=\mathcal{B}(\lambda)^{-1}\mathcal{A}(\lambda)$. Care is needed when defining the boundary conditions, domains, and appropriate spaces. Moreover, the spectral properties of the NEP depend on a choice of norms~\cite{trefethen1999computation}. We equip $\mathcal{B}(\lambda)$ with Dirichlet boundary conditions $\psi(\pm 1)=0$ and $T(\lambda)$ with boundary conditions $\psi(\pm 1)=0$ and $\psi'(\pm 1)=0$. The appropriate Hilbert space is $\mathcal{D}(\mathcal{B}(1))$ with the energy inner-product given by~\cite{di1969completeness}
\begin{equation}\setlength\abovedisplayskip{6pt}\setlength\belowdisplayskip{6pt}
\label{orr_somm_norm}
\langle \phi,\psi\rangle_E = \int_{-1}^1 [\mathcal{B}(1)\phi]\overline{\psi}\, dy=\int_{-1}^1 \phi\overline{\psi} +\frac{d\phi}{dy}{\frac{d\overline{\psi}}{dy}}\, dy.
\end{equation}
We consider $\omega=0.264002$ and $R=5772.22$, which correspond to the critical neutral point for stability. In this case, the spectrum is discrete.\footnote{Other examples such as Blasius boundary layer flow have a continuous spectral component~\cite{grosch1978continuous}. Pseudospectra can also be computed using the functions $\gamma_n$ for such problems.} We are mainly interested in eigenvalues in the first complex quadrant that have positive phase velocities and are relevant for computing the downstream response to a forcing term at the origin. However, there is a physical relevance to eigenvalues in other quadrants~\cite[Chapt.~7]{schmid2002stability}.

This problem goes under the name of \texttt{orr\_sommerfeld} in the NLEVP collection, where it is discretized using a Chebyshev collocation method~\cite{tisseur2001structured}. We compute the pseudospectra of these discretizations with appropriate weight matrices to take into account the norm induced by~\cref{orr_somm_norm}. For this problem, the pseudospectra of the discretized operators converge to the correct pseudospectra as the discretization size increases. However, deciding which regions of the computed pseudospectra are trustworthy can be challenging. We also compute the pseudospectra using the functions $\gamma_n$ in~\cref{def_gamma_n}, using a Legendre Galerkin spectral method with basis recombination to enforce the boundary conditions. For this basis, $\mathcal{A}\mathcal{P}_n^*$ and $\mathcal{B}\mathcal{P}_n^*$ are lower banded operators. Hence, we use rectangular truncation to compute $\gamma_n(z,T)$ and apply~\cref{injection_moduli_theorem}.\footnote{In this example, the spectrum is discrete and $\partial \Lambda(T)=\Lambda(T)$. Hence~\cref{injection_moduli_theorem} tells us that $\gamma_n(z,T)=\sigma_{\mathrm{inf}}(T(z)\mathcal{P}_n^*)$, which means that $T(z)^*\mathcal{P}_n^*$ is not needed to compute pseudospectra.}

\cref{fig:orr_somm1} (left) shows pseudospectra of the discretized operators using $n=64$ (NLEVP collection default) and \cref{fig:orr_somm1} (middle) shows $\gamma_{n}$ for $n=64$. There is a region where both pseudospectra agree. However, as $\mathrm{Im}(\lambda)$ increases, so do the differences between the pseudospectral sets. Due to~\cref{pseudospectra_verified}, we can trust the output provided by $\gamma_n$ and use it to discern which spectral regions of the Chebyshev collocation method have converged.~\cref{fig:orr_somm1} (right) shows pseudospectra computed using $\gamma_{n}$ for $n=128$. This confirms our suspicions that there is a branch of spurious eigenvalues in the discretized NEP when $n = 64$. In~\cref{sec:butterfly}, we will also see a striking example where the pseudospectra of the discretized operators do not converge.

\section{Stability analysis of infinite-dimensional Beyn's method}\label{sec:stab_analysis}
The goal of this section is to obtain pseudospectral set inclusions for InfBeyn (see~\cref{alg:Beyn's method}) by using Keldysh's theorem (see~\cref{thm:Keldysh}). InfBeyn computes the eigenvalues of the NEP inside a closed contour $\Gamma$ via a computed pencil. The pencil is given by 
\begin{equation} 
\tilde{F}(z) = \tilde{\mathcal{U}}^*\tilde{A}_1\tilde{V}_0 - z \tilde{\mathcal{U}}^*\tilde{A}_0\tilde{V}_0, 
\label{eq:approxF} 
\end{equation} 
where $\tilde{A}_0 = \tilde{\mathcal{U}}\tilde{\Sigma}_0\tilde{V}_0^*$ is the SVD of $\tilde{A}_0$. Here, $\tilde{A}_0$ and $\tilde{A}_1$ are the approximations of $A_0$ and $A_1$, respectively, computed by InfBeyn via a quadrature rule and truncated singular value decomposition (see~\cref{eqn:approx_Beyn}). To understand the pseudospectral set inclusions of InfBeyn, we first relate $\tilde{F}$ to the following pencil: 
\begin{equation} 
F(z) = A_1V_0 - zA_0V_0,
\label{eq:F} 
\end{equation} 
which is about understanding the errors incurred by quadrature rules (see~\cref{sec:FirstPart}). In~\cref{eq:F}, $V_0$ denotes the right singular vector matrix of $A_0$. Note also that the range of $F$ lies in $\mathcal{H}$, whereas its domain is $\mathbb{C}^m$.

Second, we relate $F$  to $T$, which is about controlling the error of InfBeyn when performed with exact integration (see~\cref{sec:SecondPart}). Similar pseudospectral set inclusions are known for the FEAST method for linear eigenvalue problems~\cite{horning2020feast}; however, the analysis is more challenging for NEPs.  

\subsection{Setup}\label{sec:setup}
Suppose that $\Gamma$ is a contour that does not intersect $\Lambda(T)$ and bounds a simply-connected region $\mathrm{int}(\Gamma)$ containing eigenvalues $\lambda_1, \ldots, \lambda_m$ (counted including algebraic multiplicities). If $f$ is a holomorphic function on a neighborhood of $\mathrm{int}(\Gamma)$, then by the Cauchy integral formula, we have
$$
\frac{1}{2\pi i}\int_{\Gamma} f(z)T(z)^{-1}\, dz=Vf(J)W^*,
$$
where $J$ is a block Jordan matrix and $V$ are $W$ are the generalized right and left eigenvectors of $T$ in~\cref{thm:Keldysh}. We assume that the quadrature rule used by InfBeyn is accurate in the sense that our approximation, denoted by $\tilde{A}_j$, to $A_j$ satisfies
\begin{equation} 
\|A_j - \tilde{A}_j\| \leq \epsilon, \qquad j=1,2,
\label{eq:QuadError}
\end{equation} 
where $\epsilon>0$.  Recall that $A_0$, $A_1$, $\tilde{A}_0$ and $\tilde{A}_1$ are of rank $m$ (see~\cref{sec:BeynsMethod}). Throughout the analysis, we also assume that $VW^*\mathcal{V}V_0$ is of rank $m$.

\subsection{Controlling the errors incurred by quadrature rules}\label{sec:FirstPart}
We begin by controlling how the errors in InfBeyn's quadrature rules perturb the spectral properties of its linear pencils, i.e., controlling the difference between $\tilde{F}$  (see~\cref{eq:approxF}) and $F$ (see~\cref{eq:F}). These pencils map to different spaces, so we bound the difference between $\sigma_{\mathrm{inf}}(F)$ and $\sigma_{\mathrm{inf}}(\tilde{F})$, which directly bounds the differences in pseudospectra. Since $\tilde{\mathcal{U}}$ and $\tilde{V}_0$ have orthonormal columns and $\tilde{A}_1$ and $\tilde{A}_0$ are of rank $m$,
$$
\sigma_{\inf}(\tilde{F}(z)) = \sigma_{m}(\tilde{F}(z))=\sigma_{m}\left(\tilde{\mathcal{U}}(\tilde{\mathcal{U}}^*\tilde{A}_1\tilde{V_0}-z\tilde{\mathcal{U}}^*\tilde{A}_0\tilde{V_0})\tilde{V}_0^*\right) = \sigma_{m}(\tilde{\mathcal{U}}\tilde{\mathcal{U}}^*\tilde{A}_1\tilde{V_0}\tilde{V}_0^*-z\tilde{A}_0),
$$
where the last equality follows as $\tilde{\mathcal{U}}\tilde{\mathcal{U}}^*$ and $\tilde{V_0}\tilde{V_0}^*$ act as the identity on the column and row space of $\tilde{A}_0$, respectively. By a similar argument, we also have $\sigma_{\inf}(F(z))=\sigma_m(\mathcal{U}\mathcal{U}^*A_1V_0V_0^*-zA_0)$. It follows that
$$
\left|\sigma_{\inf}(\tilde{F}(z))-\sigma_{\inf}(F(z))\right|\leq \|P_1A_1P_2-\tilde{P}_1\tilde{A}_1\tilde{P}_2\|+|z|\|A_0-\tilde{A}_0\|, 
$$
where $P_1 = \mathcal{U}\mathcal{U}^*$, $\tilde{P}_1 = \tilde{\mathcal{U}}\tilde{\mathcal{U}}^*$, $P_2=V_0V_0^*$, and $\tilde{P}_2=\tilde{V}_0\tilde{V}_0^*$.  By the triangle inequality, we find that 
\[
\|P_1A_1P_2-\tilde{P}_1\tilde{A}_1\tilde{P}_2\|
\leq \|A_1\|(\|P_1-\tilde{P}_1\|+\|P_2-\tilde{P}_2\|)+\|A_1-\tilde{A}_1\|.
\]
Since $A_0$ and $\tilde{A}_0$ have the same rank, we know that~\cite{chen2016perturbation}:
$$
\|P_1-\tilde{P}_1\|+\|P_2-\tilde{P}_2\|\leq 2\min\{\|A_0^\dagger\|,\|\tilde{A}_0^\dagger\|\}\|A_0-\tilde{A}_0\|\leq 2\|A_0^\dagger\|\|A_0-\tilde{A}_0\|,
$$
where $A_0^\dagger$ denotes the pseudoinverse of $A_0$. From~\cref{eq:QuadError}, we conclude that
\begin{equation} 
\left|\sigma_{\inf}(\tilde{F}(z))-\sigma_{\inf}(F(z))\right|\leq (2\|A_0^\dagger\|\|A_1\|+|z|+1)\epsilon.
\label{eq:FinalQuadError} 
\end{equation}
One can interpret~\cref{eq:FinalQuadError} as telling us that the pseudospectral sets of $F(z)$ and $\tilde{F}(z)$ are close. Precisely how close is determined by the errors incurred when computing $A_0$ and $A_1$ with a quadrature rule, i.e., $\|A_0-\tilde{A}_0\|$ and $\|A_1-\tilde{A}_1\|$, as well as the quantity $\|A_0^\dagger\|\|A_1\|$ that is related to the intrinsic spectral properties of $T$ (see~\cref{eq:Regularity}). 

\subsection{Stability analysis with exact integration}\label{sec:SecondPart}
Since~\cref{eq:FinalQuadError} tells us that the pseudospectral sets of $\tilde{F}$ and $F$ are relatively close, we now turn to relating the pseudospectral sets of $F$ to $T$. In the end, we aim to prove pseudospectral set inclusions between $\tilde{F}$ (the pencil used to compute eigenvalues in InfBeyn) and $T$ (the original NEP) inside $\Gamma$.  We have the following pseudospectral set inclusion: 
\begin{theorem} 
Given the same assumptions on $T$ as in~\cref{thm:Keldysh} and the setup in~\cref{sec:setup}. For sufficiently small $\delta>0$, we have the following pseudospectral set inclusions inside $\Gamma$: 
\begin{equation}
\Lambda_{\delta_1}(T)\cap \mathrm{int}(\Gamma) \subset \Lambda_{\delta}(F) \cap \mathrm{int}(\Gamma)\subset\Lambda_{\delta_2}(T) \cap \mathrm{int}(\Gamma),
\label{eq:PseudospectralInclusions} 
\end{equation} 
where $V$ and $W$ are the generalized eigenvectors in~\cref{thm:Keldysh}, $\mathcal{V}$ is the random quasimatrix in~\cref{sec:BeynsMethod}, 
\[
\delta_1 = \frac{\delta}{\|VW^*\|\|VW^*\mathcal{V}\| + M\delta},\qquad \delta_2 =  \frac{\delta}{\sigma_m(VW^*)\sigma_m(VW^*\mathcal{V})- M\delta},
\]
and $M=\sup_{z\in\mathrm{int}(\Gamma)}\|R(z)\|$. 
\label{thm:ExactIntegration} 
\end{theorem} 
\begin{proof}
We first prove the right side of the inclusion. Let $z\in\mathrm{int}(\Gamma)$ such that $\sigma_{\inf}(F(z))<\delta$ and define $L_1=(VW^*)^\dagger$. If $z\in\Lambda(T)$, then~\cref{eq:PseudospectralInclusions} immediately holds, and there is nothing to prove. If $T(z)^{-1}$ exists, then we have
\begin{equation}
\label{a}
T(z)^{-1}L_1F(z)=-VW^*\mathcal{V}V_0+R(z)L_1F(z)
\end{equation}
where we used the fact that $F(z)=V(J-zI)W^*\mathcal{V}V_0$ from~\cref{thm:Keldysh}.  Since $\sigma_{\inf}(F(z))<\delta$, there exists a unit-norm $x\in\mathbb{C}^m$ with $\|F(z)x\|<\delta$. Furthermore, $u=L_1F(z)x\neq 0$; otherwise, $VW^*\mathcal{V}V_0x=0$ and $VW^*\mathcal{V}V_0$ would not be of rank $m$. We also have that
$$
\|u\|\leq \|L_1\| \|F(z)x\|<\delta/\sigma_m(VW^*).
$$
This means that
$$
\left\|T(z)^{-1}\frac{u}{\|u\|}\right\|\geq \frac{\|VW^*\mathcal{V}V_0 x\|}{\|u\|}-M>
\frac{\sigma_m(VW^*)\sigma_m(VW^*\mathcal{V})}{\delta}-M,
$$
where $M=\sup_{z\in\mathrm{int}(\Gamma)}\|R(z)\|$. It follows that if $\delta$ is sufficiently small so that $M\delta<\sigma_m(VW^*)\sigma_m(VW^*\mathcal{V})$, then $z\in \Lambda_{\delta_2}(T)$. Hence, $\Lambda_{\delta}(F)\cap\mathrm{int}(\Gamma) \subset \Lambda_{\delta_2}(T)\cap\mathrm{int}(\Gamma)$.

For the left side of the inclusion, let $z\in\mathrm{int}(\Gamma)$ with $\|T(z)^{-1}\|>1/\delta_1$, where $\delta_1 = \delta/(\|VW^*\|\|VW^*\mathcal{V}\| + M\delta)$, and define $L_2=(VW^*\mathcal{V}V_0)^\dagger$. If $\sigma_{\inf}(F(z))=0$, then the inclusion is immediate. Hence we assume that $\sigma_{\inf}(F(z))>0$ which implies that $T(z)^{-1}$ exists and 
$$
F(z)L_2[T(z)^{-1}-R(z)]=V(J-zI)W^*\mathcal{V}V_0(W^*\mathcal{V}V_0)^\dagger V^\dagger V(z-J)^{-1}W^*=-VW^*. 
$$
There exists an element $u$ of unit norm such that $\|T(z)^{-1}u\|>1/\delta_1$. Provided that $\delta_1$ is sufficiently small so that $\delta_1<1/M$, we know that 
$$
\|T(z)^{-1}u-R(z)u\|> 1/\delta_1-M>0. 
$$
Since $\left(T(z)^{-1}-R(z)\right)u=V(z-J)^{-1}W^*u$ lies in the range of $V$, it follows that $x=L_2\left(T(z)^{-1}-R(z)\right)u$ satisfies the following inequality: 
$$
\|x\|\geq \sigma_m(L_2)\|\left(T(z)^{-1}-R(z)\right)u\|>\frac{1/\delta_1-M}{\|VW^*\mathcal{V}V_0\|}=\frac{1/\delta_1-M}{\|VW^*\mathcal{V}\|}.
$$
This means we have 
$$
\left\|F(z)\frac{x}{\|x\|}\right\|=\frac{\|VW^*u\|}{\|x\|}<\frac{\|VW^*\mathcal{V}\|}{1/\delta_1-M}\|VW^*u\|\leq\frac{\|VW^*\mathcal{V}\|\|VW^*\|}{1/\delta_1-M} = \delta. 
$$
Hence, we conclude that $\Lambda_{\delta_1}(T)\cap\mathrm{int}(\Gamma)\subset \Lambda_{\delta}(F)\cap\mathrm{int}(\Gamma)$, finishing the proof. 
\end{proof} 
\Cref{thm:ExactIntegration} tells us that when InfBeyn is performed with exact integration, the constructed pencil is very reasonable for computing the eigenvalues of $T$ provided that $VW^*$ and $VW^*\mathcal{V}$ are well-conditioned, and $M$ is not too large.   

\subsection{Pseudospectral set inclusions and interpretation} 
We can now combine~\cref{eq:FinalQuadError} and~\cref{thm:ExactIntegration} to conclude the following pseudospectral set inclusions: 
 \begin{equation}
\Lambda_{\delta_{-}}(T)\cap \mathrm{int}(\Gamma) \subset \Lambda_{\delta}(\tilde{F}) \cap \mathrm{int}(\Gamma)\subset\Lambda_{\delta_+}(T) \cap \mathrm{int}(\Gamma),
\label{eq:FinalPseudospectralInclusions} 
\end{equation}
where $\gamma_\pm = \delta \pm (2\|A_0^\dagger\|\|A_1\|+1+\sup_{z\in\mathrm{int}(\Gamma)} |z|)\epsilon$ and 
\[
\delta_- = \frac{\gamma_{-}}{\|VW^*\|\|VW^*\mathcal{V}\| + M\gamma_{-}}, \quad \delta_+ = \frac{\gamma_{+}}{\sigma_{m}(VW^*)\sigma_m(VW^*\mathcal{V}) - M\gamma_{+}}.
\]
Since InfBeyn uses the pencil $\tilde{F}$ to compute the eigenvalues of $T$ inside $\Omega$,~\cref{eq:FinalPseudospectralInclusions} tells us that InfBeyn robustly computes the eigenvalues of $T$ provided that the following conditions hold: (1) $VW^*$ is well-conditioned, (2) $VW^*\mathcal{V}$ is well-conditioned,\footnote{It can be shown that for $p =5$ we have $\sigma_m(VW^*\mathcal{V})\leq 50\sigma_m(VW^*){\rm Trace}((W^*KW)^{-1})$ with probability $>99.999$\% (see~\cite[Lem.~3]{boulle2022learning} with $p = 5$ and $t = 10$). Here, $K$ is the covariance kernel in $\mathcal{GP}(0,K)$ used to randomly generate the columns of $\mathcal{V}$. Moreover, for $p = 5$, $\|VW^*\mathcal{V}\|\leq 9(m+5)\|VW^*\|{\rm Trace}(K)$ with probability $>99.999$\% (see~\cite[Lem.~4]{boulle2022learning} with $p=5$ and $s = 3$), where ${\rm Trace}(K)$ is the sum of the eigenvalues of $K$. In practice, $VW^*\mathcal{V}$ is well-conditioned when $p \geq 5$, $VW^*$ is well-conditioned, and the covariance kernel in $\mathcal{GP}(0,K)$ is reasonably selected.} (3) the holomorphic remainder is not too large inside $\mathrm{int}(\Gamma)$, i.e., $M$ is not too big, (4) $\|A_0^\dagger\|\|A_1\|$ is relatively small, and (5) $\epsilon$ is small. 

Condition (1) is directly determined by the NEP itself and is a kind of measure of the non-normality of $T$ inside $\Omega$. Once condition (1) holds, condition (2) follows in practice, provided that the sketching performed by InfBeyn is adequate at capturing the range of $\mathcal{A}_0$ and $\mathcal{A}_1$. Condition (3) measures the regularity of $T$ inside $\Omega$, while (4) is about the regularity of the pencil $F$ if no quadrature error is incurred. By Keldysh's expansion, we see that 
\begin{equation} 
\|A_0^\dagger\|\|A_1\|\leq\sigma_m(VW^*\mathcal{V})\|VJW^*\mathcal{V}\|\leq \sigma_m(VW^*\mathcal{V})\|VJW^*\|\|\mathcal{V}\|.
\label{eq:Regularity} 
\end{equation} 
Again, one expects that (4) holds, provided that the sketching performed by InfBeyn is adequate. Finally, (5) suggests that InfBeyn's quadrature rules should be relatively accurate.  In short,~\cref{eq:FinalPseudospectralInclusions} tells us that InfBeyn should be regarded as a robust method for computing the eigenvalues of a NEP inside a compact region of the complex plane.

\section{Six NEPs with unsettling discretization issues}\label{sec:Examples} 
It turns out that 25 of the 52 matrix NEPs from the NLEVP collection are derived by discretizing an infinite-dimensional NEP. To showcase the unsettling discretization effects, we take six examples from the NLEVP collection and show how discretization has modified, destabilized, and destroyed spectra. To ensure we report problems caused by discretization alone, we have verified the computed eigenvalues of the discretized NEPs with extended precision.

\subsection{Example 1: One-dimensional acoustic wave}\label{sec:AcousticWave1d}
This is a differential boundary eigenvalue problem posed on $L^2([0,1])$ that takes the form 
\begin{equation}\setlength\abovedisplayskip{6pt}\setlength\belowdisplayskip{6pt}
\frac{d^2p}{dx^2} + 4\pi^2\lambda^2p = 0, \qquad p(0) = 0, \quad \chi p'(1) + {2\pi i \lambda}p(1) =0. 
\label{eq:AcousticWave1D} 
\end{equation} 
Here, $p$ is the acoustic pressure, $\lambda$ is the frequency, and $\chi$ is the (possibly complex) impedance~\cite{chaitin2006analysis}. The eigenvalues correspond to the resonant frequencies of the system and can be calculated explicitly (for values of $\chi$ for which $\tan^{-1}(i\chi)$ is finite) as:
\begin{equation}\setlength\abovedisplayskip{6pt}\setlength\belowdisplayskip{6pt}
\lambda_k = {\tan^{-1}(i\chi)}/({2\pi})+{k}/{2}, \qquad k\in\mathbb{Z}.
\label{eq:lambda_formula}
\end{equation} 
This problem goes under the name of \texttt{acoustic\_wave\_1d} and is the first problem listed in the NLEVP collection. It is commonly discretized using finite element method (FEMs)~\cite{harari1996recent} to form a quadratic matrix NEP.

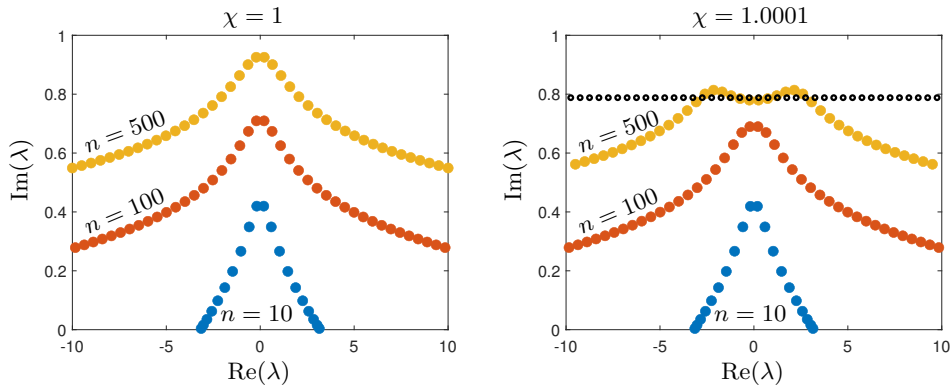
\begin{figure}
\centering
\begin{minipage}[b]{1\textwidth}
\centering
\begin{overpic}[width=0.49\textwidth,trim={0mm 0mm 0mm 0mm},clip]{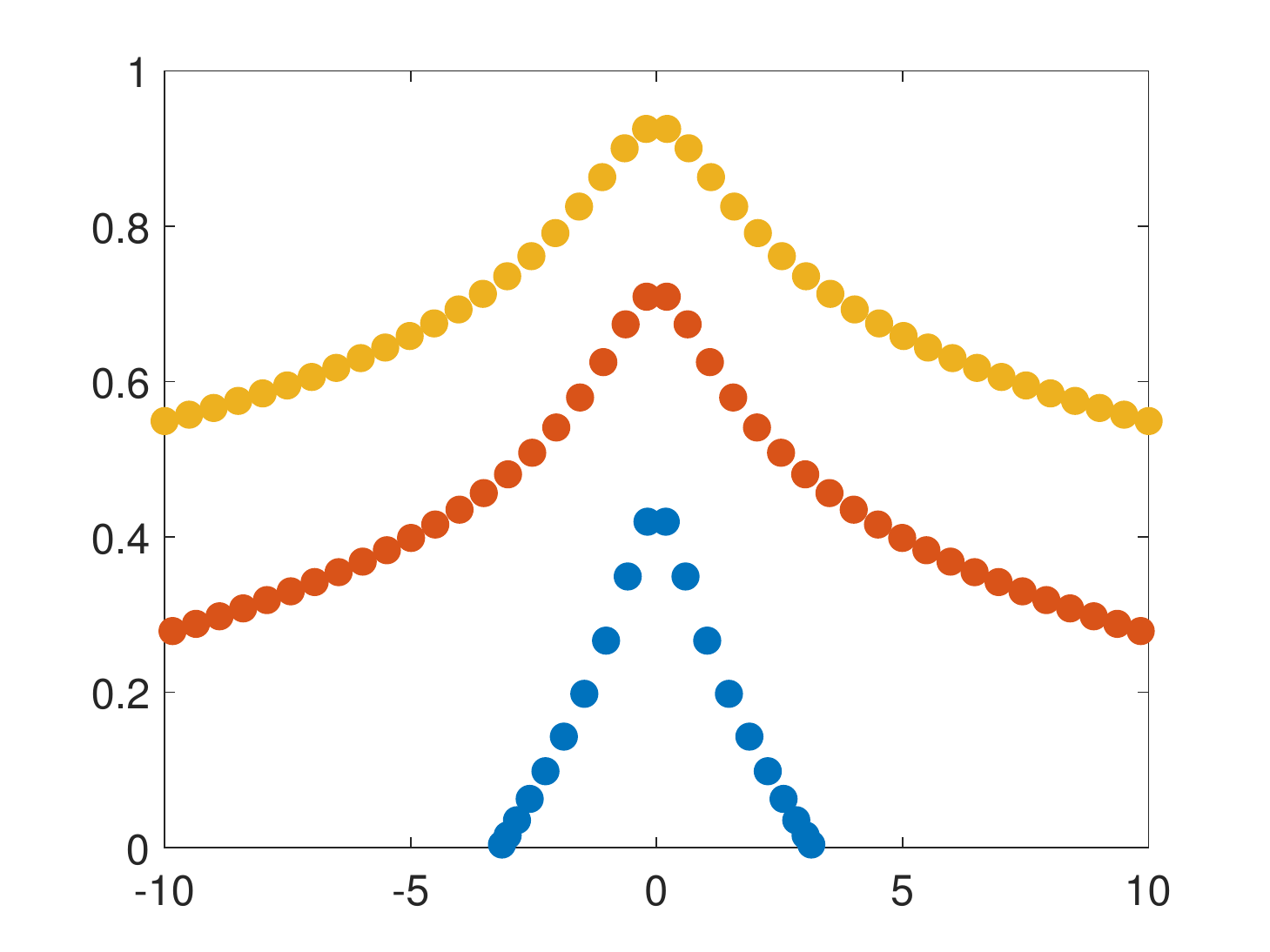}
\put (45,-2) {\small $\mathrm{Re}(\lambda)$}
\put (0,35) {\small\rotatebox{90}{$\mathrm{Im}(\lambda)$}}
\put (44,72) {\small $\chi=1$}
\put (43.5,10) {\small $n=10$}
\put (15,28) {\small \rotatebox{23}{$n=100$}}
\put (15,45) {\small \rotatebox{20}{$n=500$}}
\end{overpic}
\begin{overpic}[width=0.49\textwidth,trim={0mm 0mm 0mm 0mm},clip]{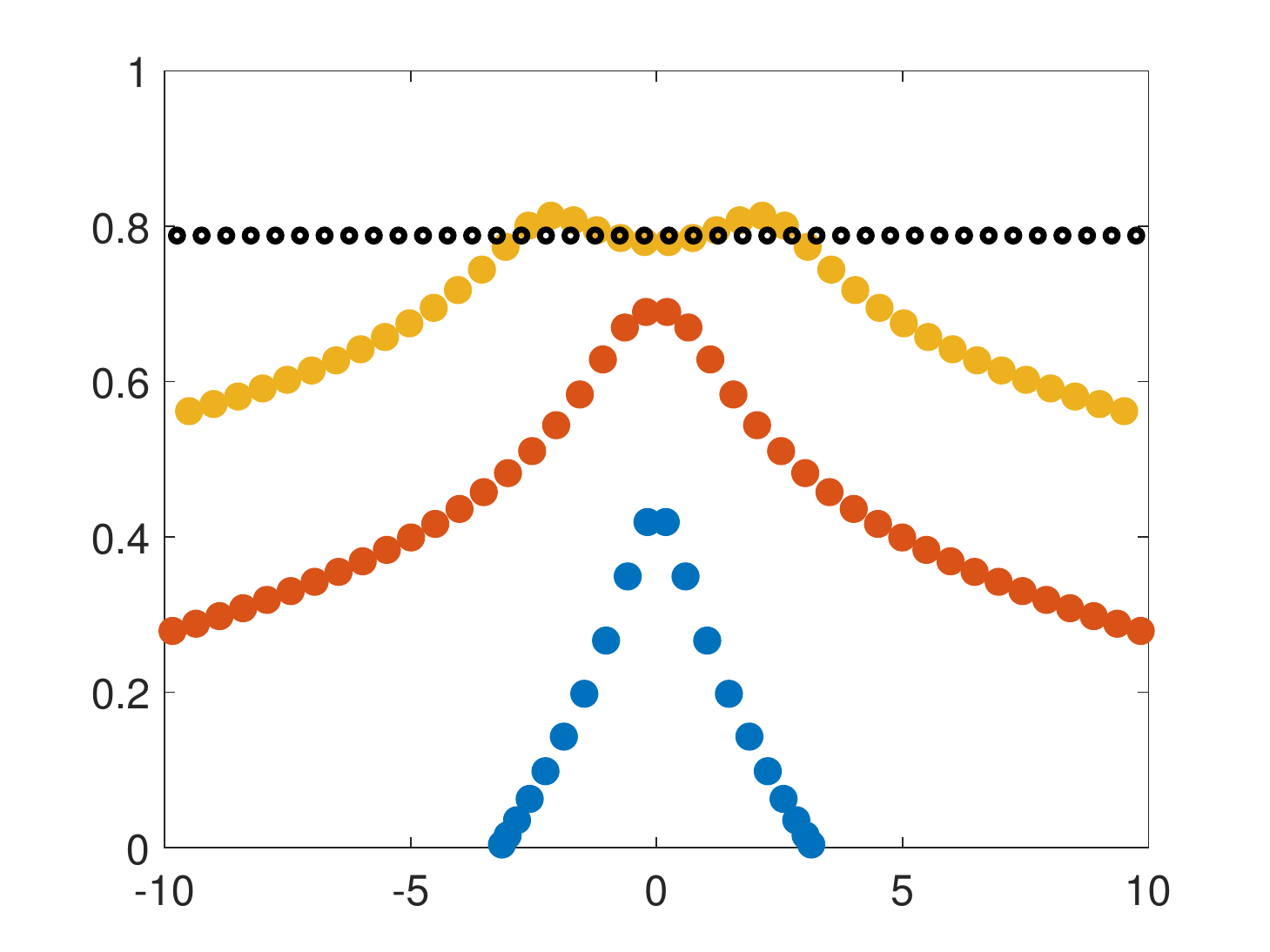}
\put (45,-2) {\small $\mathrm{Re}(\lambda)$}
\put (0,35) {\small\rotatebox{90}{$\mathrm{Im}(\lambda)$}}
\put (39,72) {\small $\chi=1.0001$}
\put (43.5,10) {\small $n=10$}
\put (15,28) {\small \rotatebox{23}{$n=100$}}
\put (15,45) {\small \rotatebox{20}{$n=500$}}
\end{overpic}
\end{minipage}\vspace{-3mm}
\caption{Computed eigenvalues for the discretized acoustic wave 1D example labeled as \texttt{acoustic\_wave\_1d} in the NLEVP collection with discretization size $n = 10$ (blue), $n=100$ (red), and $n=500$ (yellow). Left: The spectrum of the infinite-dimensional problem is empty, and all the computed eigenvalues are spurious. Right: The computed eigenvalues are correctly converging as $n\rightarrow \infty$ to the spectra of the infinite-dimensional problem (black dots), but the convergence is slow.}
\label{fig:Acoustic_Wave_1D}
\end{figure}

We first consider the default value $\chi=1$, which is also a value of $\chi$ that makes $\tan^{-1}(i\chi)$ in~\cref{eq:lambda_formula} infinite.\footnote{More precisely, we consider the operator $T(\lambda):H^2(0,1)\rightarrow L^2(0,1)\times \mathbb{C}^2$ given by $T(\lambda)p=(-p''-({2\pi \lambda}/{c})^2p,p(0),\chi p'(1) + {2\pi i \lambda}p(1))$. The operator $\lambda^2T(1/\lambda)$ has essential spectra at $\lambda=0$. Hence, this problem has essential spectra (not eigenvalues) at infinity for any $\chi$.} We compute the eigenvalues of the discretized problem for three discretization sizes $n=10$, $100$, and $500$ (see~\cref{fig:Acoustic_Wave_1D} (left)). We compute these eigenvalues using the \texttt{polyeig} command in MATLAB. One can easily show that the spectrum of~\cref{eq:AcousticWave1D} is empty for $\chi=1$. This means that all these computed eigenvalues are spurious and can be regarded as meaningless for the original problem in~\cref{eq:AcousticWave1D}.~\cref{fig:Acoustic_Wave_1D_B} (left) shows the minimum absolute value of the eigenvalues as a function of $n$. The eigenvalues march off to infinity, but incredibly slowly.

\begin{figure}
\centering
\begin{minipage}[b]{1\textwidth}
\centering
\begin{overpic}[width=0.49\textwidth,trim={0mm 0mm 0mm 0mm},clip]{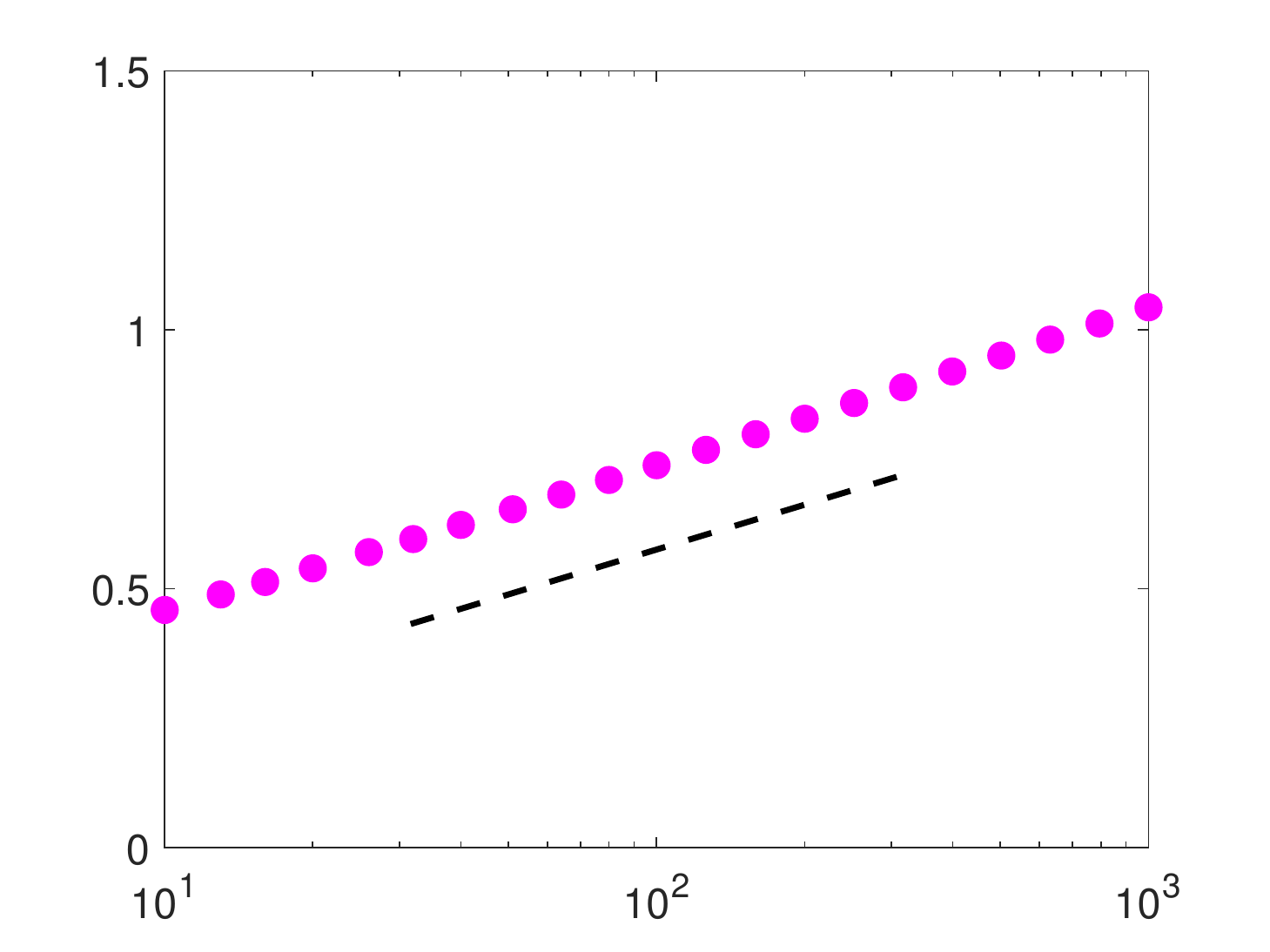}
\put (49,-2) {\small $n$}
\put (42,24) {\small \rotatebox{16}{$\mathcal{O}(\log(n))$}}
\end{overpic}
\begin{overpic}[width=0.49\textwidth,trim={0mm 0mm 0mm 0mm},clip]{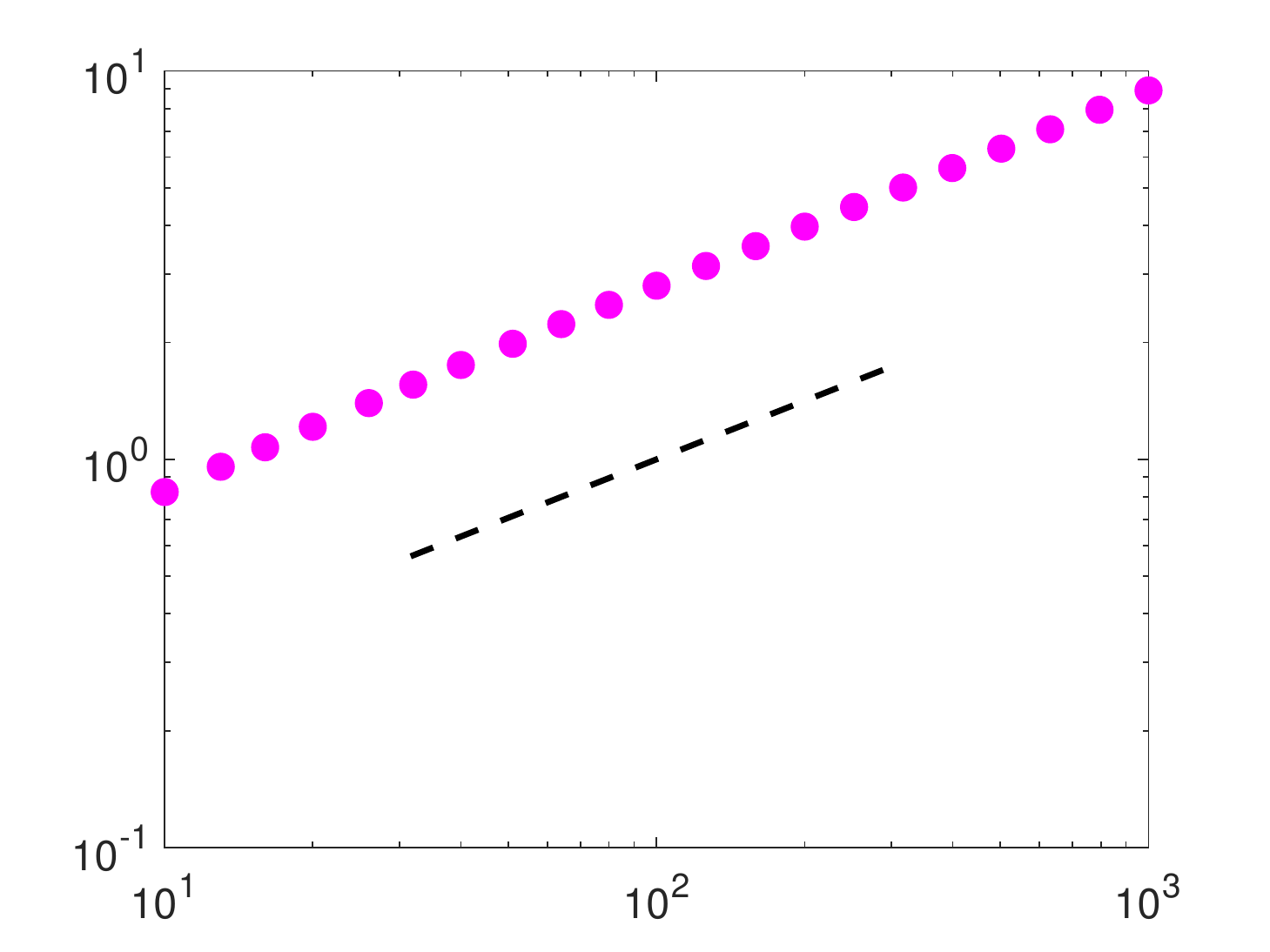}
\put (49,-2) {\small $n_0$}
\put (35,25) {\small \rotatebox{21}{$\mathcal{O}(\sqrt{n_0})=\mathcal{O}({n}^{1/4})$}}
\end{overpic}
\end{minipage}\vspace{-3mm}
\caption{The minimum absolute value of the spurious eigenvalues as a function of $n$. Left: The acoustic wave 1D example. Right: The acoustic wave 2D example, where the true discretization size is $n=n_0(n_0-1)$.}
\label{fig:Acoustic_Wave_1D_B}
\end{figure}
 
We repeat the experiment with the value $\chi=1.0001$ so that~\cref{eq:AcousticWave1D} no longer has an empty spectrum. Again, we discretize~\cref{eq:AcousticWave1D} using FEMs~\cite{harari1996recent} with $n = 10$, $100$, and $500$ and compute the eigenvalues of the matrix NEP using \texttt{polyeig}. The computed eigenvalues are now converging as $n\rightarrow\infty$; however, the rate is very slow (see~\cref{fig:Acoustic_Wave_1D} (right)). This example shows that even when the eigenvalues of the discretization are converging, it can be computationally prohibitive if the rate is slow.

\subsection{Example 2: Two-dimensional acoustic wave}\label{sec:AcousticWave2d}
In the NLEVP collection, a 2D acoustic wave example goes under the name of \texttt{acoustic\_wave\_2d} and is discretized using FEMs. The NEP is posed on $L^2([0,1]^2)$ and given by
\begin{equation}\setlength\abovedisplayskip{6pt}\setlength\belowdisplayskip{6pt}
\label{acoustic_wave_2d_problem_def}
\begin{gathered}
\frac{\partial^2p}{\partial x^2}+\frac{\partial^2p}{\partial y^2} + 4\pi^2\lambda^2p=0\\
p(0,y)=p(x,0)=p(x,1)=0,\quad\chi\frac{\partial p}{\partial x}(1,y)+2\pi i\lambda p(1,y)=0,
\end{gathered}
\end{equation}
with the same meaning of the parameters as in~\cref{eq:AcousticWave1D}.

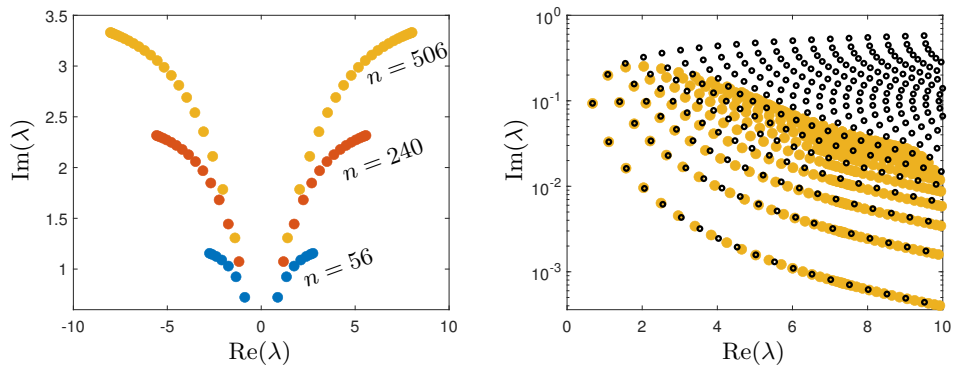
\begin{figure}
\centering
\begin{minipage}[b]{1\textwidth}
\centering
\begin{overpic}[width=0.49\textwidth,trim={0mm 0mm 0mm 0mm},clip]{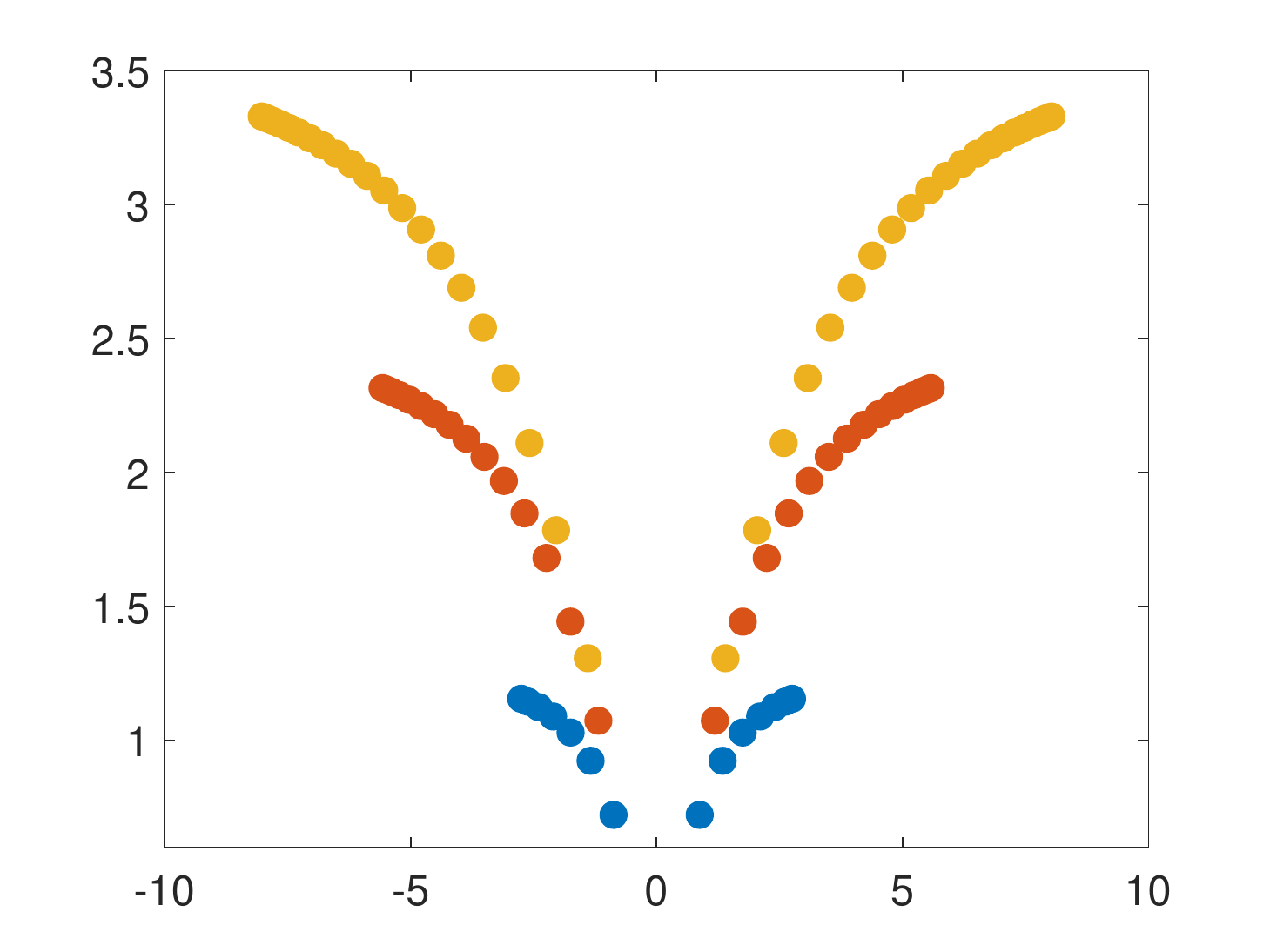}
\put (45,-2) {\small $\mathrm{Re}(\lambda)$}
\put (0,35) {\small\rotatebox{90}{$\mathrm{Im}(\lambda)$}}
\put (60,13) {\small \rotatebox{20}{$n=56$}}
\put (68,35) {\small \rotatebox{20}{$n=240$}}
\put (73,55) {\small \rotatebox{20}{$n=506$}}
\end{overpic}
\begin{overpic}[width=0.49\textwidth,trim={0mm 0mm 0mm 0mm},clip]{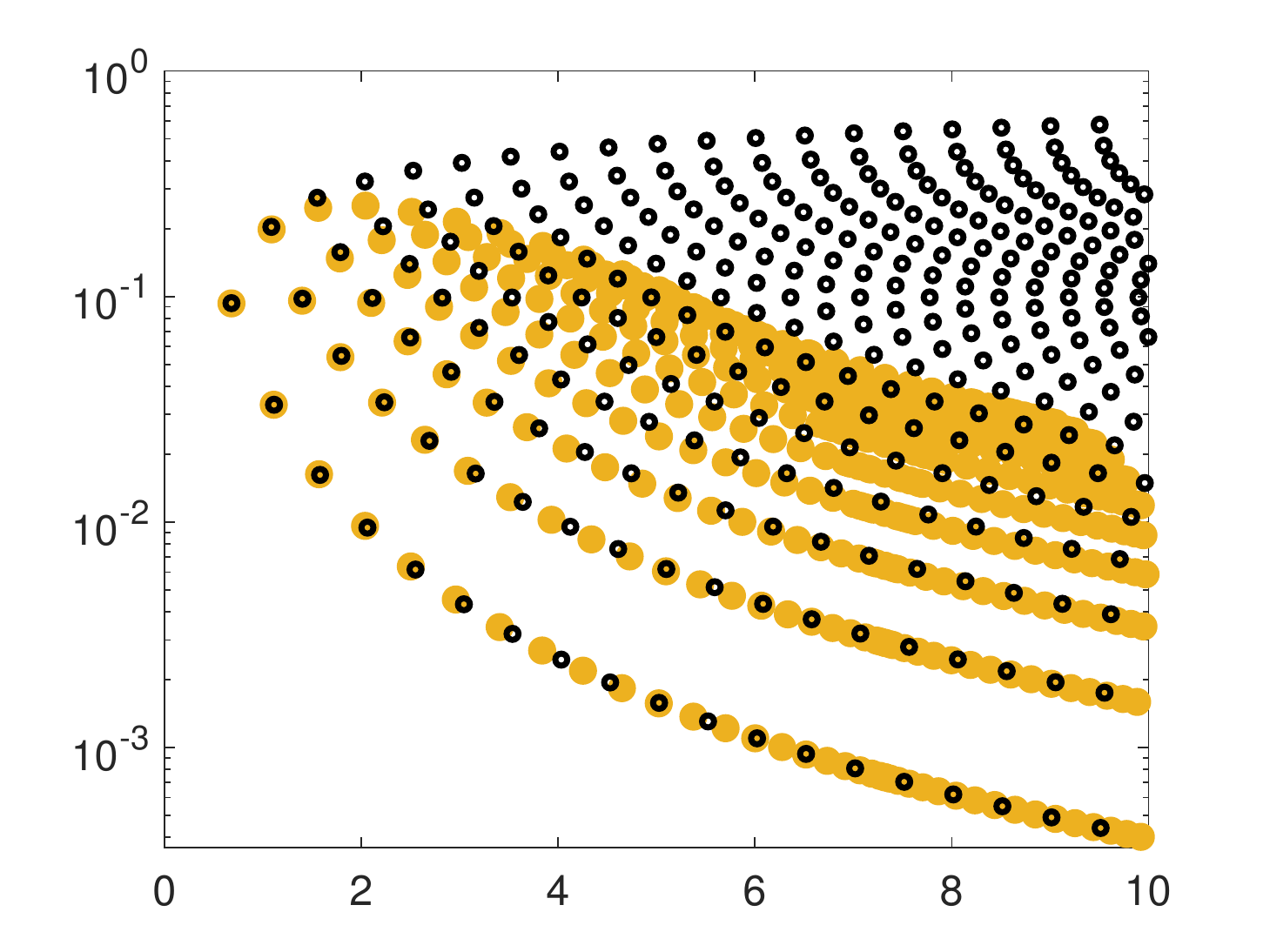}
\put (45,-2) {\small $\mathrm{Re}(\lambda)$}
\put (0,35) {\small\rotatebox{90}{$\mathrm{Im}(\lambda)$}}
\end{overpic}
\end{minipage}\vspace{-3mm}
\caption{The spectra of the acoustic wave 2D example for $\chi=1$. Left: Eigenvalues of the discretized problem, \texttt{acoustic\_wave\_2d}, for discretization sizes $n = 56$, $240$, and $506$. The discretization sizes are constrained to be of the form $n=n_0(n_0-1)$. There are no eigenvalues of~\cref{acoustic_wave_2d_problem_def} in this region. Right: A region in the complex plane close to the real axis that does contain eigenvalues of~\cref{acoustic_wave_2d_problem_def}. These are shown as black dots and are computed using InfBeyn. The yellow dots are eigenvalues of the discrete problem for $n=506$, showing convergence to a proportion of them. We only show a region in the right-half plane because the spectrum of the infinite-dimensional NEP and eigenvalues of the discretization are symmetric about the imaginary axis.}
\label{fig:Acoustic_Wave_2D}
\end{figure}

We first select $\chi=1$. For this value of $\chi$, the spectra of~\cref{acoustic_wave_2d_problem_def} is nonempty, unlike the 1D case. However,~\cref{acoustic_wave_2d_problem_def} has no eigenvalues in the region $[-10,10]\times [0.6,4]$ in the complex plane, which can be proved using the trace formula in~\cref{eq:trace_formula}. Despite this, the matrix NEPs for $n=56$, $240$, and $506$ have spurious eigenvalues in that region caused by the discretization (see~\cref{fig:Acoustic_Wave_2D} (left)).~\cref{fig:Acoustic_Wave_1D_B} (right) shows just how severe this is, the spurious eigenvalues only exit the disc of radius $10$ after a discretization size in excess of $10^6$. There is a region close to the real axis in the complex plane for which the discretizations do correctly approximate the spectra (see~\cref{fig:Acoustic_Wave_2D} (right))—determining which regions the discretization will have spurious eigenvalues and which regions the computed eigenvalues can be trusted seems challenging.\footnote{To give an idea of how hard this is, the location of spectral pollution for linear eigenvalue problems has only very recently been characterized in any sense of generality~\cite{bogli2020essential}.} In contrast, InfBeyn correctly returns no eigenvalues in the region $[-10,10]\times [0.6,4]$ in the complex plane and accurately computes the eigenvalues close to the real axis. In summary, the discretizations exhibit spurious eigenvalues in one region and in another, exhibit slow convergence. We find this example cause for concern because the two regions are relatively close together in the complex plane, making it challenging to identify spectral pollution after discretization.

When $\chi\not\in(-\infty,-1]\cup[1,\infty)$, a subset of the spectrum is given by an infinite number of simple eigenvalues that obey the following asymptotic formula:
$$\setlength\abovedisplayskip{6pt}\setlength\belowdisplayskip{6pt}
\lambda_k\sim {\mathrm{sign}\left[\mathrm{Re}\left(i\sqrt{\frac{1}{\chi^2-1}}\right)\right]}k/({2\sqrt{1-1/\chi^2}}),\qquad k\rightarrow\infty,
$$
where the sign function is required to take care of branch cuts. We now take $\chi=0.8$, and for this value of $\chi$, the eigenvalues of~\cref{acoustic_wave_2d_problem_def} that obey the above asymptotic formula are purely imaginary. \cref{fig:Acoustic_Wave_2D2} shows the approximation of these eigenvalues using the discrete problem with different discretization sizes. Again, the eigenvalues of the infinite-dimensional problem are shown as black circles and computed using InfBeyn. The eigenvalues of the discrete problem are symmetric across the imaginary axis. As the $k$th pair approach the imaginary axis, they collide, and the pair splits. One eigenvalue converges to $\lambda_k$, while the other shoots off to infinity. In other words, the discrete problem overestimates the actual multiplicity. 

\begin{figure}
\centering
\begin{minipage}[b]{1\textwidth}
\centering
\begin{overpic}[width=0.32\textwidth,trim={0mm 0mm 0mm 0mm},clip]{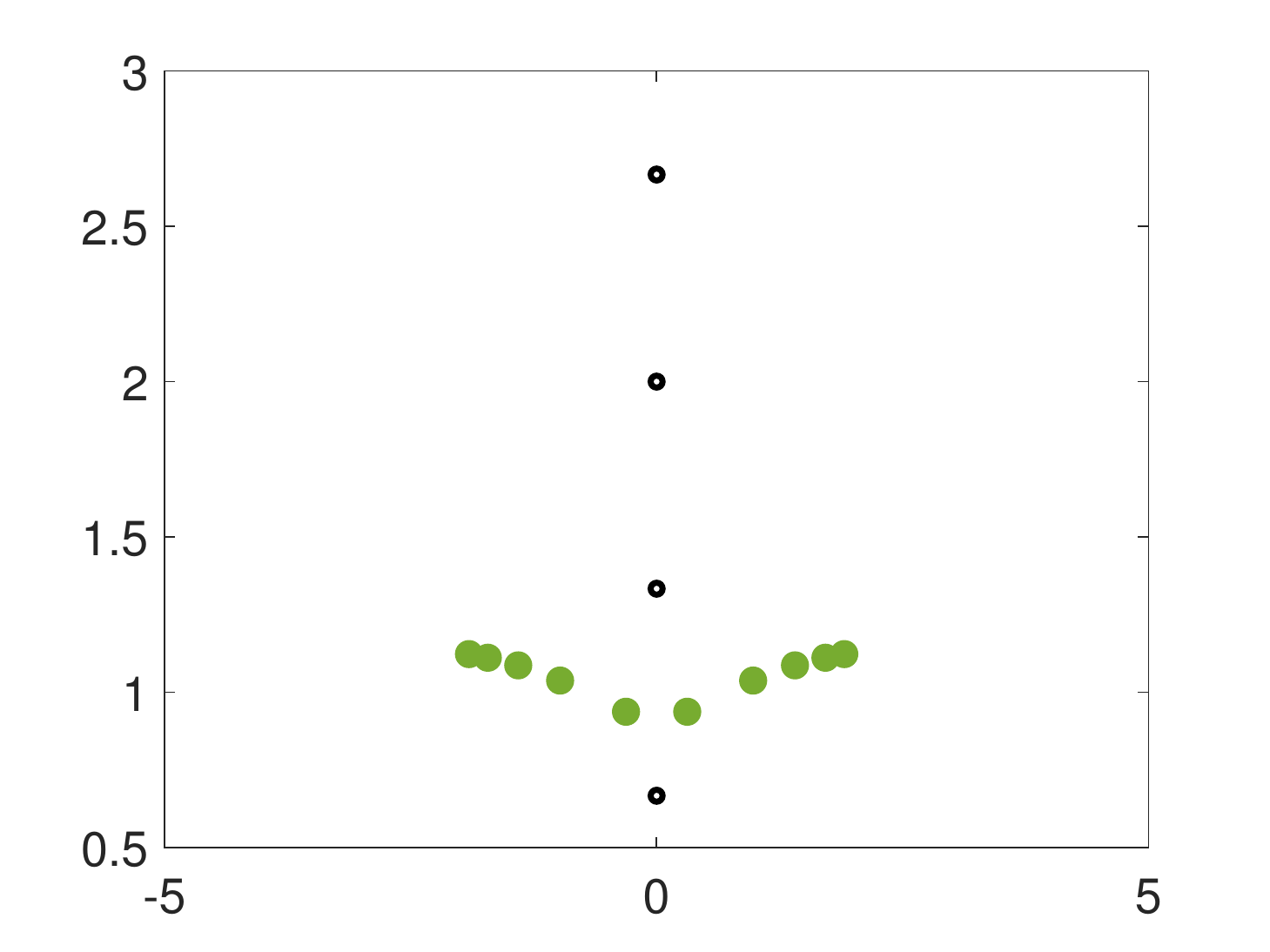}
\put (45,-2) {\small $\mathrm{Re}(\lambda)$}
\put (0,35) {\small\rotatebox{90}{$\mathrm{Im}(\lambda)$}}
\put (44,72) {\small $n=30$}
\end{overpic}
\begin{overpic}[width=0.32\textwidth,trim={0mm 0mm 0mm 0mm},clip]{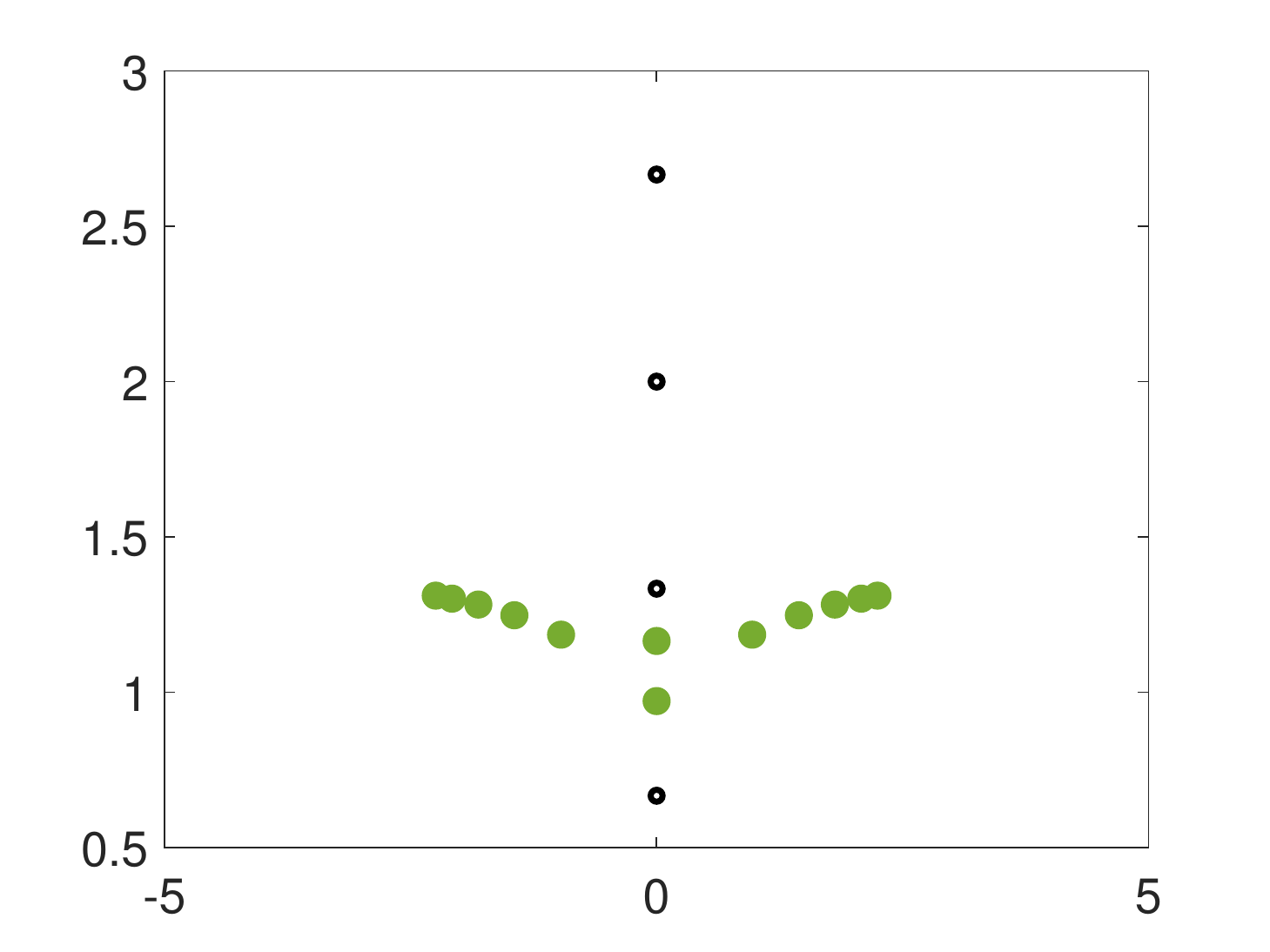}
\put (45,-2) {\small $\mathrm{Re}(\lambda)$}
\put (0,35) {\small\rotatebox{90}{$\mathrm{Im}(\lambda)$}}
\put (44,72) {\small $n=42$}
\end{overpic}
\begin{overpic}[width=0.32\textwidth,trim={0mm 0mm 0mm 0mm},clip]{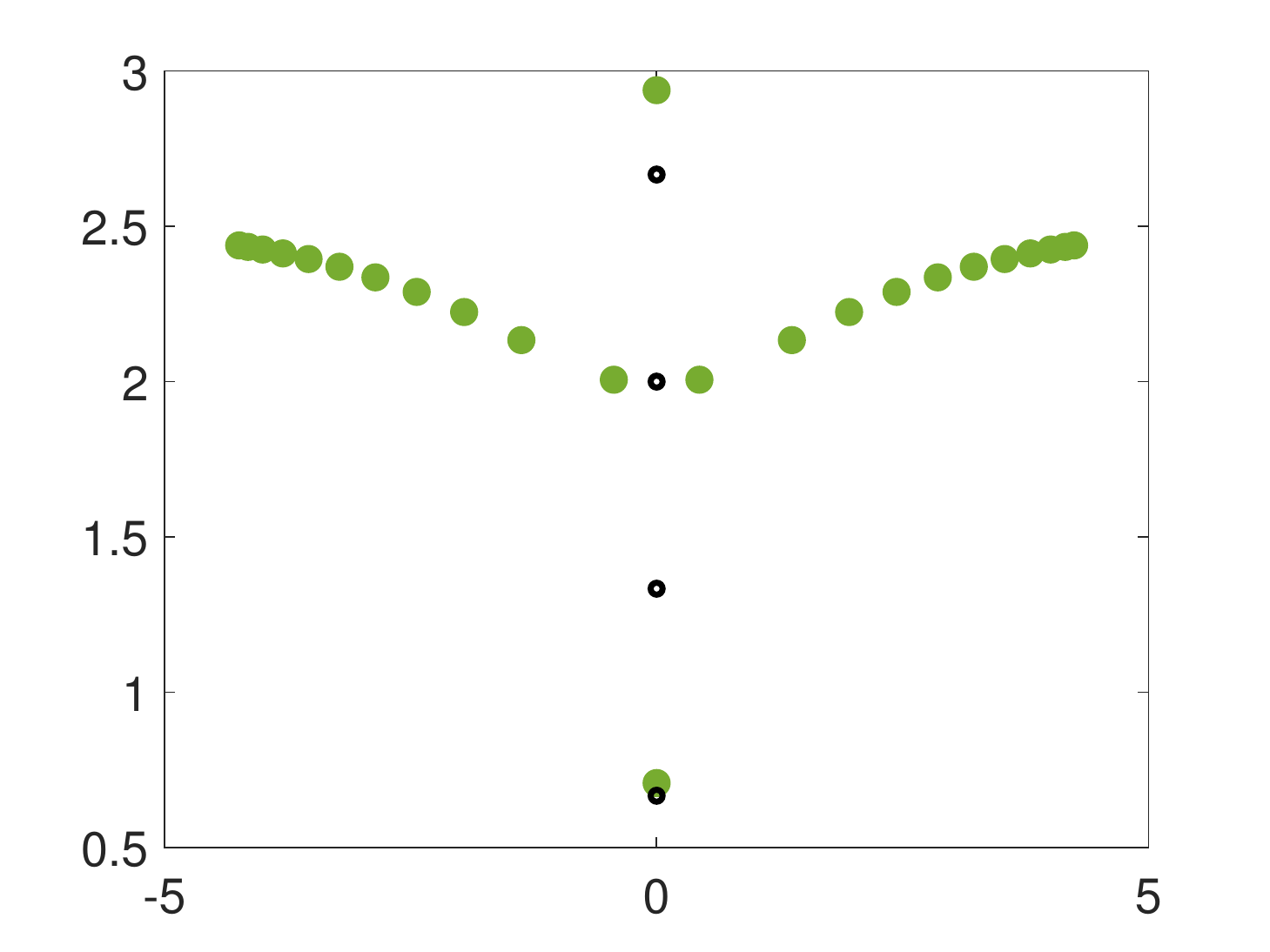}
\put (45,-2) {\small $\mathrm{Re}(\lambda)$}
\put (0,35) {\small\rotatebox{90}{$\mathrm{Im}(\lambda)$}}
\put (44,72) {\small $n=156$}
\end{overpic}
\end{minipage}\vspace{-3mm}
\caption{As the discretization size increases, we observe the eigenvalues of the discretization (green dots) collide onto the imaginary axis, and a few converge to the eigenvalues of the infinite-dimensional problem (black dots). The eigenvalues of the discretization are potentially misleading because the eigenvalues of the infinite-dimensional problem are simple. Still, it appears that two eigenvalues of the discretization are converging to each one.}
\label{fig:Acoustic_Wave_2D2}
\end{figure}
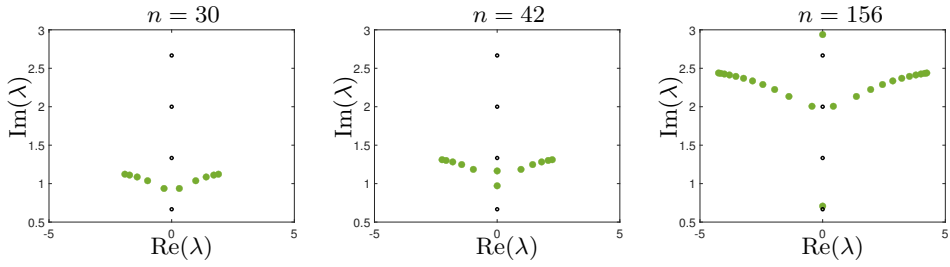

\subsection{Example 3: Butterfly}\label{sec:butterfly}
As our next example, we further show the importance of verification of approximated pseudospectra in~\cref{pseudospectra_verified} and that our techniques are not limited to differential operators. We consider the NEP called \texttt{butterfly} from the NLEVP collection, which is a rational NEP constructed from truncations of bilateral shift operators on $\ell^2(\mathbb{Z})$~\cite{mehrmann2002polynomial}. The pencil depends on a vector $c\in\mathbb{C}^{10}$ which we take as $c = [0.2i, 0, 1.3, 0, 0.1, 0, 1, 0, 0, 0]$.

\cref{fig:butterfly} (left) shows the eigenvalues and pseudospectra of the discretized problem using matrix sizes $n=500$. The eigenvalues appear to converge to four arcs in the complex plane as $n\rightarrow\infty$. In the right of~\cref{fig:butterfly}, we show pseudospectra computed using the functions $\gamma_n$ from~\cref{def_gamma_n}. The operators are infinite banded matrices acting on $l^2(\mathbb{Z})$, and hence it is straightforward to compute $\gamma_n$ directly using rectangular truncations. We use a $\lambda$-adaptive truncation size to ensure convergence of the plot. The plots clearly show that the discretized operator suffers from spectral pollution, invisibility, and destabilization. For this particular example, changing the discretization to circulant matrices approximating the shift is better. However, in general, such a procedure is not guaranteed to circumvent the issues caused by discretization.

\begin{figure}
\centering
\begin{minipage}[b]{1\textwidth}
\centering
\begin{overpic}[width=0.49\textwidth,trim={0mm 0mm 0mm 0mm},clip]{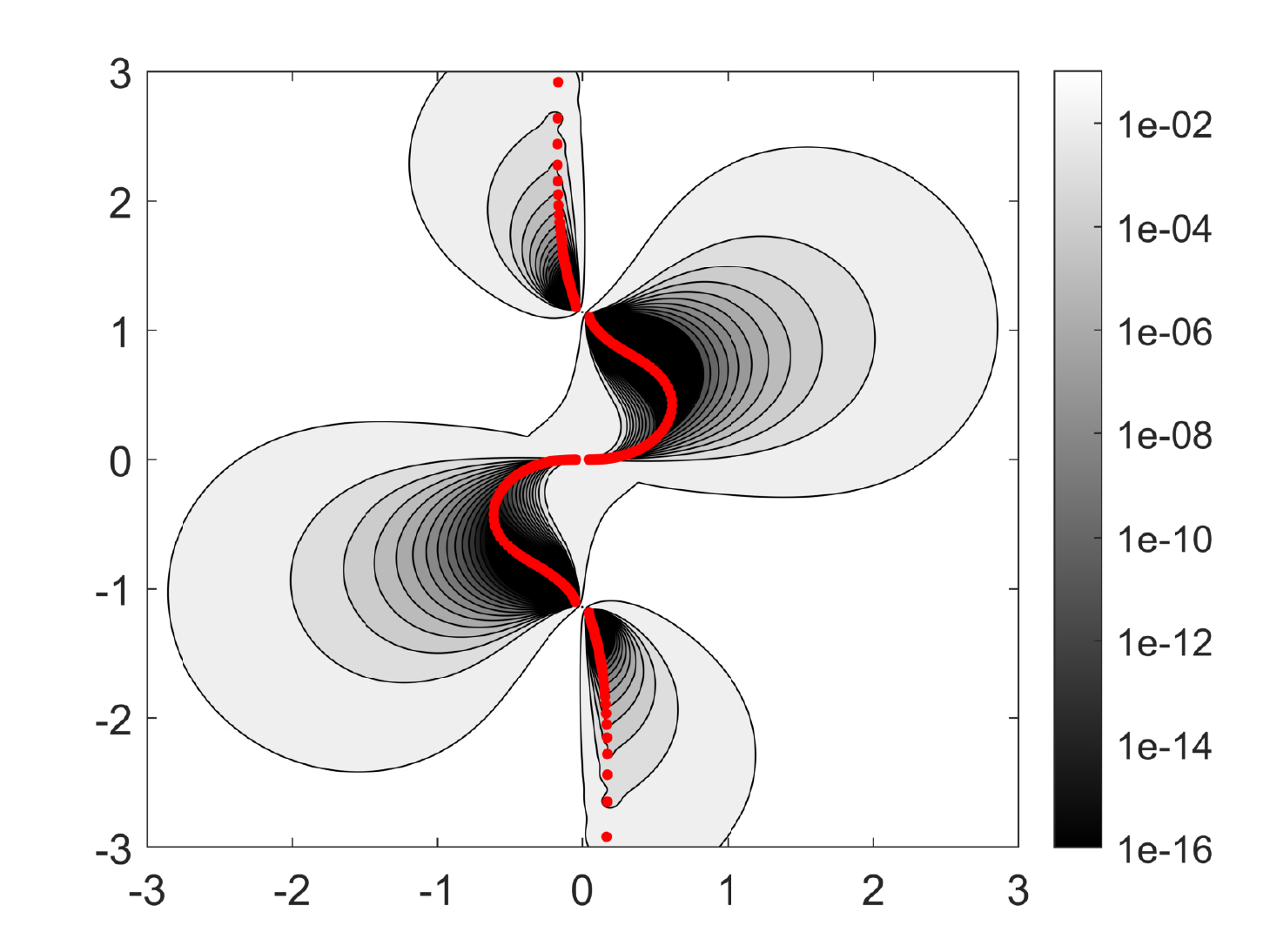}
\put (45,-2) {\small $\mathrm{Re}(\lambda)$}
\put (0,35) {\small\rotatebox{90}{$\mathrm{Im}(\lambda)$}}
\end{overpic}
\begin{overpic}[width=0.49\textwidth,trim={0mm 0mm 0mm 0mm},clip]{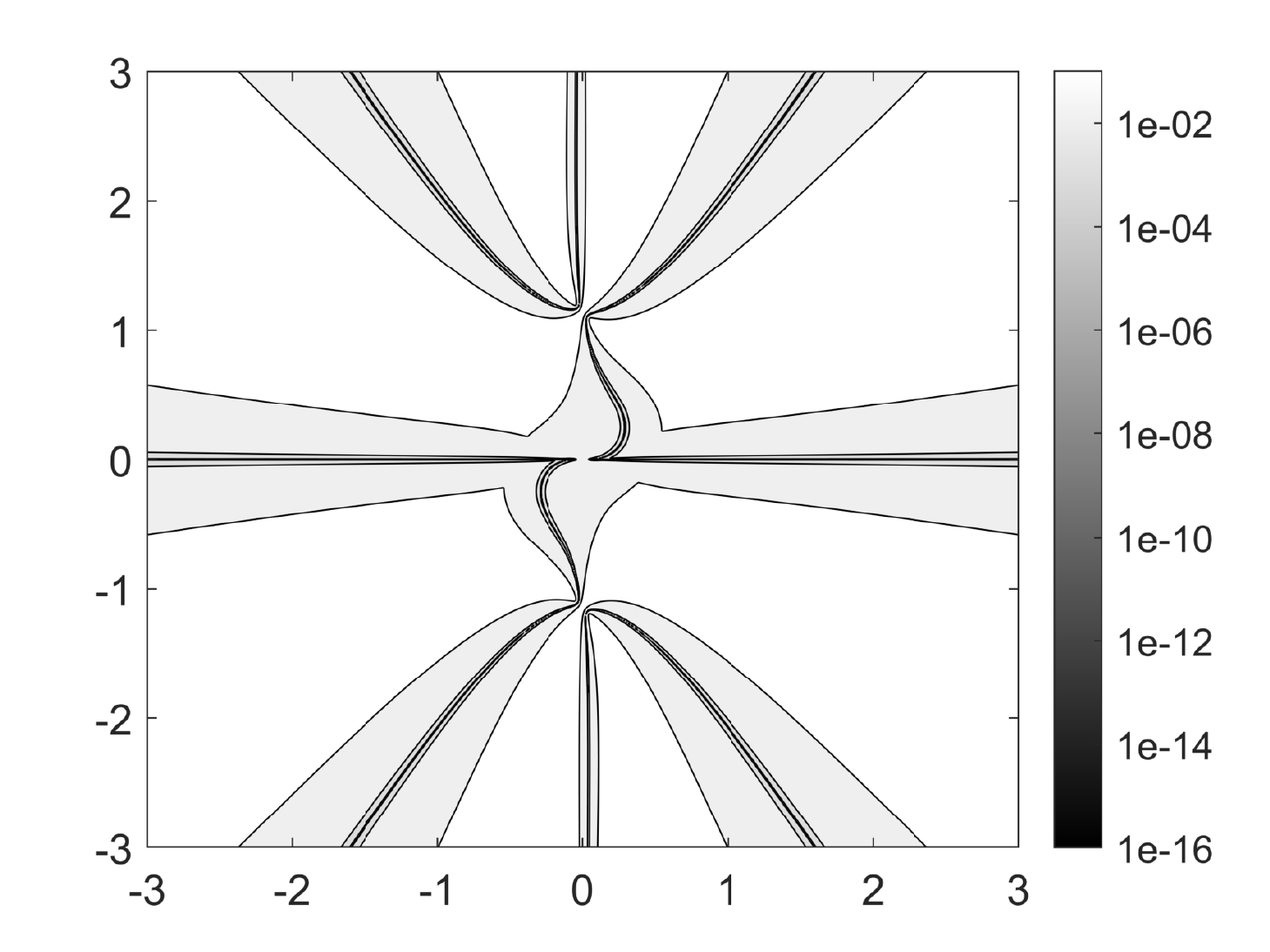}
\put (45,-2) {\small $\mathrm{Re}(\lambda)$}
\put (0,35) {\small\rotatebox{90}{$\mathrm{Im}(\lambda)$}}
\end{overpic}
\end{minipage}\vspace{-3mm}
\caption{The computed pseudospectra of the butterfly NEP. Left: The pseudospectra of the discretized pencil using matrices of size $n=500$. The eigenvalues are shown as red dots and converge to the union of four arcs as $n\rightarrow\infty$. Right: The pseudospectra computed using the functions $\gamma_n$ from~\cref{def_gamma_n} and an adaptive truncation size. These pseudospectra are guaranteed to be inside the pseudospectra of the infinite-dimensional problem and converge as $n\rightarrow\infty$.}
\label{fig:butterfly}
\end{figure}

\subsection{Example 4: Damped beam}\label{sec:damped_beam} 
We now consider the NEP that goes by the name \texttt{damped\_beam} in the NLEVP collection. It is given by
\begin{equation}\setlength\abovedisplayskip{6pt}\setlength\belowdisplayskip{6pt}
\label{damped_beam_problem_def}
\frac{d^4 v}{dx^4}(x)-\alpha_0 \lambda^2 v(x)=\beta \lambda v(x) \delta(x-{1}/{2}),\quad v(0)=\frac{d^2v}{dx^2}(0) =v(1)=\frac{d^2v}{dx^2}(1)=0,
\end{equation}
where $\alpha_0,\beta<0$ are fixed physical constants and $v$ represents the transverse displacement of the beam. We take the default NLEVP parameter values of $\alpha_0=-0.018486857142857$ and $\beta=-0.137142857142857$. The function $\delta(\cdot)$  in~\cref{damped_beam_problem_def} is the delta function and is interpreted as continuity of $v$, $v'$, and $v''$ at $x=1/2$, but with a jump in $v'''$, i.e.,
$$\setlength\abovedisplayskip{6pt}\setlength\belowdisplayskip{6pt}
\lim_{\epsilon\downarrow 0} \left[ \frac{d^3 v}{dx^3}(1/2+\epsilon)-\frac{d^3 v}{dx^3}(1/2-\epsilon)\right]=\beta \lambda v(1/2).
$$ 
The NEP is a quadratic eigenvalue problem that arises in the vibration analysis of a beam supported at both ends and damped in the middle~\cite{higham2008scaling}. 

We discretize~\cref{damped_beam_problem_def} using a finite element method with cubic Hermite polynomials as the interpolation shape functions~\cite{collar1987matrices}. There are two groups of eigenvalues for~\cref{damped_beam_problem_def}. The first group is purely imaginary and given by the following formula:
$$\setlength\abovedisplayskip{6pt}\setlength\belowdisplayskip{6pt}
\lambda^{(1,\pm)}_k=\pm{4\pi^2 k^2i}/{\sqrt{-\alpha_0}},\quad k\geq 0, 
$$
with corresponding eigenfunctions that vanish at $x=1/2$. The second group has the following asymptotic formula: 
\begin{equation}\setlength\abovedisplayskip{6pt}\setlength\belowdisplayskip{6pt}
\label{asymptotics_damped_beam}
\lambda^{(2,\pm)}_k = \frac{4}{\sqrt{\alpha_0}}\left[\pm\left(k\pi-\frac{\pi}{2}\right)i +\frac{\beta}{8k\pi \sqrt{-\alpha_0}} \right]^2+\mathcal{O}\left(\frac{1}{k}\right),\qquad k\rightarrow\infty.
\end{equation}
Asymptotic formulas are beneficial for contour-based methods as they inform us of where to center contours. In addition, if the asymptotic formula comes with an explicit error estimate, then we can choose the size of the contour. The asymptotic formula in~\cref{asymptotics_damped_beam} allows us to compute $\lambda^{(2,\pm)}_k$ for large $k$ using InfBeyn with a circular contour of radius $1$ centered at~\cref{asymptotics_damped_beam}. An alternative to asymptotics are localization theorems for NEPs~\cite{bindel2015localization}, which are also very useful for contour methods.

\begin{figure}
\centering
\begin{minipage}[b]{1\textwidth}
\centering
\begin{overpic}[width=0.49\textwidth,trim={0mm 0mm 0mm 0mm},clip]{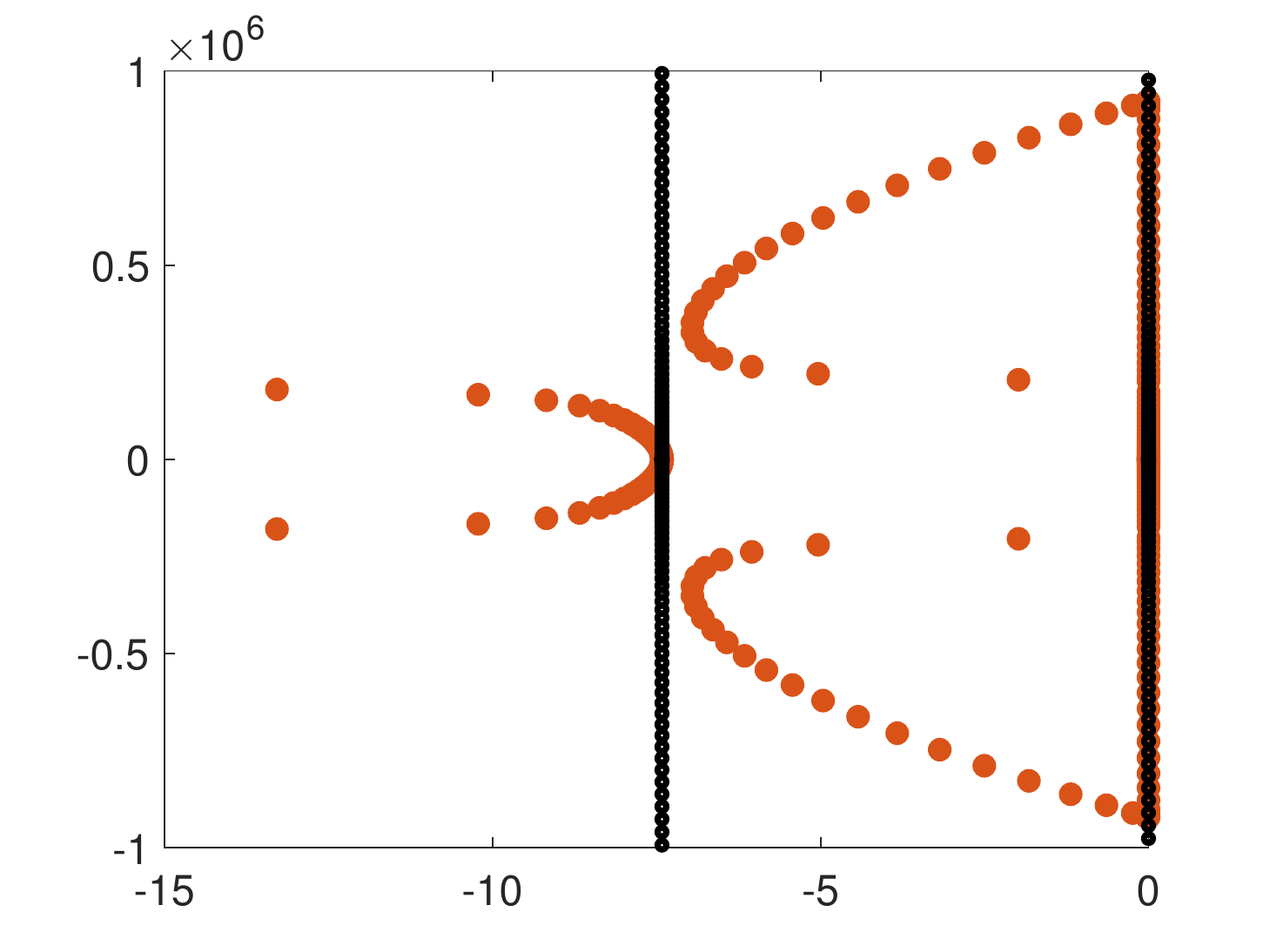}
\put (45,-2) {\small $\mathrm{Re}(\lambda)$}
\put (0,35) {\small\rotatebox{90}{$\mathrm{Im}(\lambda)$}}
\end{overpic}
\begin{overpic}[width=0.49\textwidth,trim={0mm 0mm 0mm 0mm},clip]{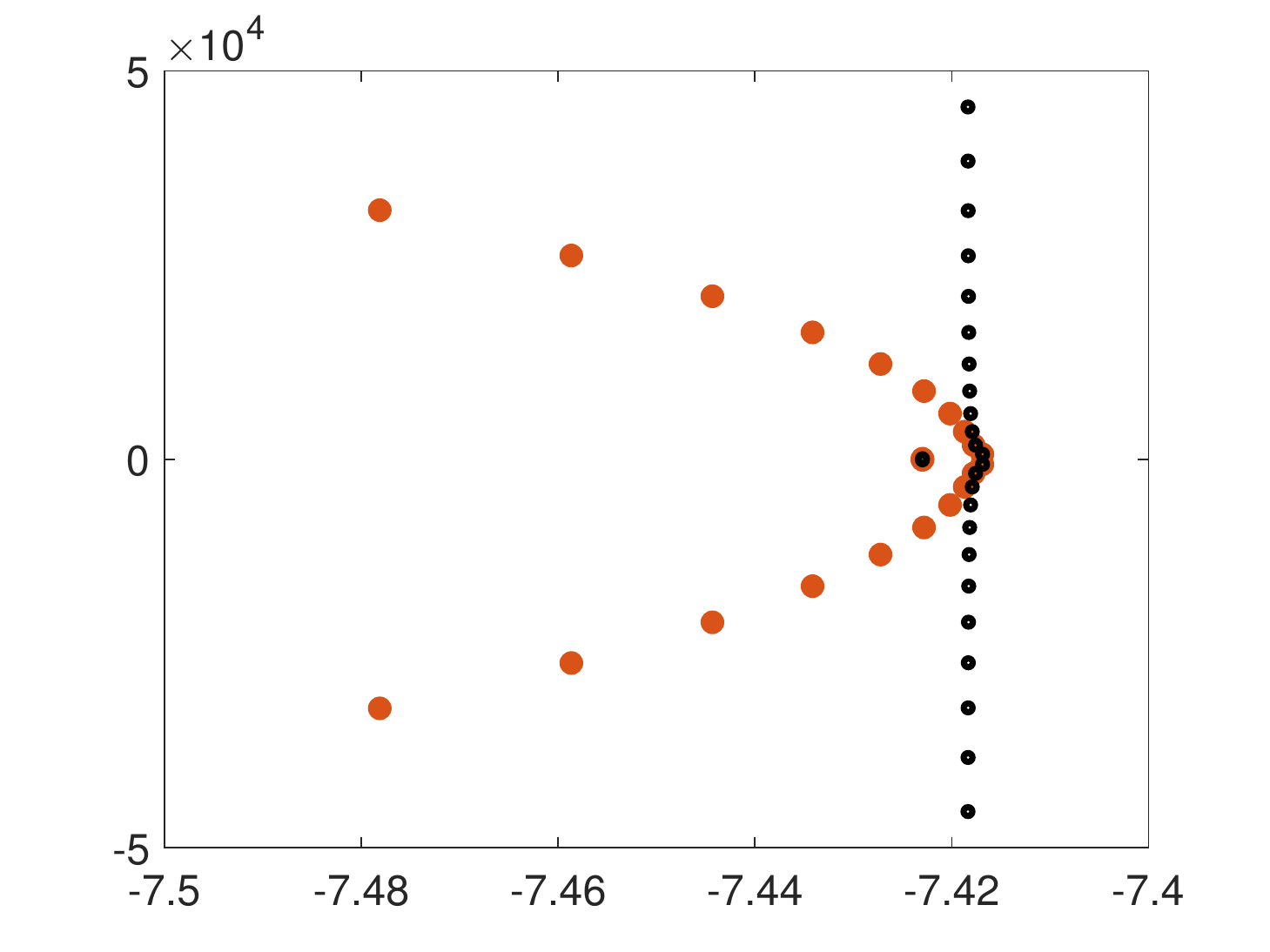}
\put (45,-2) {\small $\mathrm{Re}(\lambda)$}
\put (0,35) {\small\rotatebox{90}{$\mathrm{Im}(\lambda)$}}
\end{overpic}
\end{minipage}\vspace{-3mm}
\caption{Computed eigenvalues of discretization (red dots) compared to the eigenvalues computed from InfBeyn (black dots). Left: Eigenvalues in the region $[-15,0]\times [-10^6,10^6]$. Right: A magnified picture of the eigenvalues in the region $[-7.5,-7.4]\times [-5\times 10^4,5\times 10^4]$. The eigenvalues of the discretized NEP show significant errors.}
\label{fig:damped_beam_1}
\end{figure}

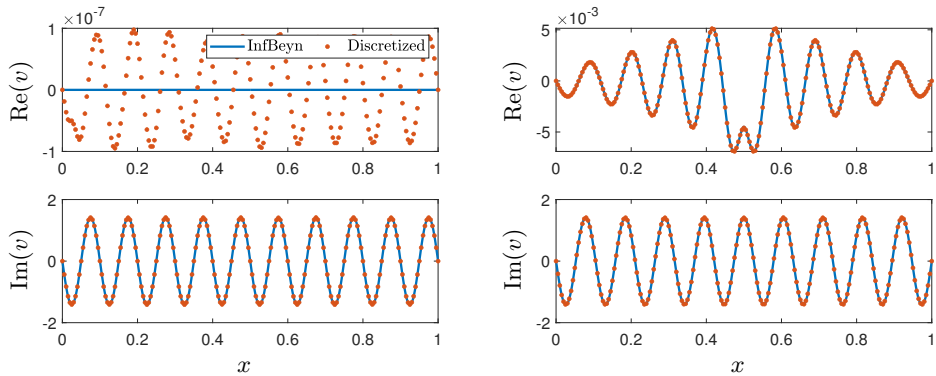
\begin{figure}
\centering
\begin{minipage}[b]{1\textwidth}
\centering
\begin{overpic}[width=0.49\textwidth,trim={0mm 0mm 0mm 0mm},clip]{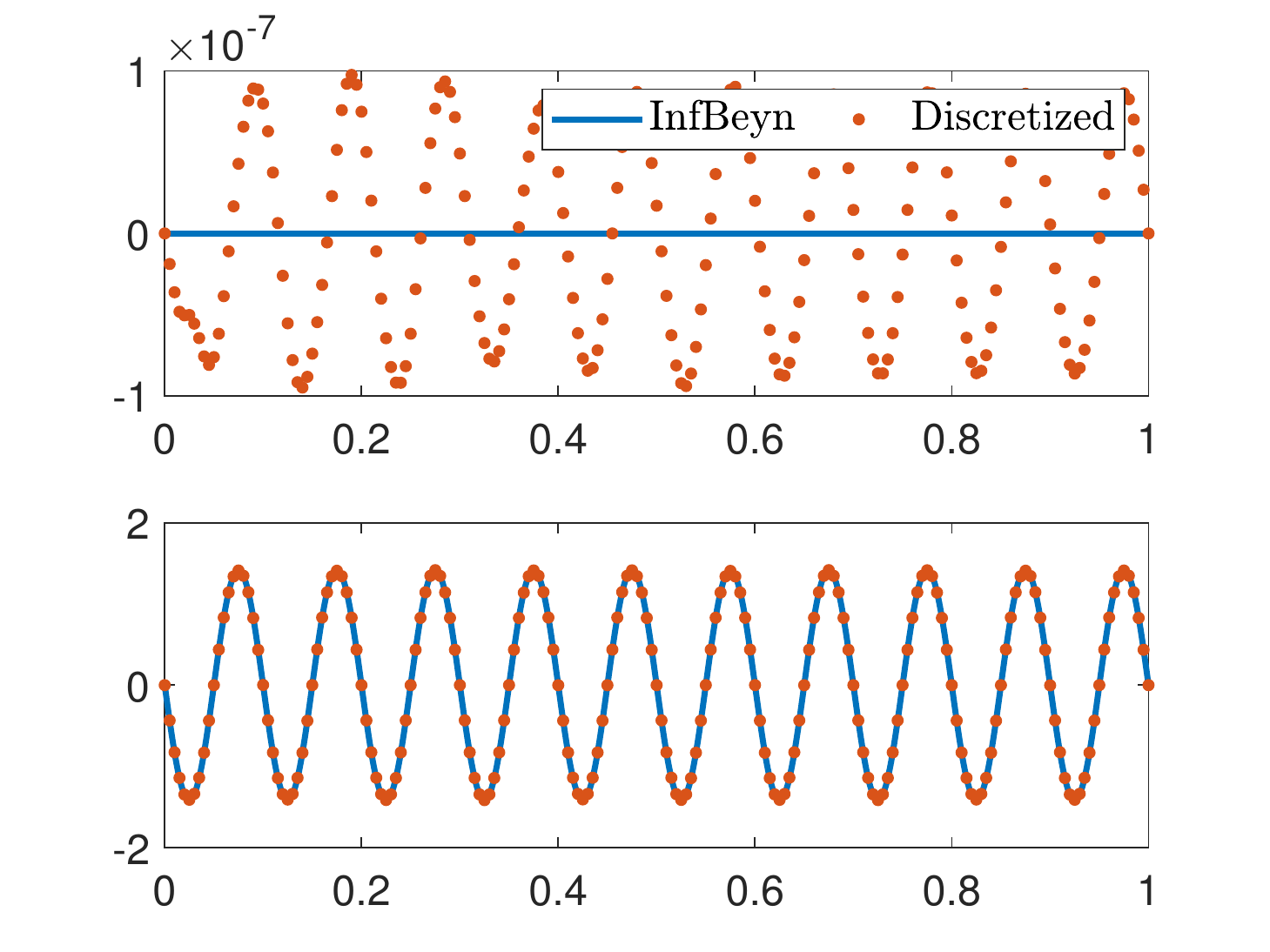}
\put (49,-2) {\small $x$}
\put (2,50) {\small \small\rotatebox{90}{$\mathrm{Re}(v)$}}
\put (2,15) {\small \small\rotatebox{90}{$\mathrm{Im}(v)$}}
\end{overpic}
\begin{overpic}[width=0.49\textwidth,trim={0mm 0mm 0mm 0mm},clip]{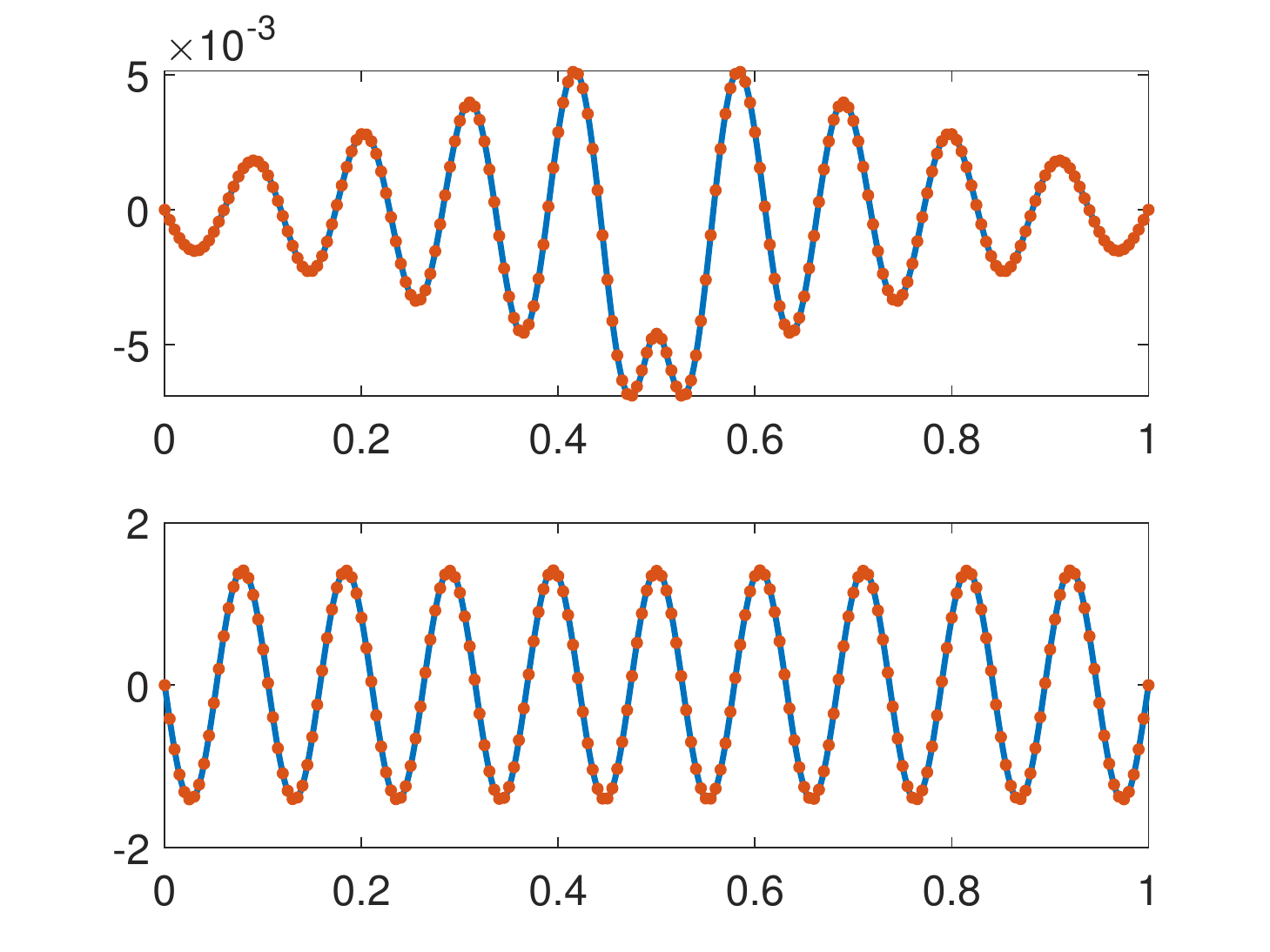}
\put (49,-2) {\small $x$}
\put (2,50) {\small \small\rotatebox{90}{$\mathrm{Re}(v)$}}
\put (2,15) {\small \small\rotatebox{90}{$\mathrm{Im}(v)$}}
\end{overpic}
\end{minipage}\vspace{-3mm}
\caption{The real (top row) and imaginary (bottom row) of the eigenfunctions corresponding to $\lambda^{(1,+)}_{10}$ (left) and $\lambda^{(2,+)}_{10}$ (right). Surprisingly, the eigenfunctions are well-resolved by the discrete NEP while the corresponding eigenvalues $\lambda^{(1,+)}_{10}$ and $\lambda^{(2,+)}_{10}$ are not.}
\label{fig:damped_beam_2}
\end{figure}

\cref{fig:damped_beam_1} shows the eigenvalues of the discretized problem for the discretization size $n=100$, as well as the eigenvalues computed using InfBeyn. Comparing InfBeyn's approximation of $\lambda^{(2,+)}_k$ for $1\leq k\leq 100$ and the first four terms of the asymptotics shows that InfBeyn has computed all of the eigenvalues in~\cref{fig:damped_beam_2} to relative error smaller than $10^{-12}$. The discretization does a good job of approximating the real part of the first group of eigenvalues $\{\lambda^{(1,\pm)}_k\}$, but only a handful of the eigenvalues are accurate due to errors in the imaginary part. Surprisingly,  we observe that the corresponding eigenfunctions are well-resolved by the discretization. For example, \cref{fig:damped_beam_2} shows the approximation of the eigenfunctions corresponding to $\lambda^{(1,+)}_{10}$ and $\lambda^{(2,+)}_{10}$. The $L^2([0,1])$ subspace angle between the approximate eigenfunction and the true eigenfunction (computed using InfBeyn) are approximately $10^{-3}$. However, the error in the approximation of $\lambda^{(1,+)}_{10}$ and $\lambda^{(2,+)}_{10}$, are 48.1040 and 35.5109, respectively. Therefore, unlike linear eigenvalue problems, resolving the eigenfunctions is insufficient for accurately computing the corresponding eigenvalues. We find this extremely unsettling.~\cref{fig:damped_beam_3} shows the computed eigenvalues when we replace $\alpha_0$~\cref{damped_beam_problem_def} by a variable coefficient $\alpha(x)$.\footnote{For NEPs consisting of coupled PDEs with constant coefficients, we can sometimes solve for the eigenvalue-dependent solution on each domain and reduce the problem to a finite-dimensional NEP relating the boundary values~\cite{spiders_embree}. This can be done for~\cref{damped_beam_problem_def} when all the coefficients are constant but cannot generally be done for variable coefficients.} The regions in the complex plane where the computed eigenvalues are reliable depends non-trivially on the coefficient.

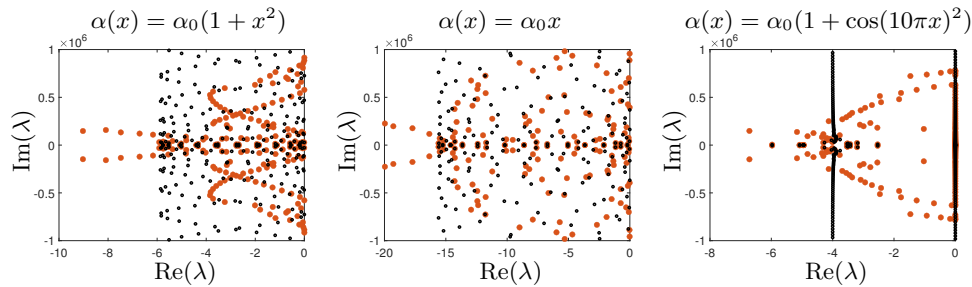
\begin{figure}
\centering
\begin{minipage}[b]{1\textwidth}
\vspace{4mm}
\centering
\begin{overpic}[width=0.32\textwidth,trim={0mm 0mm 0mm 0mm},clip]{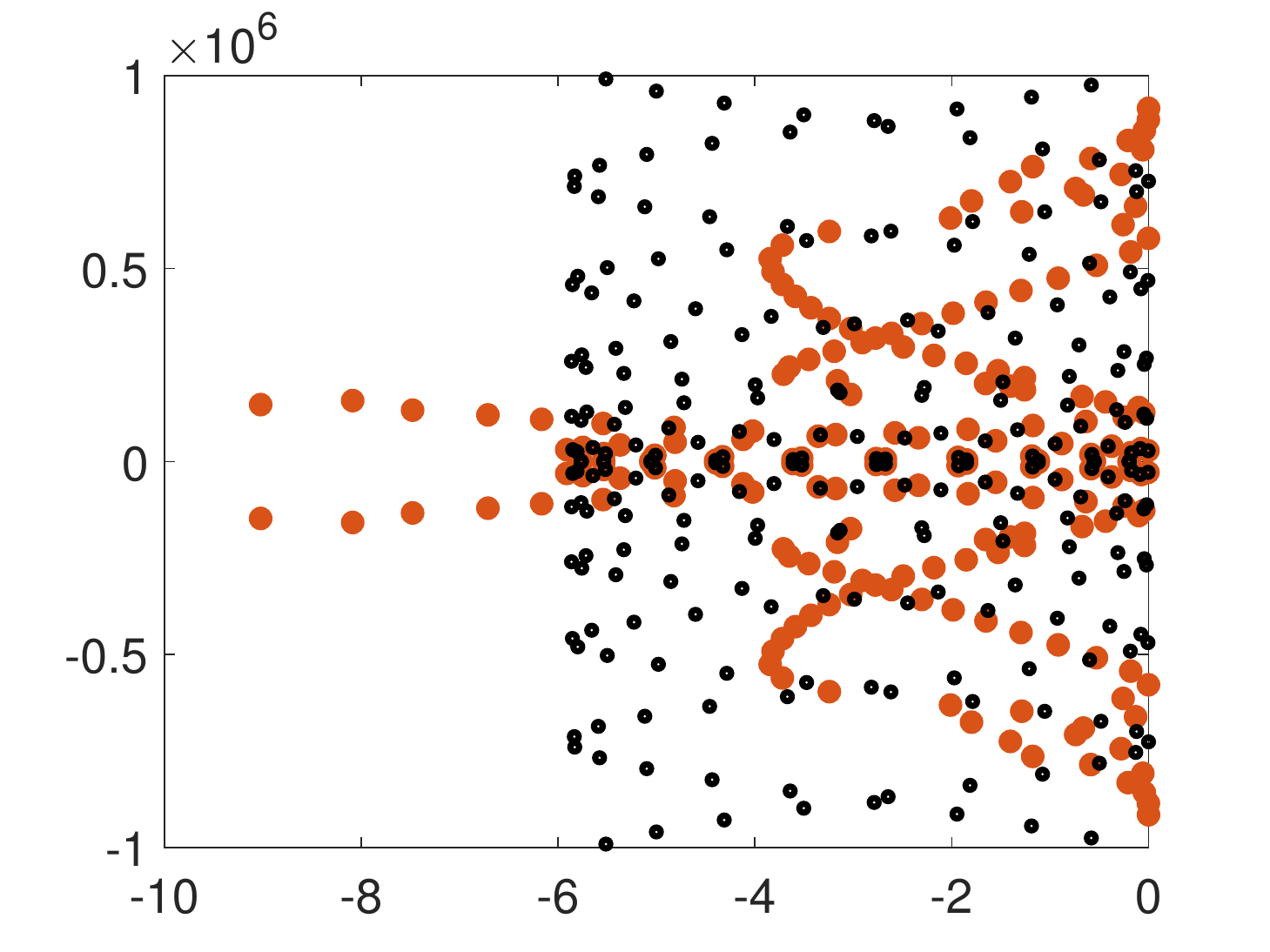}
\put (43,-4) {\small $\mathrm{Re}(\lambda)$}
\put (-2,30) {\small\rotatebox{90}{$\mathrm{Im}(\lambda)$}}
\put (23,76) {\small $\alpha(x)=\alpha_0(1+x^2)$}
\end{overpic}
\begin{overpic}[width=0.32\textwidth,trim={0mm 0mm 0mm 0mm},clip]{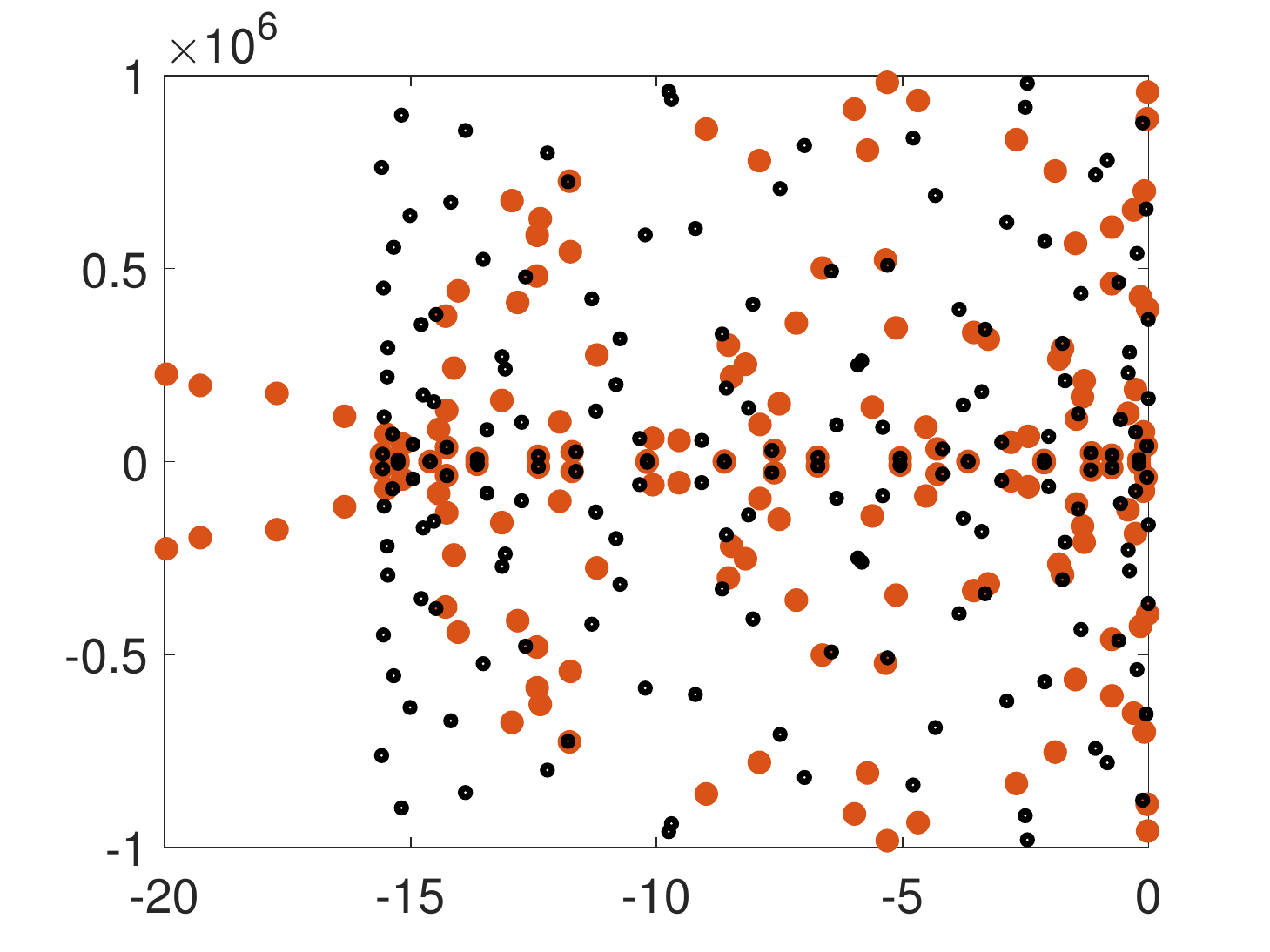}
\put (43,-4) {\small $\mathrm{Re}(\lambda)$}
\put (-2,30) {\small\rotatebox{90}{$\mathrm{Im}(\lambda)$}}
\put (30,76) {\small $\alpha(x)=\alpha_0x$}
\end{overpic}
\begin{overpic}[width=0.32\textwidth,trim={0mm 0mm 0mm 0mm},clip]{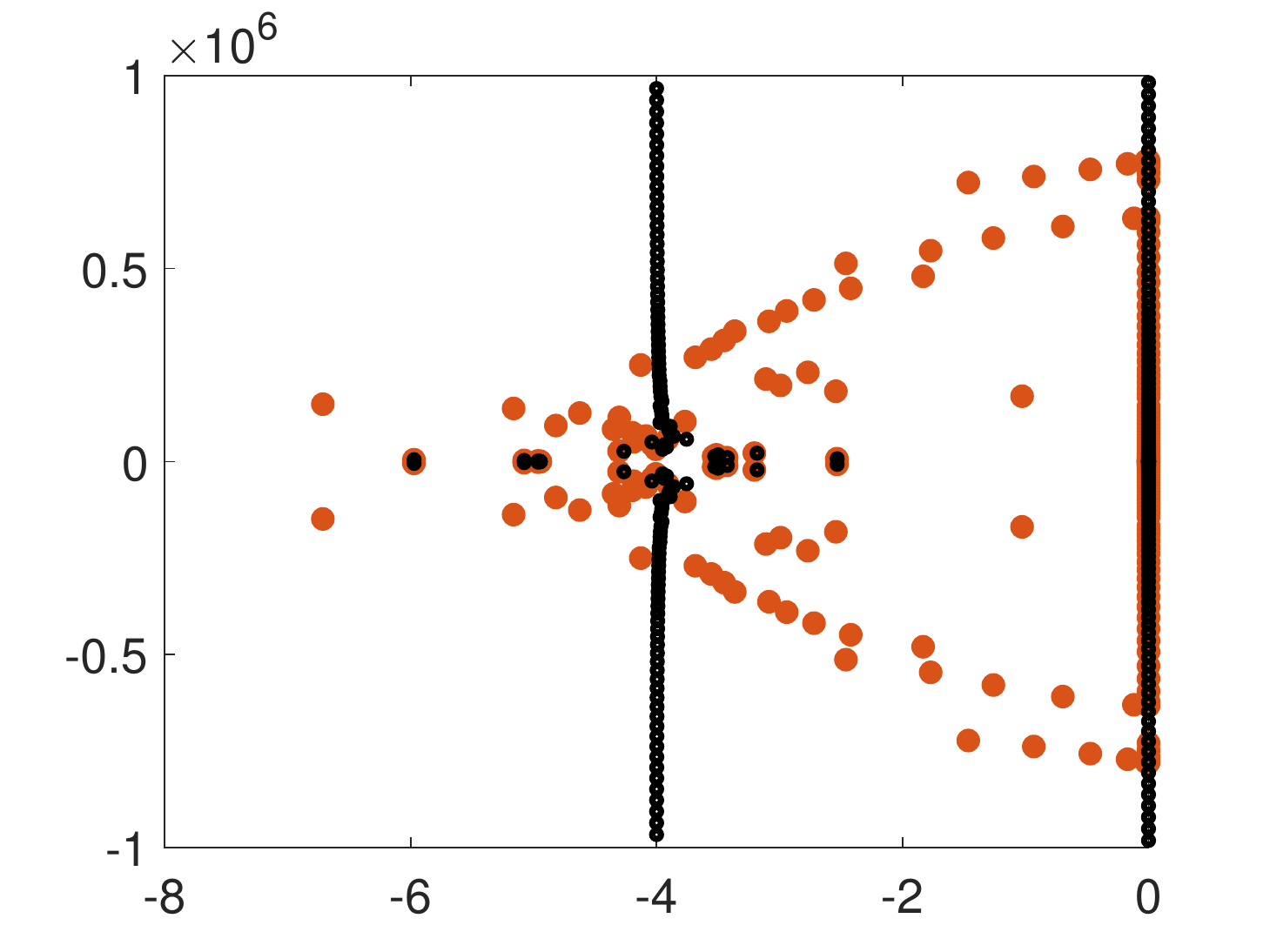}
\put (43,-4) {\small $\mathrm{Re}(\lambda)$}
\put (-2,30) {\small\rotatebox{90}{$\mathrm{Im}(\lambda)$}}
\put (5,76) {\small $\alpha(x)=\alpha_0(1+\cos(10\pi x)^2)$}
\end{overpic}
\end{minipage}\vspace{-3mm}
\caption{Same as~\cref{fig:damped_beam_1} but with the shown variable coefficient $\alpha(x)$ replacing $\alpha_0$ in~\cref{damped_beam_problem_def}.}
\label{fig:damped_beam_3}
\end{figure}

\subsection{Example 5: A loaded string}\label{sec:loaded_string}
Next, we look at another NEP in the NLEVP collection named \texttt{loaded\_string}. The NEP is given by
\begin{equation}\setlength\abovedisplayskip{6pt}\setlength\belowdisplayskip{6pt}
-\frac{d^2u}{d x^2}=\lambda u,\qquad u(0)=0,\quad\frac{d u}{d x}(1)+\frac{\lambda \kappa M}{\lambda -\kappa}u(1)=0.
\label{eq:loadedSpring}
\end{equation} 
It models the vibrations of a string with a load of mass $M$ attached to an
elastic spring with stiffness $\kappa$~\cite{solov2006preconditioned}. We use the default parameters $M=\kappa=1$.  The eigenvalues of physical interest lie in the interval $(\kappa,\infty)\subset\mathbb{R}$ and are solutions of
\begin{equation}\setlength\abovedisplayskip{6pt}\setlength\belowdisplayskip{6pt}
\label{loaded_string_analytic}
\cos\left(\sqrt{\lambda}\right) + \frac{\sqrt{\lambda}\kappa M}{\lambda-\kappa}\sin\left(\sqrt{\lambda}\right)=0.
\end{equation}

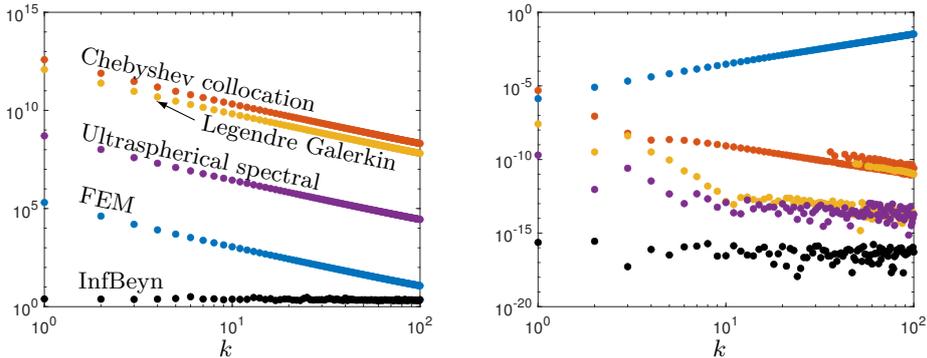
\begin{figure}
\centering
\begin{minipage}[b]{1\textwidth}
\centering
\begin{overpic}[width=0.49\textwidth,trim={0mm 0mm 0mm 0mm},clip]{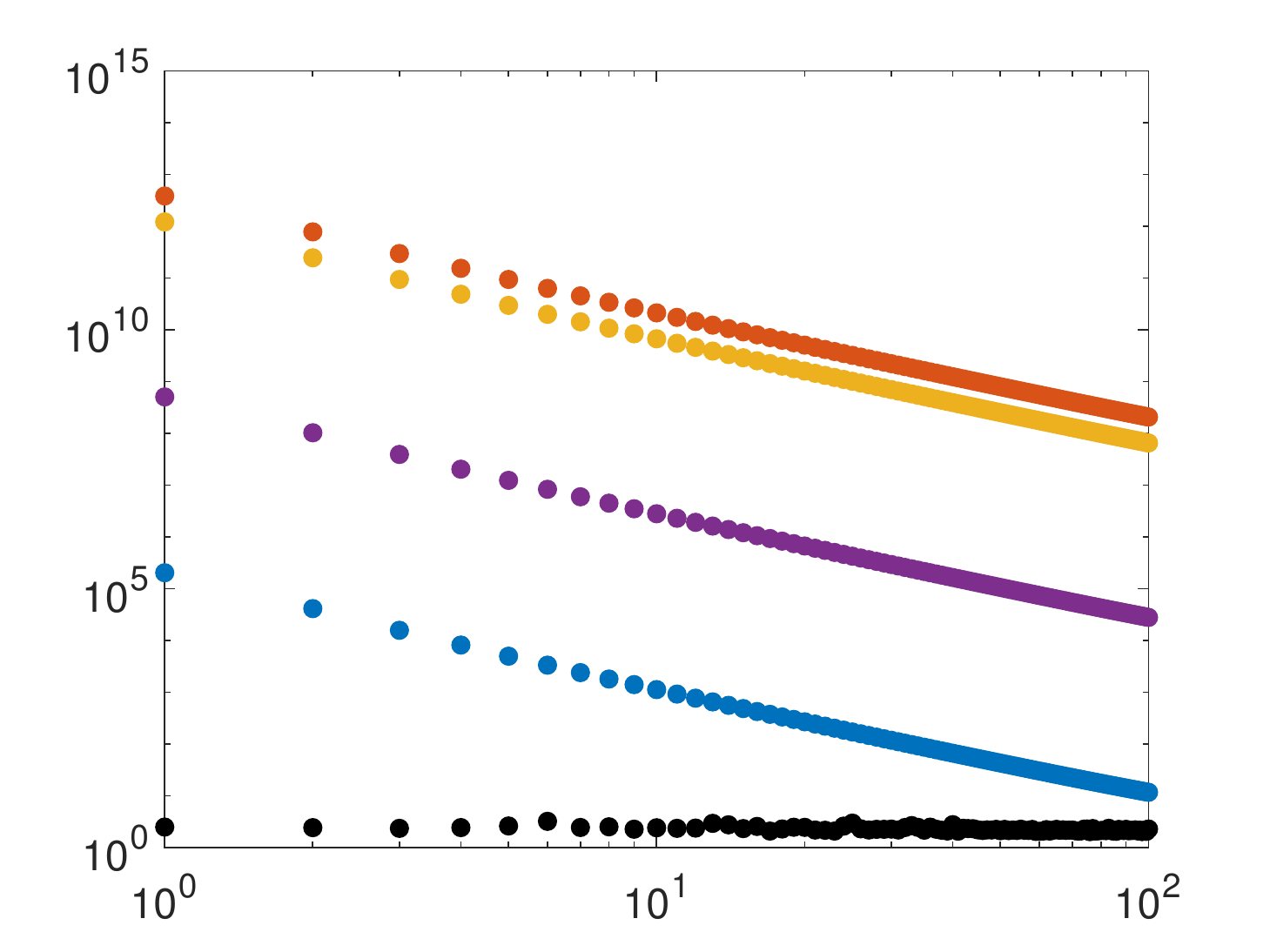}
\put (49,-2) {\small $k$}
\put (20,12) {\small InfBeyn}
\put (20,30) {\small\rotatebox{-12.5}{FEM}}
\put (20,43) {\small\rotatebox{-12}{Ultraspherical spectral}}
\put (20,59) {\small\rotatebox{-12}{Chebyshev collocation}}
\put(44,47){\vector(-2,1){8}}
\put (45,45) {\small\rotatebox{-12}{Legendre Galerkin}}
\end{overpic}
\begin{overpic}[width=0.49\textwidth,trim={0mm 0mm 0mm 0mm},clip]{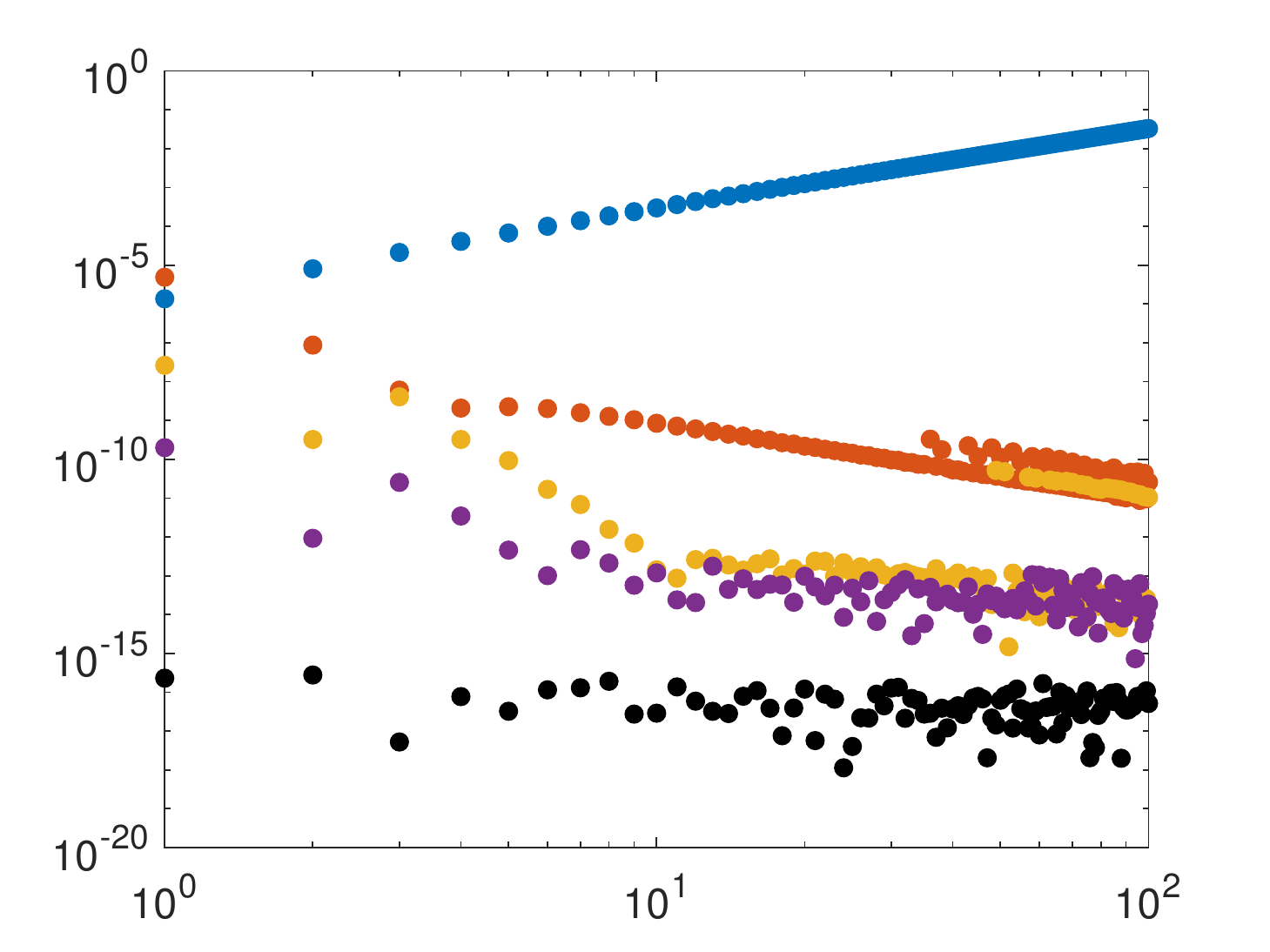}
\put (49,-2) {\small $k$}
\end{overpic}
\end{minipage}
\caption{A comparison of four discretization methods and InfBeyn for~\cref{eq:loadedSpring}. Left: A discretization method can construct a NEP that has severely ill-conditioned eigenvalues. Right: The relative accuracy of the computed eigenvalues.}
\label{fig:loaded_string}
\end{figure}

Since the infinite-dimensional NEP has a Rayleigh quotient that increases monotonically with the spectral parameter, one can show that a linear FEM constructs a discrete NEP whose eigenvalues converge to the spectrum of~\cref{eq:loadedSpring} without spectral pollution or missing eigenvalues~\cite{solov1997finite}. However, discretization can still introduce severe ill-conditioning, potentially but not necessarily, causing inaccurate computed eigenvalues in floating-point arithmetic. 

We consider four methods of discretization and compare them with InfBeyn. The first discretization uses finite elements as proposed in~\cite{solov2006preconditioned}, the second uses a Chebyshev collocation method~\cite{driscoll2016rectangular}, the third uses a standard Galerkin method using Legendre polynomials, and the fourth uses the ultraspherical spectral method~\cite{olver2013fast}. To calculate the accuracy of each discretization method, we compute solutions to~\cref{loaded_string_analytic} using Newton's method with initial guesses provided by the asymptotic formula $\lambda_k\sim (k-1/2)^2\pi^2$ as $k\rightarrow\infty$. This asymptotic formula also guides us in selecting the centers of the contours for InfBeyn. \cref{fig:loaded_string} (left) shows the relative condition numbers of the first 100 eigenvalues of the resulting discrete NEPs (see~\cite[Thm.~2.20]{guttel2017nonlinear} for the condition number formula) for $500\times 500$ discretizations. We also show the corresponding condition numbers for InfBeyn. While the condition numbers are interesting, they give little insight into the final accuracy of the computed eigenvalues (see \cref{fig:loaded_string} (right)), which may be because floating-point rounding errors are causing highly structured perturbations.  Moreover, the relative accuracy of the computed eigenvalues is also due to how fast the eigenvalues of the discretization converge. 

\subsection{Example 6: Planar waveguide}\label{sec:planar_waveguide}
For our last example, we take a look at the NEP that is called \texttt{planar\_waveguide} in the NLEVP collection. This NEP describes the propagation properties of electromagnetic waves in multilayered media, characterized by a refractive index $\eta$ that varies in $x$-direction~\cite{stowell2010variational}. The original problem is a linear eigenvalue problem on the unbounded domain $\mathbb{R}$ that has both discrete and essential spectra~\cite{marcuse2013theory}. 

More precisely, consider a material that consists of $J+1$ layers described by refractive indices $\eta_0,\ldots,\eta_J$ and the positions of the interfaces $x_1=0<x_2<\cdots<x_J=L$, so that $\eta(x)=\eta_0$ if $x<x_1$, $\eta(x)=\eta_j$ if $x_j<x<x_{j+1}$ and $\eta(x)=\eta_J$ if $x>x_J$. The truncated domain is $[x_1,x_J]=[0,L]$. For a frequency $k$, we define $\delta_{\pm}={k^2(\eta_0^2\pm\eta_J^2)}/{2}$. The NEP is given by
\begin{equation}\setlength\abovedisplayskip{6pt}\setlength\belowdisplayskip{6pt}
\begin{gathered}
\frac{d^2\phi}{d x^2}(x)+k^2[\eta^2(x)-\mu(\lambda)]\phi(x)=0,\quad \mu(\lambda)=\frac{\delta_+}{k^2}+\frac{\delta_-}{8k^2\lambda^2}+\frac{\lambda^2}{k^2},\\
\frac{d\phi}{d x}(0)+\left(\frac{\delta_-}{2\lambda}-\lambda\right)\phi(0)=0,\quad \frac{d\phi}{d x}(L)+\left(\frac{\delta_-}{2\lambda}+\lambda\right)\phi(L)=0.
\end{gathered}
\label{eq:WaveGuide}
\end{equation}
We take the default parameters $J=5$, $\eta_0 = 1.5,\eta_1 = 1.66,\eta_2 = 1.6,\eta_3 = 1.53,\eta_4 = 1.66,\eta_5 = 1.0$, $x_2=0.5,x_3=1.0,x_4=1.5,x_5=2.0$ and $k=2\pi/0.6328$.

\begin{figure}
\centering
\begin{minipage}[b]{1\textwidth}
\centering
\begin{overpic}[width=0.49\textwidth,trim={0mm 0mm 0mm 0mm},clip]{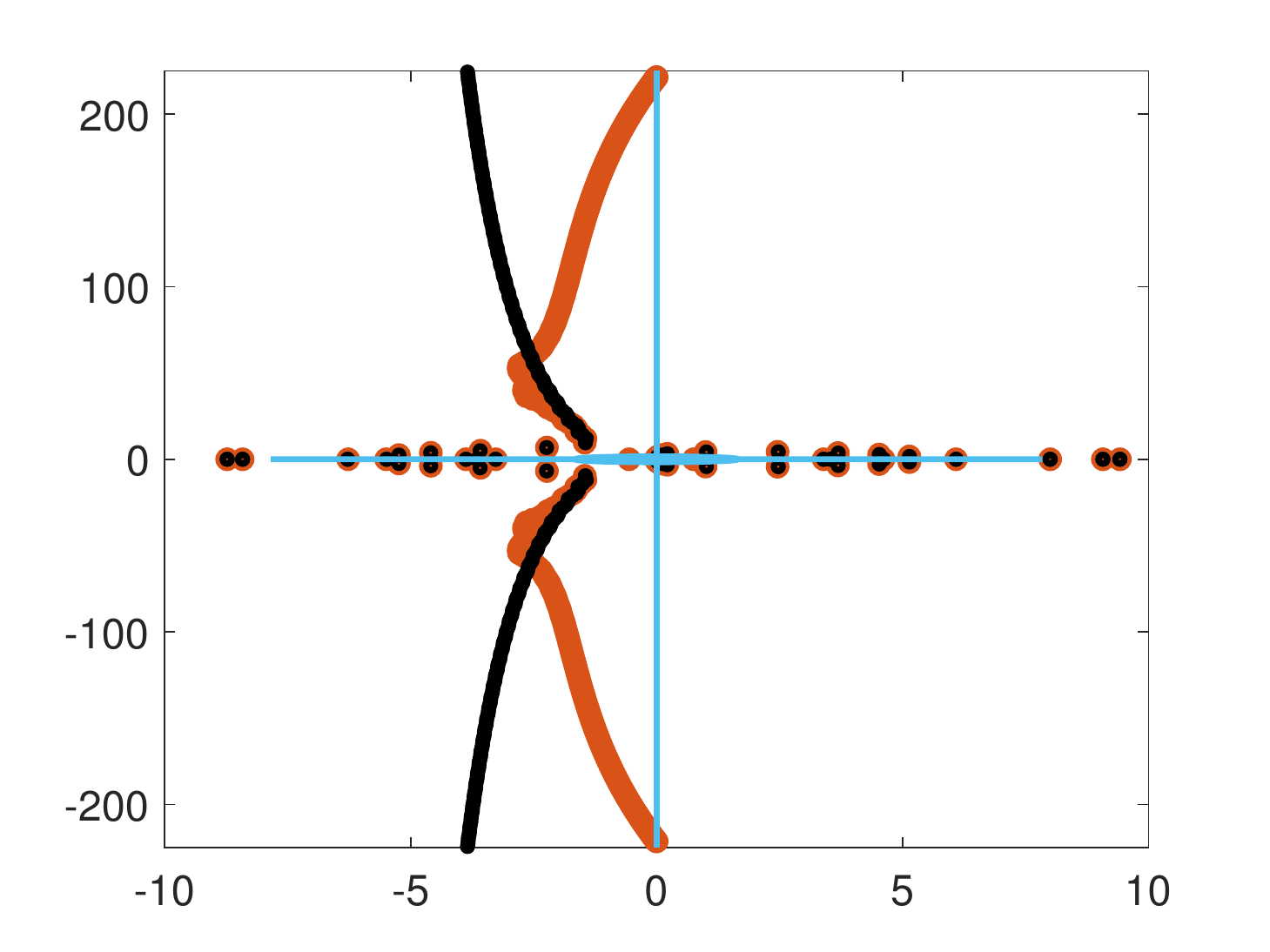}
\put (45,-2) {\small $\mathrm{Re}(\lambda)$}
\put (2,33) {\small\rotatebox{90}{$\mathrm{Im}(\lambda)$}}
\end{overpic}
\begin{overpic}[width=0.49\textwidth,trim={0mm 0mm 0mm 0mm},clip]{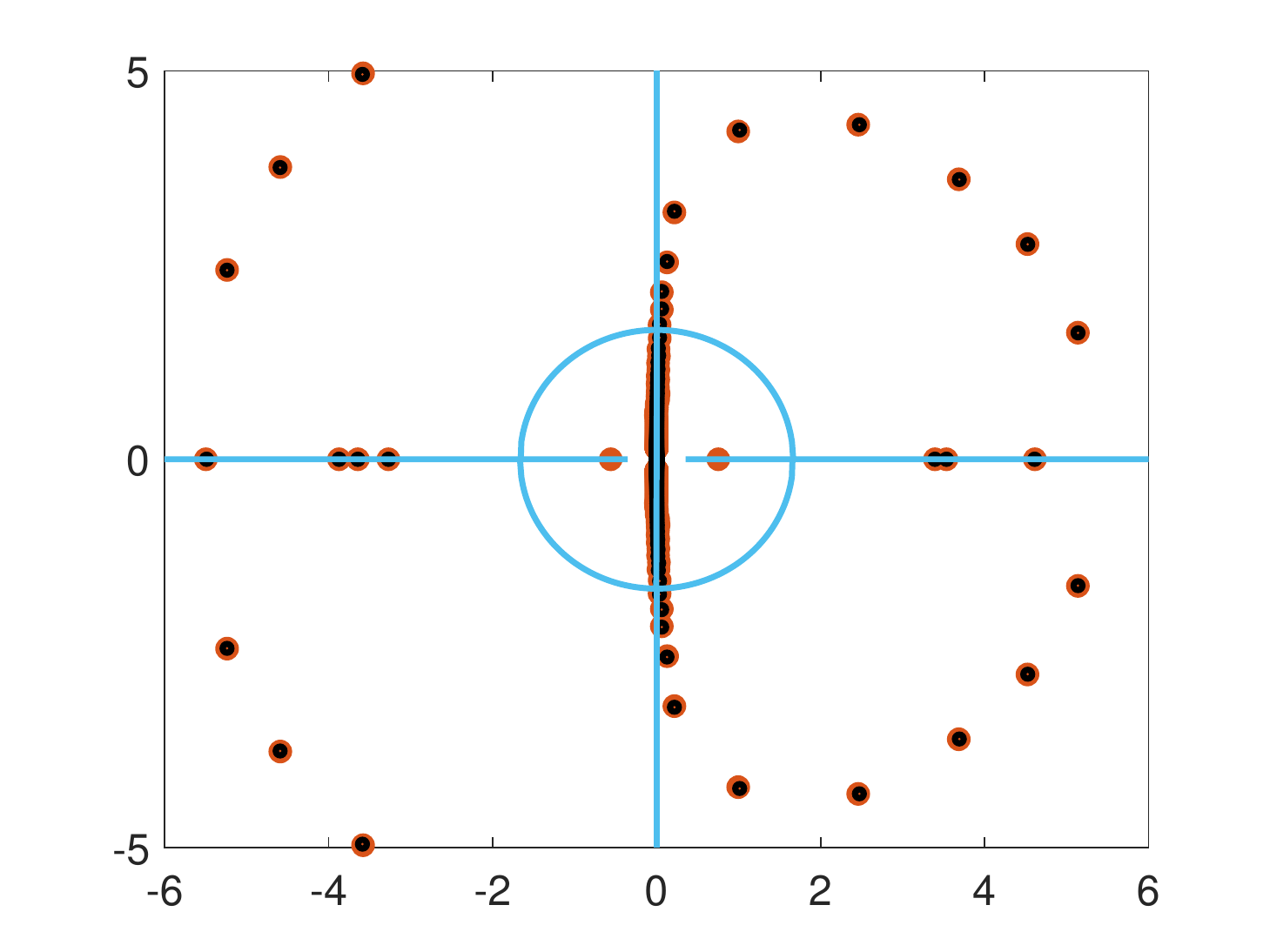}
\put (45,-2) {\small $\mathrm{Re}(\lambda)$}
\put (2,33) {\small\rotatebox{90}{$\mathrm{Im}(\lambda)$}}
\end{overpic}
\end{minipage}\vspace{-3mm}
\caption{Spectra of the planar waveguide problem in the $\lambda$ plane. Left: The eigenvalues of the discretized planar waveguide problem are shown in red. The eigenvalues computed by InfBeyn are shown as black circles and verified using infinite-dimensional residuals. Right: A magnified region near $\lambda=0$. The discretized problem has spurious modes, and several branches of the modes collapse onto the essential spectrum of the underlying problem posed on $\mathbb{R}$ (shown in light blue).}
\label{fig:waveguide1}
\end{figure}

\begin{figure}
\centering
\begin{minipage}[b]{1\textwidth}
\centering
\begin{overpic}[width=0.49\textwidth,trim={0mm 0mm 0mm 0mm},clip]{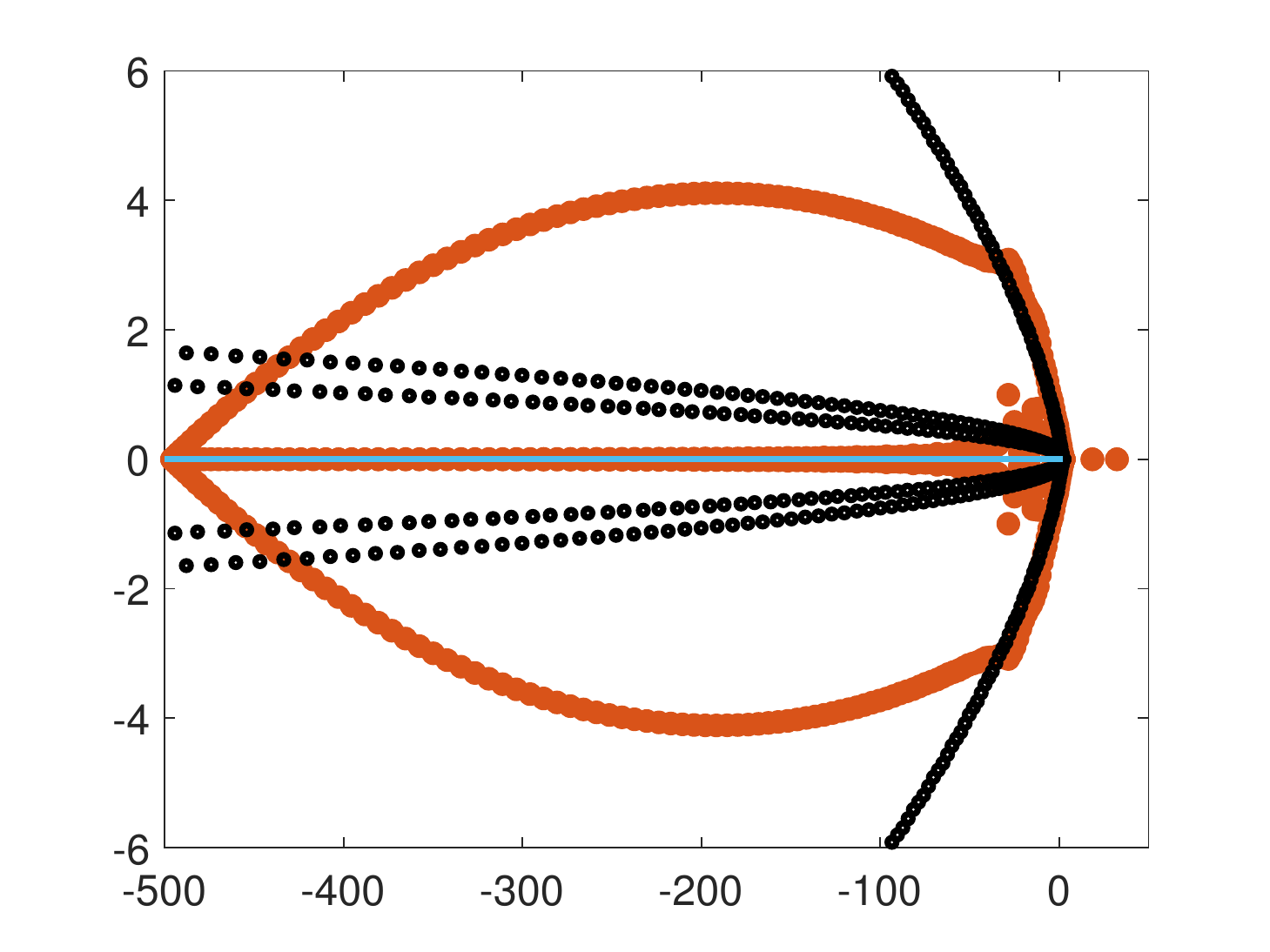}
\put (45,-2) {\small $\mathrm{Re}(\mu)$}
\put (2,33) {\small\rotatebox{90}{$\mathrm{Im}(\mu)$}}
\put (25,10) {\small\rotatebox{0}{spurious modes}}
\linethickness{1pt}
\put (63,13){\color{blue}\vector(1,1){24}}
\put (15,65) {\small\rotatebox{0}{collapse onto ess. spec.}}
\put (20,63){\color{blue}\vector(1,-2){12}}
\end{overpic}
\begin{overpic}[width=0.49\textwidth,trim={0mm 0mm 0mm 0mm},clip]{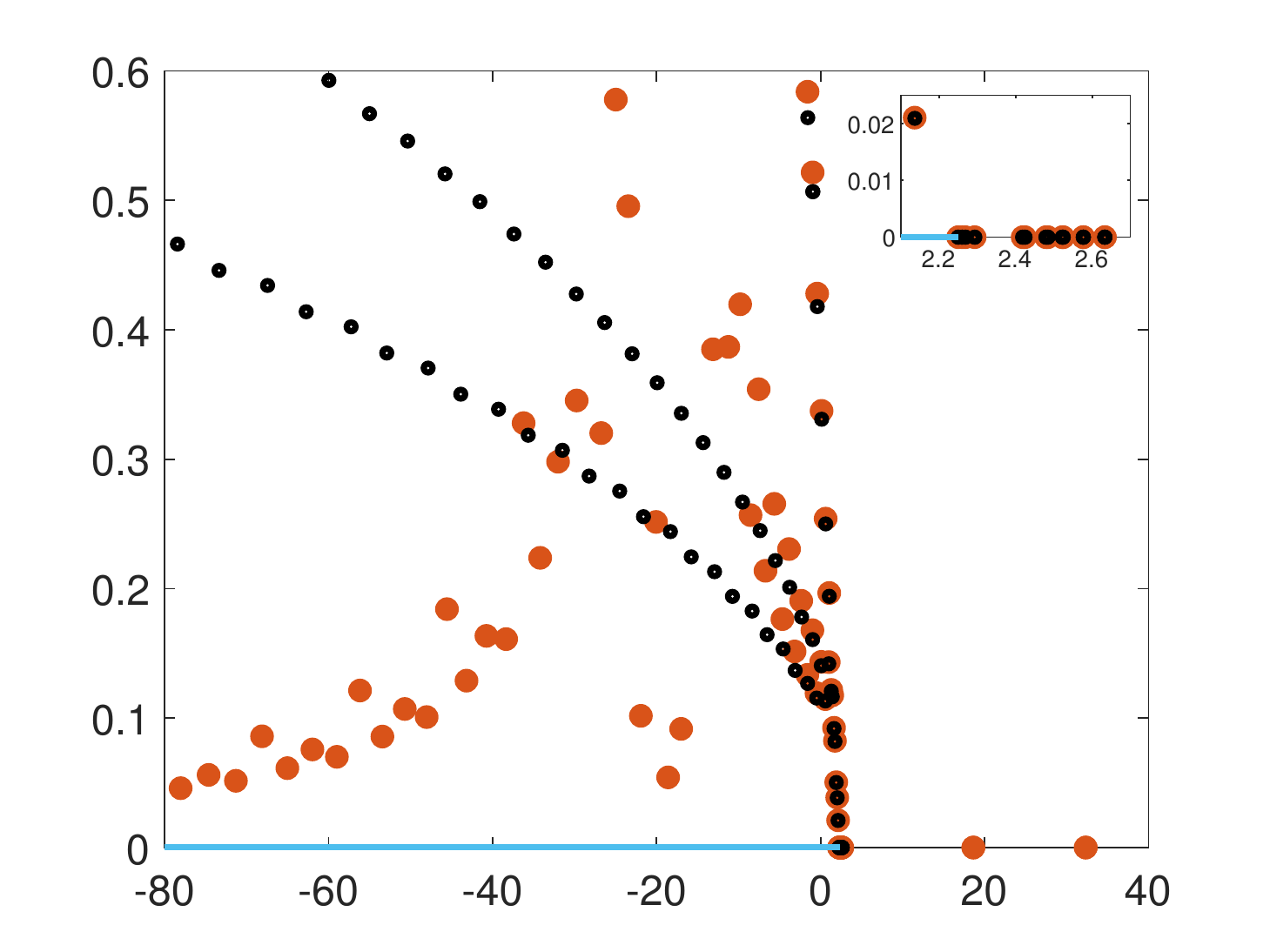}
\put (45,-2) {\small $\mathrm{Re}(\mu)$}
\put (1,33) {\small\rotatebox{90}{$\mathrm{Im}(\mu)$}}
\end{overpic}
\end{minipage}\vspace{-3mm}
\caption{Same as \cref{fig:waveguide1} but in the $\mu(\lambda)$ plane. The essential spectrum of the underlying problem posed on $\mathbb{R}$ is the semi-infinite interval $\mu\in(-\infty,\max\{\eta_0^2,\eta_J^2\}]$.}
\label{fig:waveguide2}
\end{figure}
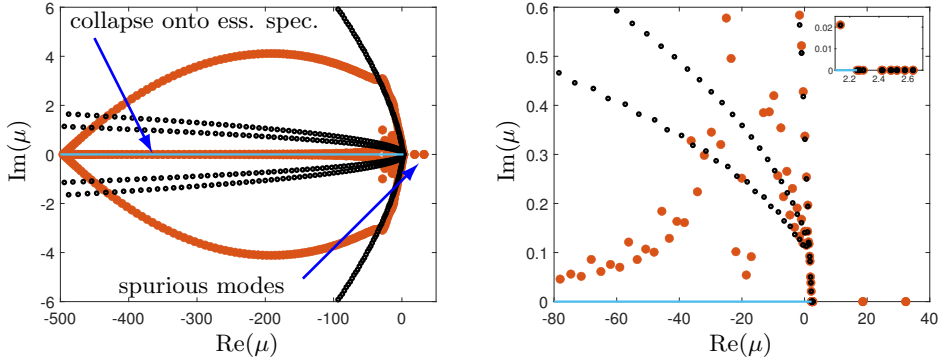

We discretize~\cref{eq:WaveGuide} using piecewise linear finite elements.  \cref{fig:waveguide1} shows the eigenvalues computed using finite elements and a discretization size of $n=129$ (NLEVP's default value), as well as those computed using InfBeyn. We also show the essential spectrum of the linear problem on the unbounded domain, which is given by $\lambda$ such that $\mu(\lambda)\in (-\infty,\max\{\eta_0^2,\eta_J^2\}]$.\footnote{To see this, one considers the altered linear problem on $\mathbb{R}$ with $\eta_1=\cdots=\eta_J$, where the essential spectrum is $\mu(\lambda)\in (-\infty,\max\{\eta_0^2,\eta_J^2\}]$. One then shows that the resolvents of the altered problem differ from the original by a compact operator~\cite[p.~244]{kato2013perturbation}.} The majority of the eigenvalues of the discretized NEP are inaccurate, and there are also spurious guided modes. Many modes have collapsed onto the essential spectrum of the linear eigenvalue problem posed on $\mathbb{R}$. This issue is common with discretized problems of this type~\cite[Fig.~3]{gopalakrishnan2022computing}. In contrast, $\lambda=0$ is the only point in the essential spectrum of the NEP. In the right panel of~\cref{fig:waveguide1}, we see the accumulation of the discrete spectrum at $\lambda=0$, which also causes issues for the discretized problem. \cref{fig:waveguide2} shows a similar plot in the $\mu(\lambda)$ plane.

As a final experiment, we show a region near the essential spectrum at $\lambda=0$ (see~\cref{fig:waveguide3} (left)). We see clustering of the spectrum at this point, computed using InfBeyn. On the right, we have shown convergence to three eigenvalues as the number of quadrature nodes increases, where a circular contour around several eigenvalues is used. This demonstrates the effectiveness of infinite-dimensional contour methods such as InfBeyn for problems with accumulating spectra, owing to their locality, parallelizability, and rapid convergence.

\begin{figure}
\centering
\begin{minipage}[b]{1\textwidth}
\centering
\begin{overpic}[width=0.49\textwidth,trim={0mm 0mm 0mm 0mm},clip]{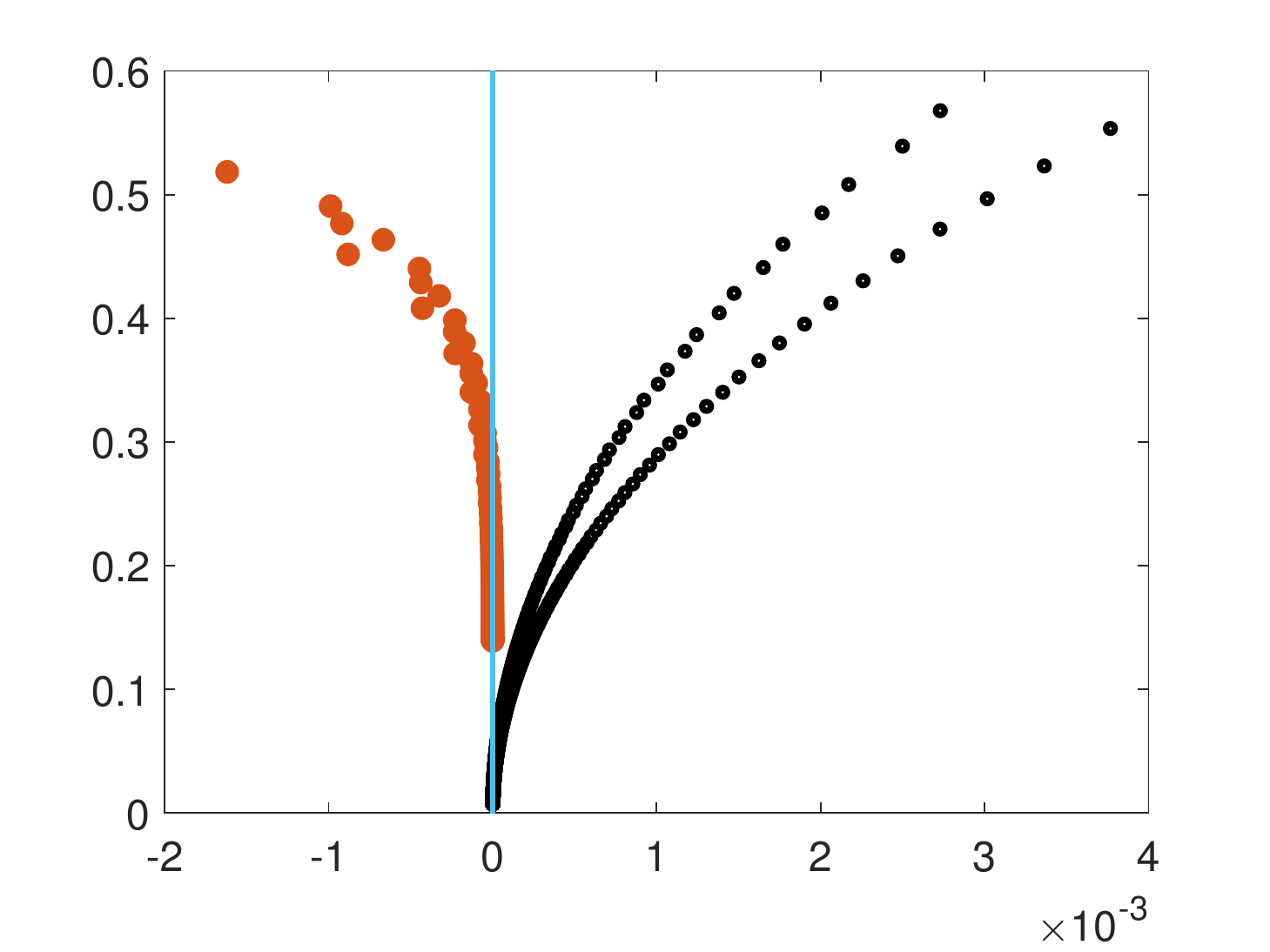}
\put (45,-2) {\small $\mathrm{Re}(\lambda)$}
\put (0,35) {\small\rotatebox{90}{$\mathrm{Im}(\lambda)$}}
\end{overpic}
\begin{overpic}[width=0.49\textwidth,trim={0mm 0mm 0mm 0mm},clip]{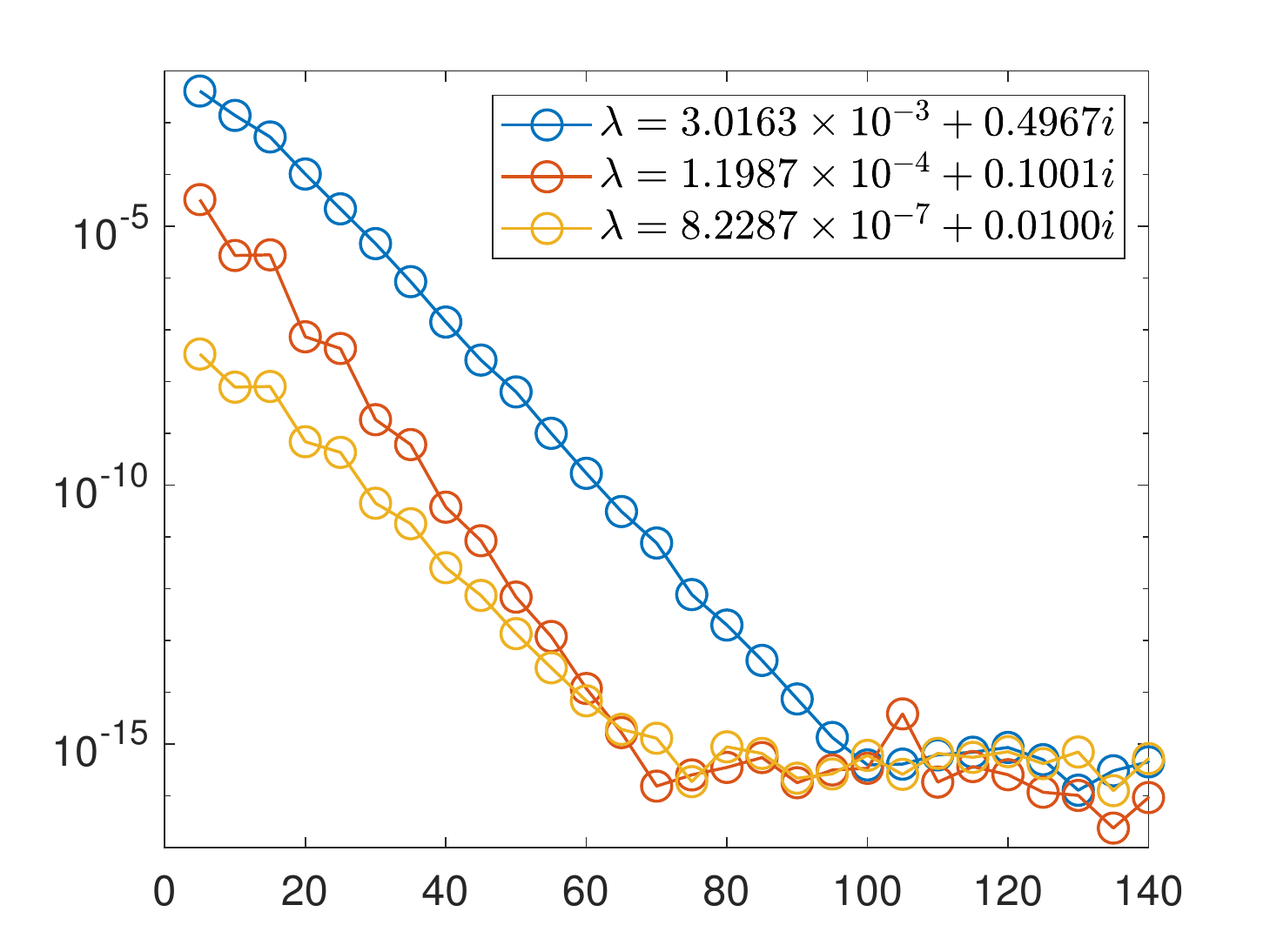}
\put (49,-2) {\small $\ell$}
\put (35,73) {\small $|\lambda-\lambda^{(\ell)}|/\sqrt{|\mu(\lambda)|}$}
\end{overpic}
\end{minipage}\vspace{-3mm}
\caption{Left: Eigenvalues accumulating at the essential spectral point $\lambda=0$. Right: Relative accuracy of the computed eigenvalues $\lambda^{(\ell)}$ for $\ell$ quadrature points in~\cref{eqn:approx_Beyn}. We have relativized by $\sqrt{|\mu(\lambda)|}$ to capture the different rational scalings of $\lambda$ in the NEP.}
\label{fig:waveguide3}
\end{figure}

\section{Conclusion}\label{sec:Conclusion}
As we show with six examples, discretizing infinite-dimensional NEPs can modify, destabilize, or destroy eigenvalues. By delaying discretization, we proposed practical algorithms for computing spectra and pseudospectra of infinite-dimensional NEPs that avoid these issues, and we proved their stability and convergence. We hope that the paper is the beginning of infinite-dimensional NEP solvers. For example, while InfBeyn deals with regions where the spectrum is discrete and our method for computing pseudospectra can deal with essential spectrum, we imagine that there is an infinite-dimensional algorithm for computing essential spectra of NEPs directly. We are also intrigued by the consequences for down-the-line applications such as reduced order models~\cite{mehrmann2016eigenvalue,grabner2016numerical,brennan2020contour} and developing structure-preserving infinite-dimensional solvers. Infinite-dimensional NEP solvers offer the potential for more robust calculations of physically-relevant spectra in challenging applications. 

\renewcommand{\baselinestretch}{0.89}
\bibliography{references}

\begin{thebibliography}{10}

\bibitem{spiders_embree}
{\sc J.~P. Baker, M.~Embree, and J.~N. Green}, {\em Networks of vibrating
  strings: A model nonlinear eigenvalue problem}, in preparation,  (2023).

\bibitem{beer1993topologies}
{\sc G.~Beer}, {\em Topologies on closed and closed convex sets}, vol.~268,
  Springer Science \& Business Media, 1993.

\bibitem{betcke2013nlevp}
{\sc T.~Betcke, N.~J. Higham, V.~Mehrmann, C.~Schr{\"o}der, and F.~Tisseur},
  {\em {NLEVP}: {A} collection of nonlinear eigenvalue problems}, ACM Trans.
  Math. Soft., 39 (2013), pp.~1--28.

\bibitem{betcke2005reviving}
{\sc T.~Betcke and L.~N. Trefethen}, {\em Reviving the method of particular
  solutions}, SIAM Rev., 47 (2005), pp.~469--491.

\bibitem{beyn2012integral}
{\sc W.-J. Beyn}, {\em An integral method for solving nonlinear eigenvalue
  problems}, Lin. Alg. Appl., 436 (2012), pp.~3839--3863.

\bibitem{bindel2015localization}
{\sc D.~Bindel and A.~Hood}, {\em Localization theorems for nonlinear
  eigenvalue problems}, SIAM Rev., 57 (2015), pp.~585--607.

\bibitem{bindel2006theory}
{\sc D.~Bindel and M.~Zworski}, {\em Theory and computation of resonances in 1d
  scattering},  (2006).

\bibitem{bogli2020essential}
{\sc S.~B{\"o}gli, M.~Marletta, and C.~Tretter}, {\em The essential numerical
  range for unbounded linear operators}, J. Func. Anal., 279 (2020), p.~108509.

\bibitem{botchev2009svd}
{\sc M.~Botchev, G.~Sleijpen, and A.~Sopaheluwakan}, {\em An {SVD}-approach to
  {J}acobi--{D}avidson solution of nonlinear {H}elmholtz eigenvalue problems},
  Lin. Alg. Appl., 431 (2009), pp.~427--440.

\bibitem{boulle2022a}
{\sc N.~Boulle and A.~Townsend}, {\em A generalization of the randomized
  singular value decomposition}, in International Conference on Learning
  Representations, 2022.

\bibitem{boulle2022learning}
{\sc N.~Boull{\'e} and A.~Townsend}, {\em Learning elliptic partial
  differential equations with randomized linear algebra}, Foundations of
  Computational Mathematics,  (2022), pp.~1--31.

\bibitem{brennan2020contour}
{\sc M.~C. Brennan, M.~Embree, and S.~Gugercin}, {\em Contour integral methods
  for nonlinear eigenvalue problems: {A} systems theoretic approach}, arXiv
  preprint arXiv:2012.14979,  (2020).

\bibitem{chaitin2006analysis}
{\sc F.~Chaitin-Chatelin and M.~B. van Gijzen}, {\em Analysis of parameterized
  quadratic eigenvalue problems in computational acoustics with homotopic
  deviation theory}, Numer. Lin. Alg. Appl., 13 (2006), pp.~487--512.

\bibitem{chen2016perturbation}
{\sc Y.~M. Chen, X.~S. Chen, and W.~Li}, {\em On perturbation bounds for
  orthogonal projections}, Numer. Algor., 73 (2016), pp.~433--444.

\bibitem{colbrook2020foundations}
{\sc M.~Colbrook}, {\em The foundations of infinite-dimensional spectral
  computations}, PhD thesis, University of Cambridge, 2020.

\bibitem{GithubRepo}
{\sc M.~J. Colbrook}, {\em inf{NEP}}.
\newblock \url{https://github.com/MColbrook/infNEP}, 2023.

\bibitem{colbrook2022foundations}
{\sc M.~J. Colbrook and A.~C. Hansen}, {\em The foundations of spectral
  computations via the solvability complexity index hierarchy}, J. Euro. Math.
  Soc.,  (2022).

\bibitem{colbrook2019compute}
{\sc M.~J. Colbrook, B.~Roman, and A.~C. Hansen}, {\em How to compute spectra
  with error control}, Phys. Rev. Lett., 122 (2019), p.~250201.

\bibitem{colbrook2021rigorous}
{\sc M.~J. Colbrook and A.~Townsend}, {\em Rigorous data-driven computation of
  spectral properties of {K}oopman operators for dynamical systems}, arXiv
  preprint arXiv:2111.14889,  (2021).

\bibitem{collar1987matrices}
{\sc A.~R. Collar and A.~Simpson}, {\em Matrices and engineering dynamics},
  Ellis Horwood, 1987.

\bibitem{di1969completeness}
{\sc R.~Di~Prima and G.~Habetler}, {\em A completeness theorem for
  non-selfadjoint eigenvalue problems in hydrodynamic stability}, Archive for
  Rational Mechanics and Analysis, 34 (1969), pp.~218--227.

\bibitem{driscoll2016rectangular}
{\sc T.~A. Driscoll and N.~Hale}, {\em Rectangular spectral collocation}, IMA
  J. Numer. Anal., 36 (2016), pp.~108--132.

\bibitem{engstrom2010complex}
{\sc C.~Engstr{\"o}m and M.~Wang}, {\em Complex dispersion relation
  calculations with the symmetric interior penalty method}, Inter. J. Numer.
  Meth. Eng., 84 (2010), pp.~849--863.

\bibitem{gopalakrishnan2022computing}
{\sc J.~Gopalakrishnan, B.~Q. Parker, and P.~Vandenberge}, {\em Computing leaky
  modes of optical fibers using a {FEAST} algorithm for polynomial
  eigenproblems}, Wave Motion, 108 (2022), p.~102826.

\bibitem{grabner2016numerical}
{\sc N.~Gr{\"a}bner, V.~Mehrmann, S.~Quraishi, C.~Schr{\"o}der, and U.~von
  Wagner}, {\em Numerical methods for parametric model reduction in the
  simulation of disk brake squeal}, ZAMM-J. Appl. Math. Mech., 96 (2016),
  pp.~1388--1405.

\bibitem{green2006pseudospectra}
{\sc K.~Green and T.~Wagenknecht}, {\em Pseudospectra and delay differential
  equations}, J. Comput. Appl. Math., 196 (2006), pp.~567--578.

\bibitem{grigorieff1973approximation}
{\sc R.~D. Grigorieff and H.~Jeggle}, {\em Approximation von
  {E}igenwertproblemen bei nichtlinearer {P}arameterabh{\"a}ngigkeit},
  Manuscripta Mathematica, 10 (1973), pp.~245--271.

\bibitem{grosch1978continuous}
{\sc C.~E. Grosch and H.~Salwen}, {\em The continuous spectrum of the
  {Orr--Sommerfeld} equation. {P}art 1. {T}he spectrum and the eigenfunctions},
  J. Fluid Mech., 87 (1978), pp.~33--54.

\bibitem{guttel2017nonlinear}
{\sc S.~G{\"u}ttel and F.~Tisseur}, {\em The nonlinear eigenvalue problem},
  Acta Numer., 26 (2017), pp.~1--94.

\bibitem{harari1996recent}
{\sc I.~Harari, K.~Grosh, T.~Hughes, M.~Malhotra, P.~Pinsky, J.~Stewart, and
  L.~Thompson}, {\em Recent developments in finite element methods for
  structural acoustics}, Arch. Comput. Meth. Eng., 3 (1996), pp.~131--309.

\bibitem{higham1998structured}
{\sc D.~J. Higham and N.~J. Higham}, {\em Structured backward error and
  condition of generalized eigenvalue problems}, SIAM J. Matrix Anal. Appl., 20
  (1998), pp.~493--512.

\bibitem{higham2008scaling}
{\sc N.~J. Higham, D.~S. Mackey, F.~Tisseur, and S.~D. Garvey}, {\em Scaling,
  sensitivity and stability in the numerical solution of quadratic eigenvalue
  problems}, Inter. J. Numer. Meth. Eng., 73 (2008), pp.~344--360.

\bibitem{higham2002more}
{\sc N.~J. Higham and F.~Tisseur}, {\em More on pseudospectra for polynomial
  eigenvalue problems and applications in control theory}, Lin. Alg. Appl., 351
  (2002), pp.~435--453.

\bibitem{horning2020feast}
{\sc A.~Horning and A.~Townsend}, {\em Feast for differential eigenvalue
  problems}, SIAM J. Numer. Anal., 58 (2020), pp.~1239--1262.

\bibitem{jarlebring2008spectrum}
{\sc E.~Jarlebring}, {\em The spectrum of delay-differential equations:
  {N}umerical methods, stability and perturbation}, PhD thesis, 2008.

\bibitem{jeggle1977discrete}
{\sc H.~Jeggle and W.~Wendland}, {\em On the discrete approximation of
  eigenvalue problems with holomorphic parameter dependence}, Proc. Roy. Soc.
  Edin. Sec. A: Math., 78 (1977), pp.~1--29.

\bibitem{kato1958perturbation}
{\sc T.~Kato}, {\em Perturbation theory for nullity, deficiency and other
  quantities of linear operators}, J. d’Analyse Math., 6 (1958),
  pp.~261--322.

\bibitem{kato2013perturbation}
{\sc T.~Kato}, {\em Perturbation Theory for Linear Operators}, vol.~132,
  Springer Science \& Business Media, second~ed., 1976.

\bibitem{keldysh1951characteristic}
{\sc M.~V. Keldysh}, {\em On the characteristic values and characteristic
  functions of certain classes of non-self-adjoint equations}, in Dokl. Akad.
  Nauk SSSR, vol.~77, 1951, pp.~11--14.

\bibitem{keldysh1971completeness}
\leavevmode\vrule height 2pt depth -1.6pt width 23pt, {\em On the completeness
  of the eigenfunctions of some classes of non-selfadjoint linear operators},
  Uspekhi matematicheskikh nauk, 26 (1971), pp.~15--41.

\bibitem{kukelova2008polynomial}
{\sc Z.~Kukelova, M.~Bujnak, and T.~Pajdla}, {\em Polynomial eigenvalue
  solutions to the 5-pt and 6-pt relative pose problems.}, in BMVC, vol.~2,
  2008, p.~2008.

\bibitem{lancaster2002lambda}
{\sc P.~Lancaster}, {\em Lambda-matrices and vibrating systems}, Courier
  Corporation, 2002.

\bibitem{liao2010nonlinear}
{\sc B.-S. Liao, Z.~Bai, L.-Q. Lee, and K.~Ko}, {\em Nonlinear
  {R}ayleigh--{R}itz iterative method for solving large scale nonlinear
  eigenvalue problems}, Taiwanese J. Math., 14 (2010), pp.~869--883.

\bibitem{marcuse2013theory}
{\sc D.~Marcuse}, {\em Theory of dielectric optical waveguides}, Elsevier,
  2013.

\bibitem{mehrmann2011nonlinear}
{\sc V.~Mehrmann and C.~Schr{\"o}der}, {\em Nonlinear eigenvalue and frequency
  response problems in industrial practice}, J. Math. Ind., 1 (2011),
  pp.~1--18.

\bibitem{mehrmann2016eigenvalue}
{\sc V.~Mehrmann and C.~Schroder}, {\em Eigenvalue analysis and model reduction
  in the treatment of disc brake squeal}, SIAM News, 49 (2016), pp.~1--3.

\bibitem{mehrmann2002polynomial}
{\sc V.~Mehrmann and D.~Watkins}, {\em Polynomial eigenvalue problems with
  {H}amiltonian structure.}, Electr. Trans. Numer. Anal., 13 (2002),
  pp.~106--118.

\bibitem{mennicken2003non}
{\sc R.~Mennicken and M.~M{\"o}ller}, {\em Non-self-adjoint boundary eigenvalue
  problems}, vol.~192, Gulf Professional Publishing, 2003.

\bibitem{michiels2006pseudospectra}
{\sc W.~Michiels, K.~Green, T.~Wagenknecht, and S.-I. Niculescu}, {\em
  Pseudospectra and stability radii for analytic matrix functions with
  application to time-delay systems}, Lin. Alg. Appl., 418 (2006),
  pp.~315--335.

\bibitem{olver2013fast}
{\sc S.~Olver and A.~Townsend}, {\em A fast and well-conditioned spectral
  method}, SIAM Rev., 55 (2013), pp.~462--489.

\bibitem{orszag1971accurate}
{\sc S.~A. Orszag}, {\em Accurate solution of the {Orr--Sommerfeld} stability
  equation}, J. Fluid Mech., 50 (1971), pp.~689--703.

\bibitem{reddy1993pseudospectra}
{\sc S.~C. Reddy, P.~J. Schmid, and D.~S. Henningson}, {\em Pseudospectra of
  the {Orr--Sommerfeld} operator}, SIAM J. Appl. Math., 53 (1993), pp.~15--47.

\bibitem{sakoda2001photonic}
{\sc K.~Sakoda, N.~Kawai, T.~Ito, A.~Chutinan, S.~Noda, T.~Mitsuyu, and
  K.~Hirao}, {\em Photonic bands of metallic systems. {I}. {P}rinciple of
  calculation and accuracy}, Phys. Rev. B, 64 (2001), p.~045116.

\bibitem{schafke1965s}
{\sc F.~W. Sch{\"a}fke and A.~Schneider}, {\em S-hermitesche
  rand-eigenwertprobleme. {I}}, Mathematische Annalen, 162 (1965), pp.~9--26.

\bibitem{schmid2002stability}
{\sc P.~J. Schmid, D.~S. Henningson, and D.~Jankowski}, {\em Stability and
  transition in shear flows. applied mathematical sciences, vol. 142}, Appl.
  Mech. Rev., 55 (2002), pp.~B57--B59.

\bibitem{solov1997finite}
{\sc S.~I. Solov'ev}, {\em The finite-element method for symmetric nonlinear
  eigenvalue problems}, Zhurnal Vychislitel'noi Matematiki i Matematicheskoi
  Fiziki, 37 (1997), pp.~1311--1318.

\bibitem{solov2006preconditioned}
{\sc S.~I. Solov’{\"e}v}, {\em Preconditioned iterative methods for a class
  of nonlinear eigenvalue problems}, Lin. Alg. Appl., 415 (2006), pp.~210--229.

\bibitem{steinbach2009boundary}
{\sc O.~Steinbach and G.~Unger}, {\em A boundary element method for the
  {D}irichlet eigenvalue problem of the {L}aplace operator}, Numer. Math., 113
  (2009), pp.~281--298.

\bibitem{stowell2010variational}
{\sc D.~Stowell and J.~Tausch}, {\em Variational formulation for guided and
  leaky modes in multilayer dielectric waveguides}, Comm. Comput. Phys., 7
  (2010), p.~564.

\bibitem{tisseur2000backward}
{\sc F.~Tisseur}, {\em Backward error and condition of polynomial eigenvalue
  problems}, Lin. Alg. Appl., 309 (2000), pp.~339--361.

\bibitem{tisseur2001structured}
{\sc F.~Tisseur and N.~J. Higham}, {\em Structured pseudospectra for polynomial
  eigenvalue problems, with applications}, SIAM J. Matrix Anal. Appl., 23
  (2001), pp.~187--208.

\bibitem{townsend2015continuous}
{\sc A.~Townsend and L.~N. Trefethen}, {\em Continuous analogues of matrix
  factorizations}, Proc. Roy. Soc. A: Math., Phys. Eng. Sci., 471 (2015),
  p.~20140585.

\bibitem{trefethen1999computation}
{\sc L.~N. Trefethen}, {\em Computation of pseudospectra}, Acta Numer., 8
  (1999), pp.~247--295.

\bibitem{trefethen2005spectra}
{\sc L.~N. Trefethen and M.~Embree}, {\em Spectra and pseudospectra: the
  behavior of nonnormal matrices and operators}, Princeton University Press,
  2005.

\bibitem{voss2002rational}
{\sc H.~Voss}, {\em A rational spectral problem in fluid-solid vibration},
  Preprints des Institutes f{\"u}r Mathematik,  (2002).

\bibitem{wagenknecht2008structured}
{\sc T.~Wagenknecht, W.~Michiels, and K.~Green}, {\em Structured pseudospectra
  for nonlinear eigenvalue problems}, J. Comput. Appl. Math., 212 (2008),
  pp.~245--259.

\bibitem{contHutch}
{\sc J.~Zvonek, A.~Horning, and A.~Townsend}, {\em conthutch++: Stochastic
  trace estimation for functions}, in preparation,  (2023).

\end{thebibliography}
\bibliographystyle{siam}

\end{document}